\documentclass[final,1p,times]{elsarticle}

\usepackage{amssymb}
\usepackage{amsthm}
\usepackage{amsmath,amssymb,amsopn,amsfonts,mathrsfs,amsbsy,amscd}
\usepackage{longtable}
\usepackage{caption}
\usepackage{multirow}
\usepackage{booktabs}
\usepackage{array}

\newcommand{\prs}{\langle\;,\;\rangle}

\newcommand{\im}{\mathrm{Im}}
\newcommand{\too}{\longrightarrow}

\newcommand{\esp}{\quad\mbox{and}\quad}

\newcommand{\so}{\mathfrak{so}}
\def\br{[\;,\;]}

\newcommand{\A}{{\mathcal A}}

\newcommand{\G}{\mathfrak{g}}
\newcommand{\g}{\mathfrak{g}}

\newcommand{\af}{\mathfrak{a}}
\newcommand{\df}{\mathfrak{d}}

\newcommand{\h}{{\mathfrak{h}}}

\newcommand{\Lu}{{\mathrm{L}}}
\newcommand{\Ru}{{\mathrm{R}}}

\newcommand{\ad}{{\mathrm{ad}}}

\newcommand{\tr}{{\mathrm{tr}}}

\newcommand{\ass}{{\mathrm{ass}}}

\newcommand{\al}{\alpha}

\newcommand{\ga}{\gamma}

\newcommand{\e}{\epsilon}

\newcommand{\la}{\lambda}

\newtheorem{theo}{Theorem}[section]
\newtheorem{pr}{Proposition}[section]
\newtheorem{Le}{Lemma}[section]

\newtheorem{remark}{Remark}

\font\bb=msbm10

\def\R{\hbox{\bb R}}

\begin{document}

\begin{frontmatter}
	
	
	
	
	\title{ Flat Lorentzian Lie groups: A complete description }
	
	\author[label1]{ Mohamed Boucetta}
	\address[label1]{ Cadi-Ayyad University\\
		BP 549 Marrakech Morocco\\e-mail: m.boucetta@uca.ac.ma  
		
	}
	
	
	
	
	
	\begin{abstract} In this paper, we  establish a complete structural description of flat Lorentzian Lie groups, i.e., Lie groups endowed with a flat left-invariant Lorentzian metric, thereby resolving a long-standing open problem in the theory of pseudo-Riemannian Lie groups. Our main result shows that any flat Lorentzian Lie group  is of Kundt type, and that  the underlying Lie algebra falls into one of six explicit classes. A key ingredient of the proof is a refined analysis of the double extension process, which reveals that all flat Lorentzian Lie algebras arise — directly or in a generalized sense — from flat Euclidean ones. As a consequence, we obtain an explicit list of the Lie algebra types admitting such metrics in dimensions three and four, together with normal forms of the corresponding metrics, recovering and unifying several previously known partial results.

	\end{abstract}
	

\end{frontmatter}

\section{Introduction}\label{section1}

A Lie group \( G \) endowed with a left-invariant pseudo-Riemannian metric is called a \emph{pseudo-Riemannian Lie group}. When the metric is positive definite (respectively, of signature \((-,+,\ldots,+)\)), we shall refer to it as \emph{Riemannian} (respectively, \emph{Lorentzian}). The Lie algebra \( \mathfrak{g} \) of such a group, equipped with the induced inner product, is called a \emph{pseudo-Euclidean Lie algebra}; accordingly, we use the terms \emph{Euclidean} and \emph{Lorentzian} in the corresponding cases.

It is well known that any left-invariant Riemannian metric on a Lie group is geodesically complete. In contrast, in the Lorentzian setting when the metric is flat, geodesic completeness holds if and only if the group is unimodular \cite{segal}.

Let \((G,h)\) be a pseudo-Riemannian Lie group, and denote by \(\nabla\) its Levi-Civita connection. Identifying \( \mathfrak{g} \) with the space of left-invariant vector fields on \(G\), the connection induces a bilinear product \( \cdot \) on \( \mathfrak{g} \) defined by
\[
X \cdot Y = \nabla_X Y, \qquad X,Y \in \mathfrak{g}.
\]This is characterized by the two relations:
\[ [X,Y]=X\cdot Y-Y\cdot X\esp \langle X\cdot Y,Z\rangle+\langle X\cdot Z,Y\rangle=0. \]
It is a classical result that \((G,h)\) is flat if and only if \((\mathfrak{g},\cdot)\) is a left-symmetric algebra, that is,
\[
\mathrm{ass}(X,Y,Z) = \mathrm{ass}(Y,X,Z), \qquad \forall X,Y,Z \in \mathfrak{g},
\]
where
\[
\mathrm{ass}(X,Y,Z) = (X \cdot Y)\cdot Z - X \cdot (Y \cdot Z)
\]
denotes the associator. We call $\G$ a flat pseudo-Euclidean Lie algebra.

The classification of flat pseudo-Riemannian Lie groups is a fundamental and difficult problem in the theory of homogeneous spaces. It sits at the crossroads of Lie theory, differential geometry, and the study of left-symmetric algebras, and has attracted sustained attention over several decades. In the Riemannian case, a complete characterization was obtained by Milnor \cite{Milnor}, who proved that a Lie group admits a flat left-invariant Riemannian metric if and only if its Lie algebra decomposes orthogonally as $\mathfrak{g} = \mathfrak{b} \oplus \mathfrak{u}$, where $\mathfrak{b}$ is an abelian subalgebra, $\mathfrak{u}$ is an abelian ideal, and $\mathrm{ad}_b$ is skew-symmetric for every $b \in \mathfrak{b}$. A refined and more precise version of this result was later obtained in \cite[ Theorem 3.1]{ABL}, which will play a central role in the present paper.

In contrast, the Lorentzian case is far from being completely understood, despite numerous partial results. It is known that any flat Lorentzian Lie group is necessarily solvable \cite{della}. In \cite{Aub-Med}, Aubert and Medina introduced the so-called \emph{double extension process}, which provides a systematic way to construct flat Lorentzian Lie algebras from flat Euclidean ones. This construction has led to several important subclasses, including flat nilpotent Lorentzian Lie algebras \cite{Aub-Med}, those with degenerate center \cite{ABL} and non-unimodular (hence incomplete) examples \cite{BH}.

On the other hand, it can be shown that the Levi-Civita product of a flat Euclidean Lie algebra is of Novikov type, a condition that places strong algebraic constraints on the algebra (see the definition below). The class of flat Lorentzian Lie algebras whose Levi-Civita product is of Novikov type has been described in \cite{oubba,L3,L2}. More recently, a classification of flat nilpotent Lorentzian Lie groups has been obtained in \cite{bajo} and an algebraic group-theoretical version of this classification was already present in \cite[Section 3]{Margulis}.
In addition, flat pseudo-Riemannian Lie groups of signature \((2,n-2)\) have been investigated in \cite{L1,L4}.

Let \((\mathcal{A}, \bullet)\) be a left-symmetric algebra. We say that \(\mathcal{A}\) is a \emph{Novikov algebra} if it satisfies
\begin{equation}\label{Novikov1}
	(x \bullet y)\bullet z = (x \bullet z)\bullet y, \qquad \forall x,y,z \in \mathcal{A}.
\end{equation}
Novikov algebras arise naturally in the study of Poisson brackets of hydrodynamical type \cite{N}.

A flat pseudo-Riemannian Lie group is said to be of \emph{Novikov type} if its Lie algebra, endowed with the Levi-Civita product, is a Novikov algebra. Such Lie algebras admit a useful characterization given in \cite{L3,L2,oubba} (see Proposition~\ref{Novikovpr}).

To summarize the state of the art, two main classes of flat Lorentzian Lie algebras have emerged from the existing literature: those obtained via the double extension process from flat Euclidean Lie algebras, and those whose Levi-Civita product is of Novikov type. This naturally raises the following fundamental question, which has remained open until now:

\medskip
\noindent\textit{Do there exist flat Lorentzian Lie algebras that belong to neither of these two classes? And more generally, is it possible to give a complete structural description of all flat Lorentzian Lie algebras?}
\medskip

The main result of this paper provides a definitive and complete answer to this question. We prove that every flat Lorentzian Lie algebra belongs to one of six explicit families, which together cover both the double extension class and the Novikov class, and show that no other flat Lorentzian Lie algebras exist. The key geometric insight behind our approach is the following dichotomy, established in Proposition \ref{ee}: any flat Lorentzian Lie algebra either admits a timelike parallel left-invariant vector field, or is of Kundt type. This dichotomy, which is of independent geometric interest, organizes the entire classification and drives the structure of the paper.

We recall that Kundt spacetimes play a central role in general relativity and Lorentzian geometry. A Lorentzian manifold $(M, g)$ is said to be of Kundt type \cite{Z} if there exist a nowhere-vanishing null vector field $V$ and a $1$-form $\alpha$ such that
$$g(V, V) = 0, \quad \nabla_X V = \alpha(X) V, \quad \nabla_V V = 0,$$
for every vector field $X$ orthogonal to $V$. In the homogeneous setting studied here, the Kundt condition translates into a precise algebraic condition on the Levi-Civita product of the Lie algebra, which we exploit systematically throughout the paper.

Our main result is the following.

\begin{theo}\label{main}
	Let \((G,h)\) be a flat Lorentzian Lie group of dimension at least two. Then
	\((G,h)\)  is of Kundt type
	and  the Lie algebra \((\mathfrak{g},\langle\cdot,\cdot\rangle)\), where \(\langle\cdot,\cdot\rangle = h_{1_G}\), is isomorphic to one of six explicitly described families listed in Tables~\ref{1}-\ref{6}, which exhaust all possibilities.

\end{theo}

A notable and practically useful feature of the six models is that the conditions imposed on their parameters — such as commutativity relations between families of skew-symmetric endomorphisms, invariance conditions on representations, or compatibility relations between linear maps — are all of a purely linear-algebraic nature. They yield explicit algebraic conditions that can be analyzed systematically, making it easy to construct concrete examples in arbitrary dimension and to systematically explore the geometry of flat Lorentzian Lie groups far beyond the low-dimensional setting. As an illustration, we derive from our six general families a complete classification of flat Lorentzian Lie algebras in dimensions three and four, presented in Tables~\ref{tab:flat-lorentzian-3d} and~\ref{tab:flat-lorentzian-dim4}.

We emphasize that the six families
provide an exhaustive structural
description, but are not necessarily
pairwise disjoint. We return to this
point in the concluding section.

The proof of Theorem \ref{main} rests on three main ingredients. The first is the dichotomy of Proposition \ref{ee}, which reduces the problem to two structurally different cases. The second is the refinement of Milnor's theorem established in \cite[Theorem 3.1]{ABL} (see Theorem \ref{milnor}), which serves as the algebraic backbone of the entire argument: whenever a flat Euclidean Lie algebra appears in the course of the analysis, its orthogonal decomposition $\mathfrak{h} = \mathfrak{d} \oplus \mathfrak{a}$ and the associated representation $\rho : \mathfrak{a} \to \mathfrak{so}(\mathfrak{d})$ provide the primary tool for determining the remaining data step by step. The third ingredient is a set of three key lemmas — Lemmas \ref{LB}, \ref{LF}, and \ref{AM} — which allow us to solve the system of equations \eqref{curvature} arising from the flatness condition in full generality, despite its apparent complexity.

To help the reader navigate the proof, we outline its structure. In Proposition \ref{ee}, we establish the fundamental dichotomy. In the timelike parallel case, the analysis is carried out in Proposition \ref{G1}, leading directly to the model $g_1$ of Table \ref{1}. In the Kundt case, Proposition \ref{double} shows that the flatness condition is equivalent to a system \eqref{curvature} of linear-algebraic equations on the data $(E, F, A, u, v, w, \alpha, \lambda)$ parametrizing the algebra, and that the Lie bracket is then given explicitly by \eqref{bracket}. The modular vector of the algebra is computed in formula \eqref{hm}, and its isotropy — required by Proposition \ref{pr1} in the non-unimodular case — drives a case split into five mutually exclusive and exhaustive subcases: three unimodular cases (U1), (U2), (U3), and two non-unimodular cases (NU1), (NU2). In Section~\ref{section4}, we solve system~\eqref{curvature} for each of these five cases and, where necessary, simplify the resulting Lie brackets through suitable changes of basis. This leads to the five remaining models $g_2$–$g_6$, presented in Tables~\ref{2}–\ref{6}. The detailed computations supporting the proofs are collected in the appendix (Section \ref{section9}), which is organized so that each subsection corresponds to a specific proposition in the main body and can be read independently.

The paper is organized as follows. Section \ref{section2} collects the preliminary results needed throughout: properties of skew-symmetric endomorphisms, the completeness criterion for flat pseudo-Riemannian Lie groups, the characterization of Novikov algebras, the refinement of Milnor's theorem, and the three key lemmas. It also contains the proof of Proposition \ref{ee}. Section \ref{section3} derives the structural consequences of the dichotomy and reduces the classification to solving system \eqref{curvature}. Section \ref{section4} constitutes the core of the paper and solves this system in each of the five cases, completing the proof of Theorem \ref{main}.

Section \ref{sectionKundt} is devoted to establishing the Kundt character of flat Lorentzian Lie groups.    Sections \ref{sectionTables} and Section \ref{Section7}  present the six models in table form and derive the complete classification in dimensions three and four. In Section \ref{section6}, we give a conclusion of our study. Finally, Section \ref{section9} is an appendix containing the details of the lengthy computations.

In order to test the validity of our main result and to illustrate the practical usefulness of the six models of Tables \ref{1}–\ref{6}, we have systematically derived, for each family, at least one explicit example in dimension six and one in dimension seven. This was done by choosing concrete values of the parameters — representations $\rho$, skew-symmetric endomorphisms $E$, $F$, vectors $h$, $u$, $n$, $w$ — satisfying the linear-algebraic conditions listed in the corresponding table, and then computing the resulting Lie brackets explicitly. In each case, the flatness of the algebra — equivalently, the left-symmetry of the Levi-Civita product — was independently verified using SageMath. All examples passed this verification, providing strong computational evidence for the correctness of the classification. These computations also confirm the claim made above that the explicit algebraic conditions in our models can be analyzed effectively and yield explicit flat Lorentzian Lie algebras in any prescribed dimension.

\section{Preliminaries}\label{section2}

In this section, we present several useful properties of the Lie algebra of skew-symmetric endomorphisms. We also recall some basic facts concerning non-unimodular flat pseudo-Riemannian Lie groups and Novikov algebras. Finally, we state a refined version of Milnor’s theorem, which will play a central role in the sequel, and we prove the first assertion of Theorem~\ref{main}.

\subsection{The Lie algebra of skew-symmetric endomorphisms}\label{subsection21}

We begin by recalling some standard facts on pseudo-Euclidean vector spaces. A \emph{pseudo-Euclidean vector space} is a finite-dimensional real vector space \(V\) endowed with a nondegenerate symmetric bilinear form \(\langle \cdot,\cdot \rangle\) of signature \((q,n-q)\). The cases \(q=0\) and \(q=1\) correspond to \emph{Euclidean} and \emph{Lorentzian} structures, respectively.

Let \((V,\langle \cdot,\cdot \rangle)\) be a pseudo-Euclidean vector space. For any endomorphism $F:V\too V$, we denote by $F^*$ its adjoint and by \(\mathfrak{so}(V)\) the Lie algebra of skew-symmetric endomorphisms of \(V\). 

Assume now that \(\langle \cdot,\cdot \rangle\) is positive definite. Then \(\mathfrak{so}(V)\) is naturally endowed with the symmetric bilinear form
\[
\langle A,B \rangle_{\mathfrak{so}} = -\mathrm{tr}(AB), \qquad A,B \in \mathfrak{so}(V),
\]
which is positive definite. Moreover, this form is \emph{quadratic} in the sense that, for all \(A,B,C \in \mathfrak{so}(V)\),
\begin{equation}\label{quadratic}
	\langle [A,B],C \rangle_{\mathfrak{so}} + \langle [A,C],B \rangle_{\mathfrak{so}} = 0.
\end{equation}
As a consequence, the nonzero eigenvalues of the operator
\[
\mathrm{ad}_A : \mathfrak{so}(V) \to \mathfrak{so}(V), \quad B \mapsto [A,B],
\]
are purely imaginary.

This property extends to the adjoint action of \(A\) on \(\mathfrak{gl}(V)\). Indeed, if \(\lambda_1,\ldots,\lambda_n\) are the eigenvalues of \(A\), then the eigenvalues of \(\mathrm{ad}_A\) are given by \(\lambda_i - \lambda_j\), \(1 \leq i,j \leq n\). Since \(A\) is skew-symmetric, its eigenvalues are purely imaginary, and the conclusion follows.

We will use frequently the following proposition later.  

\begin{pr}\label{solvable}
	Let \((V,\langle \cdot,\cdot \rangle)\) be a Euclidean vector space and \(\mathfrak{g} \subset \mathfrak{so}(V)\) a solvable Lie subalgebra. Then:
	\begin{enumerate}\item  \(\mathfrak{g}\) is abelian, and there exists an orthonormal basis
		\[
		(e_1,f_1,\ldots,e_r,f_r,g_1,\ldots,g_s)
		\]
		of \(V\) and linear forms \(\lambda_1,\ldots,\lambda_r \in \mathfrak{g}^*\) such that, for all \(A \in \mathfrak{g}\),
		\[
		A(e_i) = \lambda_i(A) f_i, \quad
		A(f_i) = -\lambda_i(A) e_i, \quad
		A(g_j) = 0,
		\]
		for \(i=1,\ldots,r\) and \(j=1,\ldots,s\).
		\item If  \(A \in \mathfrak{so}(V)\) satisfies \(\mathrm{ad}_A(\G) \subset \G\)
		then \(\mathrm{ad}_A(\G) = \{0\}\).
	\end{enumerate}
	
\end{pr}
\begin{proof}
	\begin{enumerate}
		\item It is a classical consequence of Lie's Theorem.
		\item for all \(B,C \in \G\), one has
		\[
		\langle [A,B], C \rangle_{\mathfrak{so}} = \langle A, [B,C] \rangle_{\mathfrak{so}} = 0,
		\]
		which implies \([A,B] = 0\) for all \(B \in \G\).\qedhere
	\end{enumerate}
\end{proof}

\subsection{Completeness of flat pseudo-Riemannian Lie groups and the modular vector}	\label{subsection22}

Let \((G,h)\) be a pseudo-Riemannian Lie group and let \(\mathfrak{g}\) denote its Lie algebra, identified with the space of left-invariant vector fields. The metric \(h\) induces a nondegenerate symmetric bilinear form \(\langle \cdot,\cdot \rangle\) on \(\mathfrak{g}\). The Levi-Civita connection defines a bilinear product on \(\mathfrak{g}\), called the \emph{Levi-Civita product}, given by
\begin{equation}\label{Levi-Civita}
	2\langle u \cdot v, w \rangle
	= \langle [u,v], w \rangle + \langle [w,u], v \rangle + \langle [w,v], u \rangle.
\end{equation}

For \(u \in \mathfrak{g}\), we denote by \(\mathrm{L}_u, \mathrm{R}_u \in \mathrm{End}(\mathfrak{g})\) the left and right multiplication operators defined by
\[
\mathrm{L}_u(v) = u \cdot v, \qquad \mathrm{R}_u(v) = v \cdot u.
\]
Then \(\mathrm{L}_u\) is skew-symmetric with respect to \(\langle \cdot,\cdot \rangle\), and the adjoint representation is given by
\[
\mathrm{ad}_u = \mathrm{L}_u - \mathrm{R}_u, \qquad \mathrm{ad}_u(v) = [u,v].
\]
The curvature tensor at the identity is
\[
\mathrm{K}(u,v) = \mathrm{L}_{[u,v]} - [\mathrm{L}_u,\mathrm{L}_v].
\]
The metric is flat if and only if \(\mathrm{K}=0\). In this case, \((G,h)\)  is called a \emph{flat pseudo-Riemannian Lie group}. 

Since the study of flat pseudo-Riemannian Lie groups essentially reduces to that of their Lie algebras, we introduce the notion of a \emph{flat pseudo-Euclidean Lie algebra}. By this we mean a Lie algebra \((\mathfrak{g},[\cdot,\cdot])\) endowed with a pseudo-Euclidean metric \(\langle \cdot,\cdot \rangle\) such that, for all \(u,v \in \mathfrak{g}\),
\[
\mathrm{L}_{[u,v]} - [\mathrm{L}_u,\mathrm{L}_v] = 0,
\]
where \(\mathrm{L}_u\) denotes the left multiplication operator associated with the Levi-Civita product defined by \eqref{Levi-Civita}.

In contrast to the Riemannian case, where left-invariant metrics are always geodesically complete, completeness in the pseudo-Riemannian setting is more subtle, even in the flat case. A flat pseudo-Riemannian Lie group is geodesically complete if and only if it is unimodular \cite{segal}. Equivalently, its Lie algebra \(\mathfrak{g}\) is unimodular, that is,
\[
\mathrm{tr}(\mathrm{ad}_u) = 0 \quad \text{for all } u \in \mathfrak{g}.
\]
This condition is equivalent to the vanishing of the \emph{modular vector} \(\mathbf{h} \in \mathfrak{g}\), defined by
\begin{equation}\label{h}
	\langle u, \mathbf{h} \rangle = \mathrm{tr}(\mathrm{ad}_u) = -\mathrm{tr}(\mathrm{R}_u), \qquad u \in \mathfrak{g}.
\end{equation}
The modular vector of a flat pseudo-Euclidean Lie algebra enjoys several important properties, summarized in the following proposition proved in \cite{BH}.
\begin{pr}[\cite{BH}]\label{pr1}
	Let  $(\G,[\;,\;],\prs)$ be a
	flat pseudo-Euclidean  Lie algebra. Then:\begin{enumerate}\item The modular vector satisfies $\mathbf{h}\in[\G,\G]\cap[\G,\G]^\perp$ and 
		$\mathrm{R}_{\mathbf{h}}$ is symmetric with respect to $\prs$.
		In
		particular, if $\G$ is non-unimodular then $[\G,\G]$ is degenerate,  $\langle
		\mathbf{h},\mathbf{h}\rangle=0$ and $\mathbf{h}.\mathbf{h}=0$.\item  If $\prs$ is Lorentzian then $H=\mathrm{span}\{\mathbf{h}\}$ is a two-sided ideal $($with respect to the Levi-Civita product$)$ in $H^\perp$.\end{enumerate}
	
\end{pr}

This proposition serves as a key ingredient in the case analysis presented at the beginning of Section~\ref{section4}.

\subsection{Pseudo-Euclidean Novikov algebras as a subclass of flat pseudo-Euclidean Lie algebras}\label{subsection23}

Let \((\mathfrak{g},[\cdot,\cdot],\langle \cdot,\cdot \rangle)\) be a flat pseudo-Euclidean Lie algebra, and denote by \(\cdot\) its Levi-Civita product. The flatness condition is equivalent to
\[
\mathrm{ass}(u,v,w) = \mathrm{ass}(v,u,w), \qquad u,v,w \in \mathfrak{g},
\]
where \(\mathrm{ass}(u,v,w) = (u \cdot v)\cdot w - u \cdot (v \cdot w)\) is the associator. Thus, \((\mathfrak{g},\cdot)\) is a left-symmetric algebra.

An important subclass of left-symmetric algebras is given by Novikov algebras. A left-symmetric algebra \((\mathcal{A},\cdot)\) is called a \emph{Novikov algebra} if
\begin{equation}\label{Novikov}
	(x \cdot y)\cdot z = (x \cdot z)\cdot y, \qquad x,y,z \in \mathcal{A}.
\end{equation}
Novikov algebras arise naturally in the study of Poisson brackets of hydrodynamical type \cite{N}. Flat pseudo-Euclidean Lie algebras whose Levi-Civita product defines a Novikov structure were investigated in \cite{L3,L2}. We recall the following characterization.

\begin{pr}\label{Novikovpr}
	Let \((\mathfrak{g},[\cdot,\cdot],\langle \cdot,\cdot \rangle)\) be a pseudo-Euclidean Lie algebra and let \(\cdot\) denote its Levi-Civita product. Then \((\mathfrak{g},\cdot)\) is a Novikov algebra if and only if, for all \(u,v \in \mathfrak{g}\),
	\[
	\mathrm{L}_{u \cdot v} = [\mathrm{L}_u,\mathrm{L}_v] = 0.
	\]
\end{pr}

\subsection{A refinement of Milnor's theorem}\label{subsection24}

In  \cite{Milnor}, Milnor  proved that a Lie group admits a flat left-invariant Riemannian metric if and only if its Lie algebra decomposes orthogonally as $\mathfrak{g} = \mathfrak{b} \oplus \mathfrak{u}$, where $\mathfrak{b}$ is an abelian subalgebra, $\mathfrak{u}$ is an abelian ideal, and $\mathrm{ad}_b$ is skew-symmetric for every $b \in \mathfrak{b}$.

We present here a refinement of this theorem which will be crucial in the sequel.  This refinement is established in \cite[Theorem 3.1]{ABL}.

\begin{theo}\label{milnor}
	Let \((\G,\br,\prs)\) be a Euclidean Lie algebra and \(\cdot\) its Levi-Civita product. Then \((\G,\br,\prs)\) is flat if and only if the following conditions hold:
	\begin{enumerate}
		\item \([\mathfrak{g},\mathfrak{g}] = \mathfrak{g} \cdot \mathfrak{g}\),
		\item the derived ideal \(\mathfrak{d} := [\mathfrak{g},\mathfrak{g}]\) and its orthogonal complement \(\mathfrak{a} := \mathfrak{d}^{\perp}\) are abelian, and \(\dim \mathfrak{d} = 2r\).
	\end{enumerate}
	In this case, $(\G,\br)$ is unimodular and there exists a representation of the abelian Lie algebra \(\mathfrak{a}\),
	\[
	\rho : \mathfrak{a} \to \mathfrak{so}(\mathfrak{d})
	\]
	such that $\ker\rho$ is the center of $\G$, 
	\[
	\bigcap_{a \in \mathfrak{a}} \ker \rho(a) = \{0\},
	\]
	and the Levi-Civita product is given by
	\begin{equation}\label{Levi-Milnor}
		u \cdot v =
		\begin{cases}
			0 & \text{if } u \in \mathfrak{d} \text{ or } v \in \mathfrak{a},\\
			\rho(u)(v) & \text{if } u \in \mathfrak{a},\ v \in \mathfrak{d}.
		\end{cases}
	\end{equation}
	In particular, for all \(u,v \in \mathfrak{g}\),
	\[
	\mathrm{L}_{[u,v]} = \mathrm{L}_{u \cdot v} = [\mathrm{L}_u,\mathrm{L}_v] = 0,
	\]
	and hence \((\mathfrak{g},\cdot)\) is a Novikov algebra.
\end{theo}

\begin{remark}
	The condition \(\bigcap_{a \in \mathfrak{a}} \ker \rho(a) = \{0\}\) is equivalent to \(\sum_{a \in \mathfrak{a}} \mathrm{Im}\,\rho(a) = \mathfrak{d}\) which is equivalent to \([\mathfrak{g},\mathfrak{g}] = \mathfrak{g} \cdot \mathfrak{g}\).
\end{remark}

Theorem \ref{milnor} will play a pervasive role throughout the paper. Whenever a flat Euclidean Lie algebra $(\h, \star, \langle\cdot,\cdot\rangle_\h)$ appears — either in Proposition \ref{G1} as the full Euclidean part $e^\perp$ in the timelike case, or as the algebra $(\h, \star)$ arising in the Kundt case (Propositions \ref{u}-\ref{NU2}) — its orthogonal decomposition $\h = \df \oplus \af$ and the associated representation $\rho : \af \to \so(\df)$ will be the primary tool for determining the remaining data $(E, F, A, u, v, w, \alpha, \lambda)$ step by step. In this sense, Theorem \ref{milnor} is the algebraic backbone of the entire classification.

As a consequence of Proposition \ref{solvable}, we have the following proposition.

\begin{pr}\label{CA} Let  \((\mathfrak{a},\mathfrak{d},\rho)\) as in Theorem \ref{milnor}, then there exist a basis \((e_1,f_1,\ldots,e_r,f_r)\) of \(\mathfrak{d}\) and nonzero vectors \(u_1,\ldots,u_r \in \mathfrak{a}\) such that, for all \(a \in \mathfrak{a}\),
	\[
	\rho(a)(e_i) = \langle u_i,a \rangle f_i, \quad
	\rho(a)(f_i) = -\langle u_i,a \rangle e_i,
	\quad i=1,\ldots,r.
	\]
	Moreover, the center of the Lie algebra $\G=\df\oplus\af$ is $Z(\G)=\{u_1,\ldots,u_r\}^\perp$.
	
\end{pr}

\subsection{Three key lemmas}\label{subsection25}

In this subsection, we prove three key lemmas that will play a crucial role in the sequel.

\begin{Le}\label{LB}
	Let $E$ be a vector space, $(V,\langle \cdot,\cdot \rangle)$  a Euclidean vector space, $\Lambda : E \to \so(V)$ and $B : E \to V$  two linear maps such that, for all $a,b \in E$,
	\[
	\bigcap_{a \in E} \ker \Lambda_a = \{0\}, \quad [\Lambda_a,\Lambda_b]=0 \quad \text{and} \quad \Lambda_a \circ B(b) = \Lambda_b \circ B(a).
	\]
	Then there exists a vector $v \in V$ such that, for all $a \in E$, $B(a) = \Lambda_a(v)$.
\end{Le}

\begin{proof}
	Note that, by virtue of Proposition \ref{solvable}, the conditions $\bigcap_{a \in E} \ker \Lambda_a = \{0\}$ and $[\Lambda_a,\Lambda_b]=0$ imply that $\dim V = 2n$. We prove the result by induction on $n$.
	
	If $\dim V = 2$, any nonzero skew-symmetric endomorphism is invertible. Hence, there exists $a \in E$ such that $\Lambda_a$ is invertible. For any $b \in E$, we obtain
	\[
	B(b) = \Lambda_b \circ \Lambda_a^{-1} \bigl(B(a)\bigr),
	\]
	which gives the desired result.
	
	Assume now that the result holds for $n-1$, and let $\dim V = 2n$. Choose $a_0 \in E$ such that $\Lambda_{a_0} \neq 0$. If $\Lambda_{a_0}$ is invertible, the conclusion follows as above. Otherwise, set $V_0 = \ker \Lambda_{a_0}$. Then $V = V_0 \oplus V_0^\perp$.
	For any $a \in E$, write
	\[
	B(a) = B_1(a) + B_2(a),
	\]
	where $B_1(a) \in V_0$ and $B_2(a) \in V_0^\perp$. Since $[\Lambda_a,\Lambda_{a_0}]=0$, the endomorphism $\Lambda_a$ preserves both $V_0$ and $V_0^\perp$. The relation $\Lambda_a \circ B(b) = \Lambda_b \circ B(a)$ is then equivalent to
	\[
	\Lambda_b\bigl(B_1(a)\bigr) = \Lambda_a\bigl(B_1(b)\bigr) \quad \text{and} \quad \Lambda_a \circ B_2(b) = \Lambda_b \circ B_2(a).
	\]
	Since the restriction of $\Lambda_{a_0}$ to $V_0^\perp$ is invertible, the vector  $v=\Lambda^{-1}_{a_0}\circ B_2(a_0) \in V_0^\perp$ verifies $B_2(b) = \Lambda_b(v)$ for all $b \in E$.

	Finally, consider $B_1:E\too V_0$ and $\overline{\Lambda}:E\too\so(V_0)$ where $\overline{\Lambda}_a$ is the restrictions of $\Lambda_b$ to $V_0$. We have, for any $a,b\in E$,
	\[ [\overline{\Lambda}_a,\overline{\Lambda}_b]=0 \esp \bigcap_{a \in E} \ker \overline{\Lambda}_a = \{0\}. \]
	By the induction hypothesis applied to $(B_1,\overline{\Lambda})$, there exists $v_0 \in V_0$ such that $B_1(b) = \overline{\Lambda}_b(v_0)$ for all $b \in E$. Therefore,
	\[
	B(b) = \Lambda_b(v_0 + v),
	\]
	which concludes the proof.
\end{proof}

The process of double extension of a flat pseudo-Euclidean Lie algebra \((\h,[\cdot,\cdot],\langle \cdot,\cdot \rangle)\), described in \cite{Aub-Med} and recalled in Section \ref{section3}, involves two endomorphisms \(D,A \in \mathrm{End}(\h)\) and an element \(w \in \h\) satisfying 
\[
D - A \in \mathfrak{so}(\h),\;  [D-A,A] = A^2 - \la A + \mathrm{R}_w\esp A([a,b])=a\cdot A(b)-b\cdot A(a)
\]for any $a,b\in\h$ and where $\cdot$ is the Levi-Civita product (see \eqref{doubleM}).

In the sequel, we will encounter theses relations. The following two lemmas give their solutions in the Euclidean case when \(w=0\).

\begin{Le}\label{LF}
	Let $(V,\langle \cdot,\cdot \rangle)$ be a Euclidean vector space, $F \in \mathfrak{so}(V)$, 
	and $A \in \mathrm{End}(V)$ satisfying
	\begin{equation}\label{eq:FA}
		[F,A] = A^2 - \lambda A, \qquad \lambda \neq 0.
	\end{equation}
	Then $A$ is diagonalizable over $\mathbb{R}$ with spectrum $\{0, \lambda\}$, and the following 
	hold:
	\begin{equation}\label{eq:A2}
		A^2 = \lambda A \qquad \text{and} \qquad [F,A] = 0.
	\end{equation}
	Moreover, with respect to the orthogonal decomposition $V = \ker A \oplus (\ker A)^\perp$, 
	the operators $F$ and $A$ take the block forms
	\begin{equation}\label{eq:blocks}
		F =
		\begin{pmatrix}
			F_1 & 0 \\
			0   & F_2
		\end{pmatrix},
		\qquad
		A =
		\begin{pmatrix}
			0 & U \\
			0 & \lambda\,\mathrm{Id}_{(\ker A)^\perp}
		\end{pmatrix},
	\end{equation}
	where $F_1 \in \mathfrak{so}(\ker A)$, $F_2 \in \mathfrak{so}((\ker A)^\perp)$, 
	$U : (\ker A)^\perp \to \ker A$ is a linear map, and $F_1 \circ U = U \circ F_2$.
\end{Le}

\begin{proof}
	The proof proceeds in two steps.
	
	\medskip
	\noindent\textbf{Step 1: Proof of \eqref{eq:A2}.}
	
	Set $B := A - \lambda\,\mathrm{Id}_V$. A direct computation gives
	\[
	[F, B] = (B + \lambda\,\mathrm{Id}_V)^2 - \lambda(B + \lambda\,\mathrm{Id}_V) 
	= B^2 + \lambda B.
	\]
	Using the Leibniz rule $[F, XY] = X[F,Y] + [F,X]Y$, we establish by induction 
	that for every $n \in \mathbb{N}^*$,
	\begin{equation}\label{eq:induction}
		[F, A^n] = nA^n(A - \lambda\,\mathrm{Id}_V), 
		\qquad 
		[F, B^n] = nB^n(B + \lambda\,\mathrm{Id}_V).
	\end{equation}
	The case $n = 1$ holds by assumption. Assuming the formula for $n-1$, we compute
	\[
	[F, A^n] = [F, A]A^{n-1} + A[F, A^{n-1}]
	= A^n(A - \lambda\,\mathrm{Id}_V) + (n-1)A^n(A - \lambda\,\mathrm{Id}_V)
	= nA^n(A - \lambda\,\mathrm{Id}_V),
	\]
	and the second formula follows similarly. As a consequence, for any polynomial 
	$P(X) = \sum_{k=0}^n a_k X^k$, one has
	\begin{equation}\label{eq:poly}
		[F, P(A)] = P_{\mathrm{ad}}(A)(A - \lambda\,\mathrm{Id}_V), 
		\qquad \text{where} \quad 
		P_{\mathrm{ad}}(X) := \sum_{k=0}^n k\,a_k X^k.
	\end{equation}
	
	Let $P(X)$ be the minimal polynomial of $A$, written as 
	$P(X) = (X - \lambda)^m Q(X)$ with $m \geq 0$ and $Q(\lambda) \neq 0$. 
	Applying \eqref{eq:induction}--\eqref{eq:poly} and using $P(A) = 0$, we obtain
	\begin{align*}
		0 &= [F, P(A)] \\
		&= [F,(A - \lambda\,\mathrm{Id}_V)^m]Q(A) 
		+ (A - \lambda\,\mathrm{Id}_V)^m [F, Q(A)] \\
		&= m\,A(A - \lambda\,\mathrm{Id}_V)^m Q(A) 
		+ (A - \lambda\,\mathrm{Id}_V)^{m+1} Q_{\mathrm{ad}}(A) \\
		&= (A - \lambda\,\mathrm{Id}_V)^{m+1} Q_{\mathrm{ad}}(A).
	\end{align*}
	Thus $(X-\lambda)^{m+1} Q_{\mathrm{ad}}(X)$ is a polynomial of degree $\deg P + 1$ 
	that annihilates $A$. Since $P$ is the minimal polynomial of $A$, there exist 
	constants $\alpha, \mu \in \mathbb{R}$ such that
	\[
	(X - \lambda)^{m+1} Q_{\mathrm{ad}}(X) = \alpha(X - \mu)(X - \lambda)^m Q(X),
	\]
	which simplifies to $(X - \lambda)Q_{\mathrm{ad}}(X) = \alpha(X - \mu)Q(X)$.
	Since $Q(\lambda) \neq 0$, the polynomial $Q(X)$ is coprime to $(X - \lambda)$, 
	so $\mu = \lambda$ and $Q_{\mathrm{ad}}(X) = \alpha Q(X)$.
	
	Writing $Q(X) = \sum_{k=0}^r a_k X^k$, the identity $Q_{\mathrm{ad}}(X) = \alpha Q(X)$ 
	gives $k\,a_k = \alpha\,a_k$ for all $k = 0, \ldots, r$. Hence $\alpha = r$ and 
	$a_k = 0$ for $0 \leq k \leq r-1$, so that $Q(X) = a_r X^r$ and 
	$P(X) = a_r(X - \lambda)^m X^r$. Consequently,
	\[
	V = \ker B^m \oplus \ker A^r.
	\]
	Since $B$ is nilpotent on $\ker B^m$, equation \eqref{eq:induction} shows that 
	$F$ preserves both $\ker B^m$ and $\ker A^r$. If $B^{h-1} \neq 0$ on $\ker B^m$ 
	for some $h \geq 2$, then $[F, B^{h-1}] = -\lambda h\,B^{h-1}$, which would give 
	a real nonzero eigenvalue $-\lambda h$ for $\mathrm{ad}_F$ acting on 
	$\mathrm{End}(V)$. This contradicts the fact that $F \in \mathfrak{so}(V)$ 
	implies $\mathrm{ad}_F$ has only purely imaginary eigenvalues 
	(see Subsection~\ref{subsection21}). Therefore $B^{h-1} = 0$, so $B$ is zero 
	on $\ker B^m$, i.e.\ $A = \lambda\,\mathrm{Id}$ on $\ker B^m$, while $A = 0$ 
	on $\ker A^r$. This gives $A^2 = \lambda A$ on all of $V$, and hence $[F,A] = 0$ 
	by \eqref{eq:FA}.
	
	\medskip
	\noindent\textbf{Step 2: Block structure of $F$ and $A$.}
	
	Since $A^2 = \lambda A$, the space $V$ decomposes as 
	$V = \ker A \oplus \ker(A - \lambda\,\mathrm{Id}_V)$, but this splitting need not 
	be orthogonal. We work instead with the orthogonal decomposition 
	$V = \ker A \oplus (\ker A)^\perp$.
	
	Since $[F, A] = 0$, the operator $F$ commutes with $A$ and therefore preserves 
	$\ker A$. Since $F$ is skew-symmetric, it also preserves $(\ker A)^\perp$. 
	Setting $F_1 := F|_{\ker A}$ and $F_2 := F|_{(\ker A)^\perp}$, we obtain 
	$F_1 \in \mathfrak{so}(\ker A)$ and $F_2 \in \mathfrak{so}((\ker A)^\perp)$, 
	giving the first block form in \eqref{eq:blocks}.
	
	For the block form of $A$, note that $A$ vanishes on $\ker A$ by definition. 
	For any $x \in (\ker A)^\perp$, write $A(x) = U(x) + W(x)$ with 
	$U(x) \in \ker A$ and $W(x) \in (\ker A)^\perp$. Then
	\[
	A^2(x) = A(W(x)) = U(W(x)) + W^2(x) = \lambda U(x) + \lambda W(x),
	\]
	where the last equality uses $A^2 = \lambda A$. Comparing the components, 
	we get $W^2 = \lambda\,W$ on $(\ker A)^\perp$. Since $W$ maps $(\ker A)^\perp$ 
	to itself and has no kernel there (as $\ker A \cap (\ker A)^\perp = \{0\}$), 
	$W$ is bijective, and $W^2 = \lambda W$ forces $W = \lambda\,\mathrm{Id}_{(\ker A)^\perp}$. 
	This gives the second block form in \eqref{eq:blocks}.
	
	Finally, the relation $[F, A] = 0$ in terms of the block decomposition is 
	equivalent to $F_1 \circ U = U \circ F_2$, which completes the proof.
\end{proof}

\begin{Le}\label{AM} Let $(\G,\br,\prs)$ be a flat Euclidean Lie algebra, $\cdot$ its Levi-Civita product and $A$ an endomorphism of $\G$ such that, for all $a,b\in\G$,
	\begin{equation}\label{A} A([a,b])=a\cdot Ab-b\cdot Aa. \end{equation}
	Then  in the splitting $\G=\df\oplus\af$ described in Theorem \ref{milnor},
	\[ A=\begin{pmatrix}
		A_1&\Ru_h\\0&A_2
	\end{pmatrix} \]where $A_1:\df\too\df$,   $A_2:\af\too\af$, $h\in\df$,  $\Ru_h:\af\too\df$ is the right multiplication operator given by $\Ru_h(a)=\rho(a)(h)$ and $[A_1,\rho(a)]=0$ for any $a\in\af$.
	
\end{Le}
\begin{proof} By virtue of Theorem \ref{milnor},  $\G=\df\oplus\af$ and there exists a representation of an abelian Lie algebra $\rho:\af\to\so(\df)$ such that the Levi-Civita product is given by \eqref{Levi-Milnor}.
	The relation \eqref{A} holds for $a,b\in\df$ and implies that $A(\df)\subset \df$. So, in the orthogonal splitting $\G=\df\oplus\af$, we can write $$A=\begin{pmatrix}
		A_1&B\\0&A_2
	\end{pmatrix}.$$ For $a,b\in\af$, the relation \eqref{A} is equivalent to
	$\rho(a) B(b)=\rho(b) B(a)$ and we can apply Lemma \ref{LB}. For $a\in\af$ and $b\in\df$, the relation \eqref{A} is equivalent to
	$[\rho(a), A_1]=0$. 
\end{proof}

\subsection{Proof of the dichotomy in Theorem \ref{main}}	\label{subsection26}

The following proposition is a key ingredient and give a proof of the first assertion of Theorem~\ref{main}.

\begin{pr}\label{ee}
	Let \((\G,\br,\prs)\) be a flat Lorentzian Lie algebra and $\cdot$ its Levi-Civita product. Then there exists a nonzero vector \(e \in \mathfrak{g}\) such that one of the following holds:
	\begin{enumerate}
		\item \(\langle e,e \rangle = -1\) and \(e\) is parallel, i.e., \(u\cdot e = 0\) for all \(u \in \mathfrak{g}\);
		\item \(\langle e,e \rangle = 0\) and there exists \(\lambda \in \mathfrak{g}^*\) such that, for all \(u \in \mathfrak{g}\),
		\[
		u\cdot e = \lambda(u)\, e.
		\]
	\end{enumerate}
\end{pr}

\begin{proof}
	By \cite{della}, the Lie algebra $\G$ is
	solvable. Since the Levi-Civita product is
	flat, the left multiplication map
	\[
	\Lu:\G\longrightarrow\mathfrak{so}(\G),
	\qquad \Lu(u)=\Lu_u,
	\]
	is a representation of $\G$.
	
	Choose a nondegenerate Lorentzian
	$\Lu(\G)$-invariant subspace $W$ of
	minimal dimension. Such a subspace
	exists, since $\G$ itself satisfies
	these conditions.
	
	As $\G$ is solvable, Lie's theorem
	applied to the complexification of
	the restricted representation yields
	a nonzero real invariant subspace
	$I\subseteq W$ with $\dim I\leq2$.
	
	If $I$ were positive definite, its
	orthogonal complement
	$I^\perp\cap W$ would be a
	nondegenerate Lorentzian invariant
	subspace of strictly smaller
	dimension than $W$, contradicting
	the minimality of $W$.
	
	Thus $I$ is either Lorentzian
	or degenerate. We distinguish
	four possibilities.
	
	\smallskip
	\noindent
	\textit{Case 1: $\dim I=1$ and
		$I$ is nondegenerate.}
	Then $I=\mathbb Re$ for some
	timelike vector $e$, which we
	normalize by
	$
	\langle e,e\rangle=-1.
	$
	Since $I$ is invariant, for every
	$u\in\G$ we have $\Lu_u e=\lambda(u)e$.
	Skew-symmetry gives
	\[
	0=\langle \Lu_ue,e\rangle
	=-\lambda(u),
	\]
	and hence $\Lu_ue=0$.
	
	\smallskip
	\noindent
	\textit{Case 2: $\dim I=1$ and
		$I$ is degenerate.}
	Then $I=\mathbb Re$, where $e$
	is a nonzero null vector.
	Since $I$ is invariant, there
	exists $\lambda\in\G^*$ such that, for any $u\in\G$,
	$
	\Lu_ue=\lambda(u)e$.

	\smallskip
	\noindent
	\textit{Case 3: $\dim I=2$ and
		$I$ is nondegenerate.}
	In this case, $I$ is Lorentzian.
	Choose a null basis $(e,\bar e)$
	of $I$ such that
	\[
	\langle e,e\rangle
	=\langle\bar e,\bar e\rangle=0,
	\qquad
	\langle e,\bar e\rangle=1.
	\]
	Since $I$ is invariant, we may
	write
	\[
	\Lu_ue=\lambda(u)e+\mu(u)\bar e.
	\]
	By skew-symmetry,
	\[
	0=\langle \Lu_ue,e\rangle=\mu(u).
	\]
	Consequently,
	$
	\Lu_ue=\lambda(u)e
	$ for any $ u\in\G.$
	
	\smallskip
	\noindent
	\textit{Case 4: $\dim I=2$ and
		$I$ is degenerate.}
	The radical
	$
	\operatorname{rad}(I)=I\cap I^\perp
	$
	is a one-dimensional null subspace.
	Since $I$ is invariant and the
	operators $\Lu_u$ are skew-symmetric,
	both $I$ and $I^\perp$ are invariant.
	Therefore $\operatorname{rad}(I)$
	is invariant as well.
	Choosing a nonzero vector
	$e\in\operatorname{rad}(I)$,
	we obtain
	$
	\Lu_u e=\lambda(u)e
	$ for any $ u\in\G$ and
	for some $\lambda\in\G^*$.
	
	Thus, in every case, one of the
	two assertions of the proposition
	holds.
\end{proof}

\section{ A first step toward a complete description of  flat Lorentzian Lie  algebras } \label{section3}

In this section, we derive the structural consequences of the dichotomy established in Proposition \ref{ee}.

\begin{pr}\label{G1} Let $(\G,\br,\prs)$ be a flat Lorentzian Lie algebra and $\cdot$ its Levi-Civita product. Suppose that there exists  a vector $e\in\G$ such that $\langle e,e\rangle=-1$ and, for any $u\in\G$, $u\cdot e=0$. Then:\begin{enumerate}
		
		\item[$(i)$]	 $\G$ has an orthogonal decomposition $$\G=\R e\oplus\df \oplus \af\oplus Z_0\oplus Z_1$$ where $\df \oplus \af\oplus Z_0\oplus Z_1$ is a Euclidean vector space and $\df,\af,Z_0,Z_1$ are pairwise orthogonal, 
		\item[$(ii)$]  there exists an orthonormal basis $(e_i,f_i)_{i=1}^r$ of $\df$, an orthonormal basis $(z_i,\bar{z}_i)_{i=1}^s$ of $Z_0$ and a family of non zero vectors $(u_1,\ldots,u_r)$ which spans $\af$ such that the non vanishing Lie brackets are
		\[\begin{cases} [e,e_i]=\la_i f_i,\;[e,f_i]=-\la_i e_i,\; [e,z_j]=\mu_j \bar{z}_j,\; [e,\bar{z}_j]=-\mu_j z_j,\\
			[a,e_i]=\langle a,u_i\rangle f_i,\; [a,f_i]=-\langle a,u_i\rangle e_i,\quad i=1,\ldots,r,\; j=1,\ldots,s,\; a\in\af.
		\end{cases} \]
	\end{enumerate}
	Moreover, $(\G,.)$ is a Novikov algebra.
\end{pr}

\begin{proof}
	Let $\h = e^\perp$. Then $\h$ is a flat Euclidean Lie algebra which is invariant under $\Lu_e$. By Theorem \ref{milnor}, $\h$ admits an orthogonal decomposition
	\[
	\h = \df \oplus \af \oplus Z,
	\]
	where $Z$ is the center of $\h$, and there exists a representation of the abelian Lie algebra $\af$,
	\[
	\rho : \af \to \so(\df),
	\]
	such that $\displaystyle \bigcap_{a \in \af} \ker \rho(a) = \{0\}$.
	
	For any $u,v \in \h$, the flatness and Theorem \ref{milnor} imply 
	\[
	[\Lu_e,\Lu_u] = \Lu_{e \cdot u}, \qquad [\Lu_u,\Lu_v] = 0.
	\]
	Hence, the Lie subalgebra $\{\Lu_u \mid u \in \h\}$ is abelian and invariant under $\Lu_e$. By virtue of Proposition \ref{solvable},  
	\[
	[\Lu_e,\Lu_u] = 0, \quad \text{and} \quad \Lu_{e \cdot u} = 0, \quad \forall u \in \h.
	\]
	Thus $e \cdot u \in Z \oplus \df$ and therefore
	\[
	\mathrm{Im}\bigl((\Lu_e)_{|\h}\bigr) \subset Z \oplus \df.
	\]
	Since $\Lu_e$ is skew-symmetric, we obtain $(\Lu_e)_{|\h}(\af) = 0$.
	
	On the other hand, for any $a \in \af$, the relation $[\Lu_e,\Lu_a]=0$ implies that, for all $b \in \df$,
	\[
	\Lu_e(a \cdot b) = a \cdot (e \cdot b).
	\]
	Since $\df = \h \cdot \h$, it follows that $\Lu_e$ leaves $\df$ invariant. Consequently, it also leaves $Z$ invariant.
	
	Set
	\[
	Z_1 = \ker\bigl((\Lu_e)_{|Z}\bigr), \quad Z_0 = \mathrm{Im}\bigl((\Lu_e)_{|Z}\bigr), \quad A = (\Lu_e)_{|\df}.
	\]
	Then the Lie subalgebra $\rho(\af) + \mathbb{R}A$ of $\so(\df)$ is abelian, and Proposition \ref{solvable} applies.
	
	Finally, Proposition \ref{Novikovpr} ensures that $(\G,\cdot)$ is a Novikov algebra. This completes the proof.
\end{proof}

From Proposition \ref{G1}, we get our first model of flat Lorentzian Lie algebras listed in Table \ref{1}.

\begin{pr}\label{double} Let $(\G,\br,\prs)$ be a flat Lorentzian Lie algebra and $\cdot$ its Levi-Civita product. Suppose that there exists  a vector $e\in\G\setminus\{0\}$ such that $\langle e,e\rangle=0$ and, for any $u\in\G$, $u\cdot e=\la (u)e$ for some $\la\in\G^*$. Then:
	\begin{enumerate}
		\item  $\G=\R e\oplus\h\oplus\R f$ where $\h=\{e,f\}^\perp$ is a Euclidean vector subspace and $f$ satisfies $\langle f,f\rangle=0$ and $\langle e,f\rangle=1$. We denote by $\prs_\h$ the restriction of $\prs$ to $\h$. \item 
		$\h$ carries a product $\star$ such that its left multiplication operators are skew-symmetric and, for any $a,b\in\h$, the Levi-Civita product $\cdot$ of $\G$ has the form:
		\begin{equation}\label{Levi} \begin{cases} f.e=\la e,\; e.e=\al e,\; a.e=\langle a,u\rangle_\h e,\;e.a=Ea+\langle a,v\rangle_\h e,\\
				f.a=Fa+\langle a,w\rangle_\h e,\;a.b=a\star b+\langle Aa,b\rangle_\h e,\\
				a.f=-\langle a,u\rangle_\h f-Aa,\;f.f=-\la f-w,\; e.f=-v-\al f,
		\end{cases}\end{equation}where $E,F,A:\h\to\h$ are endomorphisms such that $E,F\in\so(\h)$, $u,v,w\in\h$,  $\al,\la\in\R$.  Denote by $\Lu_a$ and $\Ru_a$ the left and the right multiplication operators of $\star$.

		\item 	The vanishing of the curvature of $(\G,\br,\prs)$ is equivalent to
		\begin{equation}\label{curvature} \begin{cases}
				\ass_\star(a,b,c)-\ass_\star(b,a,c)=(\langle Ab,a\rangle_\h-\langle Aa,b\rangle_\h)  Ec,\\
				A([a, b])+b\star Aa-a\star Ab+\left(\langle Aa,b\rangle_\h-\langle Ab,a\rangle_\h\right) v
				+\langle b,u\rangle_\h Aa-\langle a,u\rangle_\h Ab=0,\\
				[E,\Lu_a]=\Lu_{Ea}+(\langle a,v\rangle_\h-\langle a,u\rangle_\h) E,\\
				[E,A]a-\langle a,v\rangle_\h v-a\star v+\al Aa=0,\\
				[F,\Lu_a]=\Lu_{Fa+Aa}+\langle a,w\rangle_\h E+\langle a,u\rangle_\h F,\\
				[F,A]a=A^2a-\la Aa+2\langle a,u\rangle_\h w+\langle a,w\rangle_\h v+a\star w,\\
				[F,E]=\la E+\al F+\Lu_v,\\
				Av-Fv+Ew+2\al w=0,\\ 
				A^*u+\la u-Fu+\al w=0,\\
				\langle u,v\rangle_\h=-2\al\la,\; \langle u,[a,b]_\star\rangle_\h=\al\left(\langle Ab,a\rangle_\h-\langle Aa,b\rangle_\h\right),\;  Eu=\al(u-v),
		\end{cases} \end{equation}	for any $a,b,c\in\h$, where $\ass_\star(a,b,c)=(a\star b)\star c-a\star(b\star c)$.
		\item  The modular vector of $\G$ is given by
		\begin{equation}\label{hm}
			\mathbf{h}=\mathbf{h}_0-v+(\tr(A)+\la)e-\al f,
		\end{equation}where $\mathbf{h}_0$ is the modular vector of $(\h,\star)$ given by
		\[ \langle \mathbf{h}_0,a\rangle_\h=-\tr(\Ru_a),\quad a\in\h. \]
	\end{enumerate}

\end{pr}

\begin{proof}\begin{enumerate}
		\item 	Let $\G_1 = e^\perp$. Then $\G_1$ contains $e$ and is a left ideal with respect to the Levi-Civita product. Choose a subspace $\h$ such that
		\[
		\G_1 = \mathbb{R}e \oplus \h.
		\]
		The restriction $\langle \cdot,\cdot \rangle_{|\h}$ is nondegenerate and positive definite, hence $(\h,\langle \cdot,\cdot \rangle)$ is a Euclidean vector space. Moreover, $\h^\perp$ is a two-dimensional Lorentzian subspace containing $e$. Therefore, we can choose $f \in \h^\perp$ such that
		\[
		\langle f,f\rangle = 0 \quad \text{and} \quad \langle e,f\rangle = 1.
		\]
		It follows that
		$
		\G = \mathbb{R}e \oplus \h \oplus \mathbb{R}f,
		$ 
		$\R e$ and $e^\perp=\mathbb{R}e \oplus \h$ are left ideals for  the Levi-Civita product. For any $a,b\in\h$, write $a\cdot b=a\star b+\langle Aa,b\rangle e$ where $a\star b\in\h$. This defines a product on $\h$ and its left multiplication operators are skew-symmetric. 
		Furthermore, one can see easily that the Levi-Civita product takes the form given in \eqref{Levi}.
		
		\item	The system of equations \eqref{curvature} is obtained through a direct but lengthy computation; we refer to the appendix (Section \ref{section9}, Subsection \ref{subsection91}) for the details.
		
		\item	Finally, using \eqref{h}, the modular vector $\mathbf{h}$ can be computed from \eqref{Levi} and the relation
		\[
		\langle \mathbf{h},x\rangle
		= -\sum_{i=1}^n \langle a_i \cdot x, a_i \rangle
		- \langle e \cdot x, f \rangle
		- \langle f \cdot x, e \rangle,
		\]
		where $(a_1,\ldots,a_n)$ is any orthonormal basis of $\h$.\qedhere
	\end{enumerate}
\end{proof}

Note that the Lie bracket on $\G = \mathbb{R}e \oplus \h \oplus \mathbb{R}f$ is given by
\begin{equation}\label{bracket}
	\begin{cases}
		[f,e]=v+\lambda e+\alpha f,\\
		[e,a]=Ea+\langle a,v-u\rangle_\h\, e,\\
		[a,b]=[a,b]_\star+\langle (A-A^*)a,b\rangle_\h\, e,\\
		[f,a]=(F+A)a+\langle a,w\rangle_\h\, e+\langle a,u\rangle_\h\, f,\quad a,b\in\h,
	\end{cases}
\end{equation}
where $[a,b]_\star = a \star b - b \star a$.

Thus, $\G$ is obtained from the algebra $(\h,\star)$ by a generalized double extension procedure determined by the data $(E,F,A,u,v,w,\lambda,\alpha)$ satisfying \eqref{curvature}.

%
%
%
%
%

We emphasize that the construction described in Proposition \ref{double} is significantly more general than the classical double extension process of \cite{Aub-Med}. In the classical setting, the algebra $(\h, [\cdot,\cdot]_\star, \langle\cdot,\cdot\rangle_\h)$ is required to be a flat pseudo-Euclidean Lie algebra. In contrast, in Proposition \ref{double}, the algebra $(\h, \star, \langle\cdot,\cdot\rangle_\h)$ is merely a Euclidean algebra — that is, an algebra whose left multiplication operators are skew-symmetric — with no flatness assumption. In particular, $(\h, [\cdot,\cdot]_\star)$ need not be a Lie algebra, and $\star$ need not satisfy the left-symmetric identity. The flatness of the extension $\G = \mathbb{R}e \oplus \h \oplus \mathbb{R}f$ is then entirely encoded in the system \eqref{curvature}, which couples the data $(E, F, A, u, v, w, \alpha, \lambda)$ to the algebra structure of $\h$.

For the reader’s convenience, we briefly recall the classical double extension procedure  in order to highlight the differences.

Let $(\h,[\cdot,\cdot]_0,\langle \cdot,\cdot \rangle_0)$ be a flat pseudo-Euclidean Lie algebra, and denote by $\star$ its Levi-Civita product. Let $A,D : \h \to \h$ be endomorphisms, set $F = D - A$, and assume that $F$ is skew-symmetric. Let $w \in \h$ and $\lambda \in \mathbb{R}$ be such that
\begin{equation}\label{doubleM}
	\begin{cases}
		A([a,b]_0)=a \star A(b)-b \star A(a),\\
		[D,A]=A^2-\lambda A + \mathrm{R}_{w},\\
		a \star A(b)-A(a \star b)=D(a)\star b+a \star D(b)-D(a \star b),
	\end{cases}
\end{equation}
for all $a,b \in \h$.

Consider the vector space $\G = \mathbb{R}e \oplus \h \oplus \mathbb{R}f$ endowed with the inner product $\langle \cdot,\cdot \rangle$ extending $\langle \cdot,\cdot \rangle_0$, such that $\h$ is orthogonal to $\mathrm{span}\{e,f\}$, and
\[
\langle e,e\rangle=\langle f,f\rangle=0, \quad \langle e,f\rangle=1.
\]
Define a Lie bracket on $\G$ by
\begin{equation}\label{classical_bracket}
	[f,e]=\lambda e, \quad [f,a]=D(a)+\langle w,a\rangle_0 e, \quad
	[a,b]=[a,b]_0+\langle (A-A^*)a,b\rangle_0 e,
\end{equation}
for all $a,b \in \h$. Then $(\G,[\cdot,\cdot],\langle \cdot,\cdot \rangle)$ is a flat pseudo-Euclidean Lie algebra, called the \emph{double extension} of $(\h,[\cdot,\cdot]_0,\langle \cdot,\cdot \rangle_0)$ by the data $(A,D,\lambda,w)$.

Now we give the conditions for which the Levi-Civita product given by \eqref{Levi} defines a Novikov algebra structure on $\G$. 

\begin{pr}\label{NENO}With the notation and hypotheses of Proposition \ref{double},
	$\G=\R e\oplus\h\oplus\R f$ endowed with the Levi-Civita product given by \eqref{Levi} is a Novikov algebra if and only if
	\begin{equation}\label{NE=0} \begin{cases} a\star (b\star c)-b\star (a\star c)=0,\; \Lu_{a\star b}=-\langle Aa,b\rangle_\h E,\\
			b\star Aa-a\star Ab=\langle a,u\rangle_\h Ab-\langle b,u\rangle_\h Aa,\\
			A(a\star b)=-\langle Aa,b\rangle_\h v,\;A^2a=-\langle a,u\rangle_\h w,\\
			[E,F]=[E,\Lu_a]=[F,\Lu_a]=0,\\
			EA=\Ru_v+\langle\bullet, u\rangle_\h v,\;
			FA=\Ru_w-\la A+\langle\bullet, u\rangle_\h w,\;\\
			AE=-\langle\bullet,v\rangle_\h v,\; AF=-\langle\bullet,w\rangle_\h v,\\
			\Lu_{Ea}=-\langle a,v\rangle_\h E,\Lu_{Fa}=-\langle a,w\rangle_\h E,\; \Lu_{Aa}=-\langle a,u\rangle_\h F,\\
			\la v-Ew+Fv=0,\\
			\al=0,\;\la v=0,\; \la E=0,\;|u|E=0,\; |u|v=0,\; Eu=Fu=0,\\
			\; A^*u=-\la u,\;\Ru_u=0,\\
			Aw=-\la w,\langle u,w\rangle_\h=-\la^2,\;\Lu_w=-\la F,\;\Lu_v=0,\; Av=0.
	\end{cases} \end{equation}
	
\end{pr}

\begin{proof}It is a
	direct but lengthy computation using Proposition \ref{Novikovpr}; we refer to the appendix (Section \ref{section9}, Subsection \ref{Subsection92}) for the details.
\end{proof}

\section{Complete description of flat Lorentzian Lie algebras: Proof of Theorem \ref{main} } \label{section4}

By Proposition \ref{double}, classifying flat Lorentzian Lie algebras of Kundt type is equivalent to determining all pairs $((\h, \star, \langle\cdot,\cdot\rangle_\h), (E, F, A, u, v, w, \alpha, \lambda))$ satisfying system \eqref{curvature}. Each solution gives rise, via the Lie bracket \eqref{bracket}, to a flat Lorentzian Lie algebra $\g = \mathbb{R}e \oplus \h \oplus \mathbb{R}f$ and all flat Lorentzian Lie algebras of Kundt type are obtained in this way. The goal of this section is to find all such solutions. The strategy is to use the expression \eqref{hm} for the modular vector to split the problem into five mutually exclusive and exhaustive cases, as explained below, and to solve \eqref{curvature} in each case by applying Theorem \ref{milnor} and the three key lemmas of Subsection \ref{subsection25}. In each case, we first determine the structure of $(\h, \star)$, then the endomorphisms $E$, $F$, $A$, and finally the vectors $u$, $v$, $w$, proceeding step by step.

We begin the first step of our strategy by  distinguishing two main cases: the unimodular case and the non-unimodular case.

\begin{enumerate}
	\item \textbf{The unimodular case.}  
	Assume that $(\h,\star,\langle \cdot,\cdot \rangle_\h)$ and the data $(E,F,A,u,v,w,\alpha,\lambda)$ satisfy \eqref{curvature} and that the Lie algebra $\G=\mathbb{R}e \oplus \h \oplus \mathbb{R}f$ is unimodular. This is equivalent to the vanishing of the modular vector given by \eqref{hm}, that is,
	\[
	v=\mathbf{h}_0 \quad \text{and} \quad \alpha=\tr(A)+\lambda=0.
	\]
	The last equation in \eqref{curvature} implies $Eu=0$ and $\langle u,v\rangle=0$. Taking $a=u$ in the third equation of \eqref{curvature}, we obtain
	\[
	[E,\Lu_u]=|u|^2 E.
	\]
	Since $\mathrm{ad}_{\Lu_u}:\so(\h)\to\so(\h)$ has no nonzero real eigenvalues, it follows that $|u|^2 E=0$.
	
	If $E=0$, then the first equation in \eqref{curvature} implies that $(\h,[\cdot,\cdot]_\star,\langle \cdot,\cdot \rangle_\h)$ is a flat Euclidean Lie algebra, hence unimodular. Therefore $\mathbf{h}_0=v=0$. We are thus led to the following subcases:
	\begin{enumerate}
		\item[(U1)] $u \neq 0$. In this case, $\alpha=0$, $\lambda+\tr(A)=0$, $E=0$, and $v=0$.
		
		\item[(U2)] $u=0$ and $E=0$. In this case, $\alpha=0$, $\lambda+\tr(A)=0$, and $v=0$.
		
		\item[(U3)] $E \neq 0$. In this case, $u=0$, $\alpha=0$, and $\lambda+\tr(A)=0$. From the third relation in \eqref{curvature}, we deduce that if $a \in \ker E$, then $\langle a,v\rangle=0$, hence $v \in \mathrm{Im}\,E$. Thus $v=E(v_0)$ for some $v_0 \in \im E$. Combining the seventh and the third equations of \eqref{curvature}, we obtain
		\[
		[F+\Lu_{v_0},E]=\lambda E,
		\]
		which implies $\lambda=0$.

		We now show that the adapted null
		decomposition can be chosen so that
		$v=0$.
		
		For any $x\in\h$, consider the
		isometric linear map
		\[
		\begin{aligned}
			\Phi_x:\h&\longrightarrow\G,\\
			a&\longmapsto
			a'=a+\langle a,x\rangle_\h e.
		\end{aligned}
		\]
		Its image
		$
		\h'=\Phi_x(\h)
		$
		is a Euclidean complement of
		$\mathbb Re$ in $e^\perp$.
		
		Putting
		\[
		e'=e,\qquad
		f'=f-x-\frac12|x|^2e,
		\]
		we obtain another adapted
		null decomposition
		$
		\G=\mathbb Re'\oplus\h'
		\oplus\mathbb Rf'.
		$
		Indeed,
		\[
		\langle e',e'\rangle
		=\langle f',f'\rangle=0,
		\qquad
		\langle e',f'\rangle=1,
		\]
		and
		\[
		\langle \Phi_x(a),f'\rangle=0,
		\qquad
		\langle\Phi_x(a),\Phi_x(b)\rangle
		=\langle a,b\rangle_\h.
		\]
		
		Since $u=\alpha=\lambda=0$,
		the expression \eqref{Levi}
		gives
		\[
		\begin{aligned}
			e'\cdot\Phi_x(a)
			&=e\cdot
			\bigl(a+\langle a,x\rangle_\h e\bigr)\\
			&=Ea+\langle a,v\rangle_\h e\\
			&=\Phi_x(Ea)
			+\langle a,v+Ex\rangle_\h e.
		\end{aligned}
		\]
		
		Consequently, the corresponding
		endomorphism $E'$ of $\h'$ and
		vector $v'\in\h'$ are given by
		\[
		E'=\Phi_x\circ E\circ\Phi_x^{-1},
		\qquad
		v'=\Phi_x(v+Ex).
		\]
		
		Since $v=E(v_0)$, choosing
		$x=-v_0$ gives
		$
		v'=0.
		$
		Moreover, $E'\neq0$, since
		$E'$ is conjugate to $E$.
		
		Thus, in this new adapted null
		frame, the Levi-Civita product
		takes the simplified form
		\begin{equation}\label{Levibis}
			\begin{cases}
				f'\cdot e'=e'\cdot f'=e'\cdot e'=0,\\
				a'\cdot e'=0,\qquad
				e'\cdot a'=E'a',\\
				f'\cdot a'=F'a'
				+\langle a',w'\rangle_{\h'}e',\\
				a'\cdot b'=a'\star' b'
				+\langle A'a',b'\rangle_{\h'}e',\\
				a'\cdot f'=-A'a',\\
				f'\cdot f'=-w',
			\end{cases}
		\end{equation}
		where $a',b'\in\h'$,
		$E',F'\in\so(\h')$,
		$A'\in\operatorname{End}(\h')$,
		and $w'\in\h'$.
		
		Hence, without loss of generality,
		we may assume $v=0$ throughout
		the study of case \textup{(U3)}.

	\end{enumerate}

	\item \textbf{The non-unimodular case.}
	
	Suppose that $\G$ is non-unimodular.
	By Proposition~\ref{pr1}, its modular
	vector $\mathbf{h}$ is nonzero and
	isotropic. Recall from \eqref{hm} that
	\[
	\mathbf{h}
	=\mathbf{h}_0-v+
	(\tr(A)+\lambda)e-\alpha f.
	\]
	Consequently,
	\begin{equation}\label{eq:modular-isotropy}
		|\mathbf{h}_0-v|^2
		=2\alpha(\tr(A)+\lambda).
	\end{equation}
	
	We distinguish two cases according
	to whether $\alpha$ vanishes.
	
	\begin{enumerate}
		\item[(NU1)] $\alpha=0$.
		In this case,
		\eqref{eq:modular-isotropy} gives
		$v=\mathbf{h}_0$. Hence
		$
		\mathbf{h}=(\tr(A)+\lambda)e.
		$
		Since $\mathbf{h}\neq0$, we have
		$\tr(A)+\lambda\neq0$.
		By Proposition~\ref{pr1},
		$\mathrm{R}_e$ is symmetric and
		$\mathbb Re$ is a two-sided ideal
		of $e^\perp$. Using \eqref{Levi},
		we obtain
		\[
		E=0,\qquad u=v=0.
		\]
		
		\item[(NU2)] $\alpha\neq0$.
		In this case, we first make a
		suitable change of adapted null
		frame.
		Put
		\[
		x=\frac{\mathbf{h}_0-v}{\alpha}
		\in\h.
		\]
		By \eqref{eq:modular-isotropy},
		\[
		\frac{\tr(A)+\lambda}{\alpha}
		=\frac12|x|^2.
		\]
		Consider the new adapted null frame
		defined by \[
		e'=e,\qquad
		a'=a+\langle a,x\rangle_\h e,
		\qquad
		f'=f-x-\frac12|x|^2e.
		\]As we have seen in the case (U3),
		this change of frame preserves
		the Lorentzian scalar product.
		In particular,
		\[
		\langle e',f'\rangle=1,\qquad
		\langle e',e'\rangle
		=\langle f',f'\rangle=0,
		\]
		and the new Euclidean complement is
		\[
		\h'=\{a+\langle a,x\rangle_\h e:a\in\h\}.
		\]
		
		Moreover, by the definition of $x$
		and the isotropy relation,
		\[
		\begin{aligned}
			-\alpha f'
			&=-\alpha f+\alpha x
			+\frac{\alpha}{2}\|x\|^2e
			=-\alpha f+\mathbf{h}_0-v
			+(\tr(A)+\lambda)e
			=\mathbf{h}.
		\end{aligned}
		\]
		Thus
		\begin{equation}\label{eq:modular-normalized}
			\mathbf{h}=-\alpha f'.
		\end{equation}
		
		The Levi-Civita product retains
		the form \eqref{Levi} in the new
		adapted frame, with transformed
		data
		\[
		(E',F',A',u',v',w',
		\alpha',\lambda').
		\]
		Since $e'=e$, we have
		$\alpha'=\alpha$.
		
		Comparing \eqref{eq:modular-normalized}
		with the expression \eqref{hm}
		of the modular vector in the new
		frame yields
		\[
		v'=\mathbf{h}_0',
		\qquad
		\tr(A')+\lambda'=0.
		\]
		
		Therefore, without loss of
		generality, we may work in an
		adapted null frame satisfying
		\[
		\mathbf{h}=-\alpha f,\qquad
		v=\mathbf{h}_0,\qquad
		\tr(A)+\lambda=0.
		\]
		
		All the computations in case
		\textup{(NU2)} will henceforth be
		carried out in this normalized
		frame.
		
		By Proposition~\ref{pr1},
		$\mathrm{R}_f$ is symmetric and
		$\mathbb Rf$ is a two-sided ideal
		of $f^\perp$.
		As will be shown in
		Proposition~\ref{NU2}, these
		properties imply
		\[
		u=v=w=0,\qquad
		A=F=0,\qquad
		\lambda=0.
		\]
	\end{enumerate}
	
\end{enumerate}

We now proceed to solve the system \eqref{curvature} in each of the cases listed above.

In the sequel, a Euclidean algebra is an algebra $(\h,\star,\prs)$ endowed with a scalar product such the left multiplication operators are skew-symmetric.
\subsection{ Solutions of \eqref{curvature} in the case $\mathrm{(U1)}$}\label{subsection41}

\begin{pr}\label{u} Let $(\h,\star,\prs)$ be a Euclidean algebra together with the data $(E,F,A,u,v,w,\al,\la)$ satisfying \eqref{curvature} such that the extension $\G$ is unimodular and $u\not=0$. Then:\begin{enumerate}
		\item[$(i)$]  $E=0$, $v=0$, $\al=\tr(A)+\la=0$,\item[$(ii)$]   $\h$ splits orthogonally $\h=\df\oplus\A\oplus\R u$,
		\item[$(iii)$]  there exists $\rho:\af=\A\oplus\R u\too\so(\df)$ satisfying $\bigcap_{a\in\af}\ker\rho(a)=\{0\}$,   $h\in\df$ and $n\in\A$ such that 
		\[ \begin{cases}
			\Lu_d=0,\; d\in\df,\; (\Lu_{a})_{|\af}=0,\;(\Lu_{a})_{|\df}=\rho(a),\; a\in\af,\\
			A_{|\df}=0,\; Aa=\Ru_h(a)=\rho(a)h,\; a\in\A,\; Au=m+n-\la u,\\
			F_{|\df}=-\frac1{|u|^2}(\rho(n)-\la\rho(u)),\; F_{|\af}=0,\\
			m=\rho(u)h+|u|^2h,\; w=w_1+w_2+w_3,\\
			w_1=\la h-\frac1{|u|^2}\rho(n)(h)+\frac{\la}{|u|^2}\rho(u)h,\; w_2=\frac{\la n}{|u|^2},\; w_3=-\frac{\la^2}{|u|^2}u.
		\end{cases} \]
		
	\end{enumerate}
	Moreover,  the extension $\G=\R e\oplus\h\oplus\R f$ endowed with the Levi-Civita product \eqref{Levi} is a Novikov algebra and the structure of flat Lorentzian Lie algebra in $\G$ after the simplification below is given in Table \ref{2}.

\end{pr}

\begin{proof}
	We have seen that in the unimodular case with $u \neq 0$ one has $\alpha=0$, $v=0$ and $E=0$. Hence system \eqref{curvature} reduces to
	\begin{equation}\label{Leviu}
		\begin{cases}
			\ass_\star(a,b,c)-\ass_\star(b,a,c)=0,\\
			A([a, b])+b\star Aa-a\star Ab
			+\langle b,u\rangle Aa-\langle a,u\rangle Ab=0,\\
			[F,\Lu_a]=\Lu_{Fa+Aa}+\langle a,u\rangle_\h F,\\
			2\langle a,u\rangle w+A^2a-\lambda Aa+a\star w=
			[F,A]a,\\ 
			A^*u+\lambda u-Fu=0,\\
			\langle u,[a,b]_\star\rangle=0,\;a,b,c\in\h.
		\end{cases}
	\end{equation}

	The first equation means that $(\h,[\cdot,\cdot]_\star,\langle\cdot,\cdot\rangle_\h)$ is a flat Euclidean Lie algebra. By Theorem \ref{milnor}, $\h$ decomposes orthogonally as
	\[
	\h=\df\oplus\af,
	\]
	and there exists a representation of an abelian Lie algebra
	\[
	\rho:\af\to\so(\df), \qquad \bigcap_{a\in\af}\ker\rho(a)=\{0\},
	\]
	such that the product $\star$ is given by
	\[
	a\star b=
	\begin{cases}
		0 & \text{if } a\in\df \text{ or } b\in\af,\\
		\rho(a)b & \text{if } a\in\af,\; b\in\df.
	\end{cases}
	\]
	The last equation of \eqref{Leviu} implies $\langle u,[a,b]\rangle=0$, hence $u\in\af$. Let $\A=u^\perp\cap\af$, so that $\af=\mathbb{R}u\oplus\A$. 
	With respect to the decomposition $\h=\df\oplus\af$, we write
	\[
	A=
	\begin{pmatrix}
		A_1 & B_1 \\
		A_2 & B_2
	\end{pmatrix}.
	\]
	
	\medskip
	\noindent\textbf{Step 1: Determination of $A$.}
	
	We analyze the second equation in \eqref{Leviu}. A case-by-case study yields:
	\begin{itemize}
		\item For $a,b\in\df$, the relation is automatically satisfied.
		\item For $a,b\in\A$, we obtain
		\[
		\rho(b)B_1(a)-\rho(a)B_1(b)=0.
		\]
		\item For $a\in\A$, $b=u$, we get
		\[
		(\rho(u)+|u|^2\mathrm{Id})B_1(a)=\rho(a)B_1(u), \qquad B_2(a)=0.
		\]
		Since $\rho(u)+|u|^2\mathrm{Id}$ is invertible and commutes with $\rho(a)$,  $h=(\rho(u)+|u|^2\mathrm{Id})^{-1}B_1(u)\in\df$ satisfies
		\[
		B_1(a)=\rho(a)h.
		\]
		\item For $a\in\A$, $b\in\df$, we obtain
		\[
		A_1\circ\rho(a)=\rho(a)\circ A_1, \qquad A_2\circ\rho(a)=0.
		\]
		\item For $a=u$, $b\in\df$, we obtain
		\[
		[A_1,\rho(u)]=|u|^2 A_1, \qquad A_2\circ(\rho(u)-|u|^2\mathrm{Id})=0.
		\]
		Since $\ad_{\rho(u)}$ has no nonzero real eigenvalue and $\im(\rho(u)-|u|^2\mathrm{Id})=\af$, we deduce $A_1=0$ and $A_2=0$.
	\end{itemize}
	
	From the fifth relation in \eqref{Leviu}, $\langle Au,u\rangle=-\la |u|^2$ and  therefore, in the splitting $\h=\df\oplus\A\oplus\R u$,
	\[
	A=
	\begin{pmatrix}
		0 & \Ru_h & m\\
		0 & 0 & n\\
		0 & 0 & -\lambda
	\end{pmatrix},
	\]
	where $h\in\df$, $n\in\A$, and $m=(\rho(u)+|u|^2\mathrm{Id})h$. Note that $\tr(A)+\la=0$.
	
	\medskip
	\noindent\textbf{Step 2: Determination of $F$.}
	
	In the splitting $\h=\df\oplus\af$, write
	\[ F=
	\begin{pmatrix}
		F_1 & F_2 \\
		- F_2^* & F_3 \\
	\end{pmatrix}. \]
	We now consider the third equation in \eqref{Leviu}. Again, examining cases:
	\begin{itemize}
		\item For $a\in\df$, we obtain $\rho(F_2^*(a))=0$. 
		\item For $a\in\A$, we obtain 
		\[ \left[ \begin{pmatrix}
			F_1 & F_2 \\
			- F_2^* & F_3 \\
		\end{pmatrix},\begin{pmatrix}
			\rho(a)&0 \\
			0 & 0 \\
		\end{pmatrix}  \right]=\begin{pmatrix}
			\rho(F_2(a))&0 \\
			0 & 0 \\
		\end{pmatrix}. \]
		This is equivalent to $[F_1,\rho(F_2(a))]=\rho(F_2(a))=0$.
		
		\item For $a=u$,  we obtain 
		\[ \left[ \begin{pmatrix}
			F_1 & F_2 \\
			- F_2^* & F_3 \\
		\end{pmatrix},\begin{pmatrix}
			\rho(u)&0 \\
			0 & 0 \\
		\end{pmatrix}  \right]=\begin{pmatrix}
			\rho(F_2(u))+\rho(n)-\la \rho(u)&0 \\
			0 & 0 \\
		\end{pmatrix}+|u|^2\begin{pmatrix}
			F_1 & F_2 \\
			- F_2^* & F_3 \\
		\end{pmatrix}.\]This is equivalent to
		\[
		\begin{cases} \rho(u)F_2=-|u|^2F_2,\; F_3=0,\\
			[F_1,\rho(u)]=|u|^2F_1+\rho(F_2(u))+\rho(n)-\lambda\rho(u).\end{cases}
		\]But $\rho(u)$ is skew-symmetric and hence this is equivalent to $F_2=0$, $F_3=0$ and
		\[ [\rho(u),F_1+\frac1{|u|^2}\rho(n)-\frac{\la}{|u|^2}\rho(u)]=-|u|^2\left(F_1+\frac1{|u|^2}\rho(n)-\frac{\la}{|u|^2}\rho(u)\right). \] Since $\ad_{\rho(u)}:\so(\df)\too\so(\df)$ is skew-symmetric and has no non zero eigenvalue this relation has a unique solution $F_1=\frac1{|u|^2}(-\rho(n)+\la\rho(u))$.
	\end{itemize}Finally, 
	\[ F=
	\begin{pmatrix}
		\frac1{|u|^2}(-\rho(n)+\la\rho(u)) & 0 \\
		0 & 0
	\end{pmatrix}. \]
	
	\medskip
	\noindent\textbf{Step 3: Determination of $w$.}
	
	Finally, we consider the fourth equation in \eqref{Leviu}:$$2\langle a,u\rangle w+A^2a-\la Aa+a\star w= [F,A]a.$$
	Write
	\[
	w=w_1+w_2+\nu u, \qquad w_1\in\df,\; w_2\in\A.
	\]
	For $a\in\df$ this relation holds. For $a\in\A$ or $a=u$, we get \[\begin{cases} -\la a\star h+a\star w_1=-\frac1{|u|^2}\left(\rho(n)\rho(a)-\la\rho(a)\rho(u)\right) h,\\ 2|u|^2w+n\star h-2\la m-2\la n+2\la^2u +u\star w_1=-\frac1{|u|^2}(\rho(n)-\la\rho(u))(m). \end{cases} \] But $m=\rho(u)h+|u|^2h$ and hence \[\begin{cases} 2|u|^2w_1+2\rho(n)h-2\la\rho(u)h-2\la|u|^2h+\rho(u)w_1+\frac1{|u|^2}\rho(n)\rho(u)h -\frac{\la}{|u|^2}\rho(u)^2h-\la\rho(u)h=0,\\ \rho(a)((w_1+\frac1{|u|^2}(\rho(n)-\la\rho(u))h)=\la \rho(a)(h),\\
		2|u|^2 w_2-2\la n=0,\\ 2|u|^2\nu +2\la^2=0. \end{cases} \]
	From the second relation, one can write \[ w_1=\la h-\frac1{|u|^2}\rho(n)(h)+\frac{\la}{|u|^2}\rho(u)h+X\;\qquad\mbox{with}\;\rho(a)X=0. \] If we replace in the first relation, we get \begin{align*} 0&=2\la|u|^2 h-2\rho(n)(h)+2\la\rho(u)h+2|u|^2X +2\rho(n)h-2\la\rho(u)h-2\la|u|^2h\\ &+\la\rho(u) h-\frac1{|u|^2}\rho(u)\rho(n)(h)+\frac{\la}{|u|^2}\rho(u)^2h+\rho(u)X +\frac1{|u|^2}\rho(n)\rho(u)h -\frac{\la}{|u|^2}\rho(u)^2h-\la\rho(u)h\\ &=\rho(u)X+2|u|^2X. \end{align*}The relation $\rho(u)X=-2|u|^2X$ implies $X=0$. Then
	\[ w_1=\la h-\frac1{|u|^2}\rho(n)(h)+\frac{\la}{|u|^2}\rho(u)h,\; w_2=\frac{\la n}{|u|^2},\; w_3=-\frac{\la^2}{|u|^2}u. \]
	This completes the proof of the first part of the proposition. For the verification that the Levi-Civita product on the extension $\G$ induces a Novikov algebra structure  see Section \ref{section9} Subsection \ref{subsection93}. 
	
	Now, we replace the solutions $(E,F,A,u,v,w,\al,\la)$ obtained in this case in \eqref{bracket} to get the Lie bracket on the extension $\G=\R e\oplus\df\oplus\A\oplus\R u\oplus\R f$:
	\begin{equation*}\begin{cases}
			[f,e]=\la e,\;[u,e]=|u|^2 e,\;\\ [a,b]=\rho(a)b-\langle \rho(a)b,h\rangle e,\; [u,b]=\rho(u)b+\langle m,b\rangle e, [a,u]=-\langle n,a\rangle e,\\ 
			[f,b]=-\frac1{|u|^2}(\rho(n)-\la\rho(u))(b)+\langle b,w_1\rangle_\h e,\\
			[f,a]=\rho(a)h+\frac{\la}{|u|^2}\langle a,n\rangle_\h e,\\
			[f,u]=\rho(u)h+|u|^2h+n-\la u-{\la^2}e+|u|^2 f,\\
		\end{cases}
	\end{equation*}
	If we replace $u$  by  $\frac{u}{|u|^2}$ and $n$ by $\frac{n}{|u|^2}$, we get
	a first simplification of the Lie bracket.
\end{proof}

We can go further and get a more simplified version of the Lie brackets of the Lorentzian Lie algebra obtained in the preceding proposition. Indeed, denote by
 $\G_2$ this Lie algebra. We have
	\[
	\G_2=\mathbb Re\oplus\df\oplus\A\oplus\mathbb Ru\oplus\mathbb Rf,
	\qquad \af=\A\oplus\mathbb Ru.
	\]
	The nonzero Lie brackets are
	\begin{equation}\label{eq:g2-original-h}
		\begin{cases}
			[f,e]=\lambda e,\qquad [u,e]=e,\\
			[a,d]=\rho(a)d-\langle h,\rho(a)d\rangle e,\\
			[u,d]=\rho(u)d+\langle m,d\rangle e,\\
			[u,a]=\langle n,a\rangle e,\\
			[f,d]=F_0d+\langle d,w\rangle e,\\
			[f,a]=\rho(a)h+\lambda\langle a,n\rangle e,\\
			[f,u]=m+n-\lambda u-\lambda^2|u|^2e+f,
		\end{cases}
	\end{equation}
	where $a\in\A$ and $d\in\df$ and
	\[
	F_0=-\rho(n)+\lambda\rho(u),\qquad
	m=(\rho(u)+\mathrm{Id})h,\qquad
	w=(\lambda\mathrm{Id}+F_0)h.
	\]
	
	Consider the change of adapted null frame corresponding to $x=-h$:
	\[
	e'=e,\qquad
	d'=d-\langle d,h\rangle e\quad(d\in\df),\qquad
	f'=f+h-\frac12|h|^2e.
	\]
	The vectors of $\A\oplus\mathbb Ru$ are left unchanged. Set
	\[
	\df'=\{d-\langle d,h\rangle e:d\in\df\},\qquad
	\Phi:\df\longrightarrow\df',\quad
	\Phi(d)=d'.
	\]
	The map $\Phi$ is an isometry and
	\[
	\G_2=\mathbb Re'\oplus\df'\oplus\A\oplus
	\mathbb Ru\oplus\mathbb Rf'
	\]
	is an orthogonal adapted null decomposition. Define
	\[
	\rho'(a)=\Phi\rho(a)\Phi^{-1}\quad(a\in\af),
	\qquad F_0'=\Phi F_0\Phi^{-1}
	=-\rho'(n)+\lambda\rho'(u).
	\]
	Substituting into \eqref{eq:g2-original-h} yields
	\begin{equation}\label{eq:g2-normalized-h}
		\begin{cases}
			[f',e']=\lambda e',\; [u,e']=e',\\
			[a,d']=\rho'(a)d',\; [u,d']=\rho'(u)d',\\
			[u,a]=\langle n,a\rangle e',\\
			[f',d']=F_0'd',\\
			[f',a]=\lambda\langle a,n\rangle e',\\
			[f',u]=f'+n-\lambda u-\lambda^2\|u\|^2e',
		\end{cases}
	\end{equation}
	Thus the vector parameter $h$ is inessential, and the associated
	parameters $m$ and $w$ disappear. The new representation $\rho'$
	remains a representation of the abelian Lie algebra $\af$, and
	\[
	[\rho'(a),\rho'(b)]=0,\qquad
	\bigcap_{a\in\af}\ker\rho'(a)=\{0\}.
	\]
	After dropping the primes, we may use \eqref{eq:g2-normalized-h}
	as the model for $\G_2$ in Table~\ref{2}.

\subsection{ Solutions of \eqref{curvature} in the case $(\mathrm{U2})$}\label{subsection42}

\begin{pr}\label{u0} Let $(\h,\star,\prs)$ be a Euclidean algebra together with the data $(E,F,A,u,v,w,\al,\la)$ satisfying \eqref{curvature} such that the extension $\G$ is unimodular $E=0$ and $u=0$. Then:\begin{enumerate}
		
		\item[$(i)$]  $v=0$, $\al=\la=\tr(A)=0$, \item[$(ii)$]  $\h$ splits orthogonally $\h=\df\oplus\af$,
		\item[$(iii)$]  there exists $\rho:\af\too\so(\df)$ a representation of abelian Lie algebra, $h\in\df$ and $w_2\in\af$  such that  $\bigcap_{a\in\af}\ker\rho(a)=\{0\}$ and, for all $a\in\af$, 
		\[ \begin{cases}\Lu_d=0,\; d\in\df,\; (\Lu_{a})_{|\af}=0,\;(\Lu_{a})_{|\df}=\rho(a),\; a\in\af,\\
			A=\begin{pmatrix}
				A_1&\Ru_h\\
				0&A_2
			\end{pmatrix},\;F=\begin{pmatrix}
				F_1&0\\
				0&F_2
			\end{pmatrix},\; w=F_1(h)-A_1(h)+w_2, \; F_1\in\so(\df),\;F_2\in\so(\af)\\
			[F_1,\rho(a)]=[A_1,\rho(a)]=0,\; [F_1,A_1]=A_1^2, [F_2,A_2]=A_2^2,\\
			\im(F_2+A_2)\subset\ker\rho,\; \Ru_h(a)=\rho(a)h.
		\end{cases} \]
	\end{enumerate}

	Moreover,  the extension $\G=\R e\oplus\df\oplus\af\oplus\R f$ endowed with the Levi-Civita product is a Novikov algebra if and only if 
	\[
\begin{gathered}
[F_1,\rho(a)]=0,\quad A_1=0,\quad A_2^2=0,\\
\operatorname{Im}F_2\subseteq\ker\rho,\quad
\operatorname{Im}A_2\subseteq\ker\rho,\\
A_2F_2=F_2A_2=0,\quad \rho(w_2)h=0,\quad A_2w_2=0.
\end{gathered}
\] 
	The structure of flat Lorentzian Lie algebra on $\G$ is given in Table \ref{3} after the simplification given below.

\end{pr}

\begin{proof}
	We have seen that if $\G$ is unimodular, $u=0$ and $E=0$, then $\al=0$, $v=0$ and $\tr(A)+\la=0$. In this case, system \eqref{curvature} reduces to
	\[
	\begin{cases}
		\ass_\star(a,b,c)-\ass_\star(b,a,c)=0,\\
		A([a, b])+b\star Aa-a\star Ab=0,\\
		[F,\Lu_a]=\Lu_{F(a)+A(a)},\\
		A^2a-\la Aa+a\star w=[F,A]a,\quad a,b,c\in\h.
	\end{cases}
	\]
	
	The first equation is equivalent to the fact that $(\h,\br_\star,\prs)$ is a flat Euclidean Lie algebra. By Theorem \ref{milnor}, we have an orthogonal decomposition $\h=\df\oplus\af$, where $\af$ is abelian, and there exists a representation
	\[
	\rho : \af \to \so(\df)
	\]
	such that $\bigcap_{a \in \af} \ker \rho(a)=\{0\}$, and the product $\star$ is given by
	\[
	a\star b=
	\begin{cases}
		0 & \text{if } a\in\df \text{ or } b\in\af,\\
		\rho(a)b & \text{if } a\in\af,\; b\in\df.
	\end{cases}
	\]
	
	Since the Lie subalgebra $\{\Lu_a \mid a\in\h\}$ is abelian, the third relation  shows that it is invariant under $\ad_F$. Hence, by virtue of Proposition \ref{solvable}, for any $a\in\h$,
	\[
	\Lu_{F(a)+A(a)}=[F,\Lu_a]=0.
	\]
	Using the fact that $\df=\h\star\h$ and $F$ is skew-symmetric, we deduce that $F(\df)\subset\df$ and $F(\af)\subset\af$.
	
	We now analyze the fourth equation
	\[
	A^2a-\la Aa+a\star w=[F,A]a,
	\]
	which can be rewritten as
	\[
	[F,A]=A^2-\la A+\Ru_w.
	\]
	Decompose $w=w_1+w_2$ with $w_1\in\df$ and $w_2\in\af$. By Lemma \ref{AM}, there exists $h\in\df$ such that, with respect to the decomposition $\h=\df\oplus\af$, the operators take the form
	\[
	A=\begin{pmatrix}
		A_1 & \Ru_h\\
		0 & A_2
	\end{pmatrix},\quad
	F=\begin{pmatrix}
		F_1 & 0\\
		0 & F_2
	\end{pmatrix},\quad
	\Lu_a=\begin{pmatrix}
		\rho(a) & 0\\
		0 & 0
	\end{pmatrix},\quad
	\Ru_w=\begin{pmatrix}
		0 & \Ru_{w_1}\\
		0 & 0
	\end{pmatrix}.
	\]
	Therefore,
	\[
	[F,A]=\begin{pmatrix}
		[F_1,A_1] & F_1\Ru_h-\Ru_hF_2\\
		0 & [F_2,A_2]
	\end{pmatrix},\quad
	A^2=\begin{pmatrix}
		A_1^2 & A_1\Ru_h+\Ru_hA_2\\
		0 & A_2^2
	\end{pmatrix}.
	\]
	Hence the fourth equation is equivalent to
	\[
	\begin{cases}
		[F_1,\rho(a)]=0,\quad [A_1,\rho(a)]=0,\\
		[F_1,A_1]=A_1^2-\la A_1,\quad [F_2,A_2]=A_2^2-\la A_2,\\
		F_1\Ru_h-\Ru_hF_2=A_1\Ru_h+\Ru_hA_2-\la \Ru_h+\Ru_{w_1}.
	\end{cases}
	\]
	
	Assume that $\la\neq 0$. Then, by Lemma \ref{LF}, we obtain $A_1^2-\la A_1=0$ and $A_2^2-\la A_2=0$. It follows that
	\[
	\tr(A)=\tr(A_1)+\tr(A_2)=n\la,\quad n\geq 0,
	\]
	which contradicts the relation $\tr(A)+\la=0$. Therefore, $\la=0$.
	
	Finally, using $[F_1,\rho(a)]=[A_1,\rho(a)]=0$, the relation
	\[
	F_1(a\star h)-F_2(a)\star h=A_1(a\star h)+A_2(a)\star h+a\star w_1
	\]
	is equivalent to
	\[
	\rho(a)\big(F_1(h)-A_1(h)-w_1\big)=\rho\big(F_2(a)+A_2(a)\big).
	\]
	Since $\Lu_{F(a)+A(a)}=0$, we have $\rho(F_2(a)+A_2(a))=0$, hence
	\[
	w_1=F_1(h)-A_1(h).
	\]
	This concludes the proof of the first part. The verification of the Novikov property is in Section \ref{section9}, Subsection \ref{subsection94}. Finally, by using \eqref{bracket} and the solutions obtained we get the structure of flat Lorentzian Lie algebra listed in Table \ref{3}.
\end{proof}

The Lie algebra structure obtained from this proposition can be simplified.
	In the notation of Proposition~\ref{u0}, write $w=w_1+w_2$, where
	$w_1=(F_1-A_1)h\in\df$ and $w_2\in\af$. The original nonzero
	brackets are
	\begin{equation}\label{eq:g3-old}
		\begin{cases}
			[d_1,d_2]=\langle(A_1-A_1^*)d_1,d_2\rangle e,\\
			[a_1,a_2]=\langle(A_2-A_2^*)a_1,a_2\rangle e,\\
			[a,d]=\rho(a)d-\langle h,\rho(a)d\rangle e,\\
			[f,d]=(F_1+A_1)d+\langle d,(F_1-A_1)h\rangle e,\\
			[f,a]=(F_2+A_2)a+\rho(a)h+\langle a,w_2\rangle e,
		\end{cases}
	\end{equation}
	for $a,a_1,a_2\in\af$ and $d,d_1,d_2\in\df$. Here $e$ is central.
	
	Consider the adapted null-frame change with $x=-h$:
	\[
	e'=e,\quad d'=d-\langle d,h\rangle e\ (d\in\df),\quad
	a'=a\ (a\in\af),\quad f'=f+h-\tfrac12|h|^2e.
	\]
	Set $\df'=\{d':d\in\df\}$ and
	$\Phi:\df\to\df'$, $\Phi(d)=d'$. Then $\Phi$ is an isometry,
	\[
	\G_3=\mathbb Re'\oplus\df'\oplus\af\oplus\mathbb Rf'
	\]
	is an adapted null decomposition. Transport $\rho$, $A_1$, and $F_1$
	through $\Phi$; 
	\[ \rho'(a)=\Phi\rho(a)\Phi^{-1},\; F_1'=\Phi F_1\Phi^{-1},\; A_1'=\Phi A_1\Phi^{-1}. \]
	The operators $A_2,F_2$ and the vector $w_2$ are
	unchanged. Since $[e,\G_3]=0$, the brackets transform as follows:
	\[
	\begin{aligned}
		[a,d']&=\rho(a)d-\langle h,\rho(a)d\rangle e
		=\Phi(\rho(a)d)=\rho'(a)d',\\
		[f',d']&=(F_1+A_1)d+\langle d,(F_1-A_1)h\rangle e
		+\langle(A_1-A_1^*)h,d\rangle e
		=\Phi((F_1+A_1)d)=(F_1'+A_1')d',\\
		[f',a]&=(F_2+A_2)a+\rho(a)h+\langle a,w_2\rangle e
		-\rho(a)h
		=(F_2+A_2)a+\langle a,w_2\rangle e.
	\end{aligned}
	\]
	The central brackets $[d_1,d_2]$ and $[a_1,a_2]$ remain unchanged.
	Consequently, after dropping primes, $h$ and the accompanying
	vector $w_1$ disappear from the presentation, while all the
	conditions on $\rho,A_i,F_i,w_2$ are preserved. In particular,
	the condition $\rho(w_2)h=0$ in the Novikov criterion becomes
	vacuous in this normalized presentation.

\subsection{Solutions of \eqref{curvature}
	in case $(\mathrm{U3})$}

We now consider the unimodular case
$E\neq0$, which is the most involved
part of our classification.
As shown in the preliminary case
analysis, a suitable choice of adapted
null frame allows us to assume
$v=0$. In this frame, we also have
$
u=0,\;
\alpha=\lambda=\tr(A)=0.
$

The resulting system is considerably
simpler than the original flatness
equations. Nevertheless, its complete
resolution requires several structural
arguments.

The key idea is to exploit the
orthogonal decomposition
$
\h=\ker E\oplus\operatorname{Im}E.
$
We first show that $\ker E$ is a
flat Euclidean Lie algebra and apply
the refined Milnor theorem.
We then establish a nilpotency
property for the operators
$A+R_a$, with $a\in\ker E$.
This allows us to determine the
restriction of $A$ to $\ker E$.

The most delicate part of the proof
is the determination of the structure
on $\operatorname{Im}E$, where we
show that the corresponding left
multiplication operators vanish.
Finally, the remaining equations,
together with the residual freedom
in the adapted null frame,
determine $F$ and $w$.

The resulting models have a
particularly simple structure,
despite the substantial algebraic
work required to obtain them.

\begin{pr}\label{E} Let $(\h,\star,\prs)$ be a Euclidean algebra and the data $(E,F,A,u,v,w,\al,\la)$ satisfying \eqref{curvature} such that the extension $\G$ is unimodular, $E\not=0$ and $v=0$. Then:\begin{enumerate}
		
		\item $u=0$, $\al=0$, $\tr(A)=\la=0$,\item  $\h$ splits orthogonally $\h=\df\oplus\af\oplus\mathfrak{b}$, \item  there exists $\rho_1:\af\too\so(\df)$, $\rho_2:\af\too\so(\mathfrak{b})$, $h\in\df$ and $g\in\af$  such that  $\bigcap_{a\in\af}\ker\rho_1(a)=\{0\}$,   and for any $a\in\af$, $d\in\df$, $b\in \mathfrak{b}$,
		\[
\begin{cases}
(\Lu_a)_{|\df}=\rho_1(a),\quad
(\Lu_a)_{|\af}=0,\quad
(\Lu_a)_{|\mathfrak b}=\rho_2(a),\\
\Lu_d=\langle d,h\rangle E,\qquad \Lu_b=0,\\
A(d)=A(b)=0,\quad A(a)=\rho_1(a)h,\quad
w=F_1(h)+F_2(g),\\
\ker E=\df\oplus\af,\qquad
E_{|\mathfrak b}=E_1,\\
F_{|\df}=F_1,\quad F_{|\af}=F_2,\quad
F_{|\mathfrak b}=F_3,\\
F_1\in\so(\df),\quad F_2\in\so(\af),\quad
F_3,E_1\in\so(\mathfrak b),\\
\det(E_1)\neq0,\qquad
[F_1,\rho_1(a)]=\rho_1(F_2a)=0,\\
[F_3,\rho_2(a)]=[F_3,E_1]=[E_1,\rho_2(a)]=0,\\
\rho_2(F_2a)=\langle F_2a,g\rangle E_1.
\end{cases}
\]	
	\end{enumerate}

	Moreover,  the extension $\G=\R e\oplus\h\oplus\R f$ endowed with the Levi-Civita product is a Novikov algebra and  the structure of flat Lorentzian Lie algebra on $\G$ after the simplification below is given in Table \ref{4}.
\end{pr}

\begin{proof} 
	
	We have seen in the case analysis at the beginning of this section (the case (U3)) that $u=0$,  $\la=\al=0$, $\tr(A)=0$. So \eqref{curvature} becomes
	\begin{equation}\label{LeviE} \begin{cases}
			\Lu_{[a,b]}-[\Lu_a,\Lu_b]=(\langle Ab,a\rangle_\h-\langle Aa,b\rangle_\h)  E,\\
			A([a, b])+b\star Aa-a\star Ab
			=0,\\
			[E,\Lu_a]=\Lu_{E(a)},\;[E,A]=0,\;[F,E]=0,\\
			[F,\Lu_a]=\Lu_{F(a)+A(a)}+\langle a,w\rangle_\h E,\\
			[F,A]=A^2+\Ru_{w},\\
			Ew=0,\; 
			\tr(\Ru_a)=0,\; a,b,c\in\h.
	\end{cases} \end{equation}
	Put $\h_0=\ker E$ and $\mathfrak{b}=\im E.$

	The proof is lengthy, and some technical verifications are deferred to the appendix Section \ref{section9}. We proceed in five steps.

	\medskip
	\noindent\textbf{Step 1: $\h_0$ is a flat Euclidean Lie algebra.}\\
	
	For any $a,b\in\h_0$, the third equation in \eqref{LeviE} yields $E(a\star b)=0$, showing that $\h_0$ is a subalgebra of $(\h,\star)$. Moreover, for any $a,b,c\in\h_0$, the first equation implies
	\[
	\ass(a,b,c)-\ass(b,a,c)=0,
	\]
	and thus $(\h_0,\br_\star)$ is a flat Euclidean Lie algebra. Consequently, by virtue of Theorem \ref{milnor}, there exists an orthogonal decomposition $\h_0=\df\oplus\af$, where $\af$ and $\df$ are abelian, together with a representation
	\[
	\rho_1:\af\to\so(\df)
	\]
	such that $\bigcap_{a\in\af}\ker\rho_1(a)=\{0\}$ and the restriction of $\star$ to $\h_0$ is given by \eqref{Levi-Milnor}.

	Now we proceed to solve \eqref{LeviE}.
	
	\medskip
	
	\medskip
	\noindent\textbf{Step 2: For any $a\in\h_0$, the operator $A+\Ru_a$ is nilpotent.}\\
	
	Let $a\in\h_0$, i.e., $E(a)=0$. We start by establishing the relations
	\begin{equation}\label{first}
		[{F},\Ru_a]=\Ru_{{F}(a)}+\Ru_a\circ A,\qquad
		[\Ru_a,\Lu_a]=(\Ru_a)^2-\Ru_{\Ru_a(a)}\esp [A,\Lu_a]=A\circ\Ru_a-\Ru_{Aa}.
	\end{equation}
	Since \(E(a)=0\), the first relation in \eqref{LeviE} yields, for any \(b\in\h\),
	\[
	\ass_{\star}(a,b,a)-\ass_{\star}(b,a,a)
	=(a\star b)\star a - a\star(b\star a)
	-(b\star a)\star a + b\star(a\star a)=0.
	\]
	This identity gives precisely the second relation in \eqref{first}. Moreover, the third relation in \eqref{first} is simply a reformulation of the second relation in \eqref{LeviE}.
	
	Next, let \(b\in\h\). Using the fourth relation in \eqref{LeviE}, we obtain
	\[
	{F}(b\star a) - b\star {F}(a)
	= (F(b)+A(b))\star a,
	\]
	which is exactly the first relation in \eqref{first}.

	Set $C=A+\Ru_a$ and $G={F}-\Lu_a$. Since $\tr(\Ru_a)=0$,
	$\tr(C)=\tr(A)=0.$
	Having in mind \eqref{first}, we compute
	\begin{align*}
		[G,C]
		&=[{F},A]+[{F},\Ru_a]-[\Lu_a,A]-[\Lu_a,\Ru_a]\\
		&=A^2+\Ru_{{w}}+\Ru_{{F}(a)}+\Ru_a\circ A
		+A\circ\Ru_a-\Ru_{Aa}+(\Ru_a)^2-\Ru_{\Ru_a(a)}\\
		&=C^2+\Ru_{{w}+{F}(a)-C(a)}
		= C^2+\Ru_{Q(a)},
	\end{align*}
	where $Q(a)={w}+{F}(a)-C(a)$. 
	Moreover, since $a\in\ker E$,
	the first equation of \eqref{LeviE}
	implies
	\[
	[\Ru_a,\Lu_b]=\Ru_a\circ \Ru_b-\Ru_{b\star a},
	\qquad b\in\h.
	\]
	Combining this identity with
	\[
	[A,\Lu_b]=A\circ \Ru_b-\Ru_{Ab},
	\]
	we obtain
	\begin{equation}\label{eq:C-commutator}
		[C,\Lu_b]=C\circ \Ru_b-\Ru_{Cb},
		\qquad b\in\h.
	\end{equation}
	For every $n\geq1$, cyclicity
	of the trace gives
	\[
	\tr(C^{n-1}[C,\Lu_b])=0.
	\]
	Therefore,
	\[
	\tr(C^n\circ \Ru_b)
	=\tr(C^{n-1}\Ru_{Cb}).
	\]
	Iterating this identity yields
	\begin{equation}\label{eq:C-trace}
		\tr(C^n\circ \Ru_b)
		=\tr(\Ru_{C^nb})=0,
		\qquad b\in\h,
	\end{equation}
	where the last equality follows
	from \eqref{LeviE}.
	On the other hand, recalling that
	\[
	[G,C]=C^2+\Ru_{Q(a)},
	\]
	we obtain, for $n\geq1$,
	\[
	[G,C^n]
	=nC^{n+1}
	+\sum_{j=0}^{n-1}
	C^{n-1-j}\circ \Ru_{Q(a)}\circ C^j.
	\]
	Taking traces gives
	\[
	\tr(C^{n+1})
	=-\tr(C^{n-1}\Ru_{Q(a)})=0,
	\]
	by \eqref{eq:C-trace}.
	Since
	\[
	\tr C=\tr A+\tr \Ru_a=0,
	\]
	we conclude that
	\[
	\tr(C^n)=0,\qquad n\geq1.
	\]
	Hence $C$ is nilpotent.

	\medskip
	\noindent\textbf{Step 3: Determination of $A_{|\h_0}$ and $\Lu_a$ for $a\in\h_0$.}\\

	Since $A$ commutes with $E$, it leaves $\h_0=\ker E$ invariant. By Lemma \ref{AM}, there exists $h\in\df$ such that, with respect to the decomposition $\h_0=\df\oplus\af$,
	\[
	A_{|\h_0}=\begin{pmatrix}
		A_1 & \Ru_h\\
		0 & A_2
	\end{pmatrix},
	\quad \text{with } [A_1,\rho_1(a)]=0 \text{ for all } a\in\af.
	\]
	From step 2,  $A$ is nilpotent and hence both $A_1$ and $A_2$ are nilpotent.
	
	For any $a,b\in\h_0$, using the first and third relations in \eqref{LeviE}, we obtain
	\[
	[\Lu_a,\Lu_b]=\Lu_{[a,b]_\star}-\big(\langle Ab,a\rangle-\langle Aa,b\rangle\big)E,
	\qquad [E,\Lu_a]=0.
	\]
	It follows that the Lie subalgebra $\{\Lu_a \mid a\in\ker E\}\oplus \mathbb{R}E \subset \so(\h)$ is solvable. By Proposition \ref{solvable}, it is therefore abelian, and hence, for all $a,b\in\h_0$,
	\[
	\Lu_{[a,b]_\star}=\big(\langle Ab,a\rangle-\langle Aa,b\rangle\big)E,
	\qquad [\Lu_a,\Lu_b]=0.
	\]
	Since both $\df$ and $\af$ are abelian, for any $a,b\in\df$ and $c,d\in\af$ we have
	\[
	\langle (A-A^*)a,b\rangle=\langle (A_1-A_1^*)a,b\rangle=0,
	\qquad
	\langle (A-A^*)c,d\rangle=\langle (A_2-A_2^*)c,d\rangle=0.
	\]
	Thus $A_1$ and $A_2$ are symmetric. Being both symmetric and nilpotent, they must vanish, hence $A_1=0$ and $A_2=0$.
	
	Now, for any $a\in\h_0$, the operator $\Lu_a$ commutes with $E$, and therefore preserves both $\h_0$ and $\mathfrak{b}$. Setting $\rho_2(a)=(\Lu_a)_{|\mathfrak{b}}$, we obtain
	\[
	\Lu_a=
	\begin{pmatrix}
		\rho_1(a) & 0 & 0\\
		0 & 0 & 0\\
		0 & 0 & \rho_2(a)
	\end{pmatrix},
	\qquad
	[\rho_i(a),\rho_i(b)]=0,\; i=1,2,\; a,b\in\af,
	\qquad
	\bigcap_{a\in\af}\ker\rho_1(a)=\{0\}.
	\]
	On the other hand, for any $a\in\df$ and $b\in\af$, we have
	\[
	\Lu_{[a,b]_\star}=\big(\langle Ab,a\rangle-\langle Aa,b\rangle\big)E.
	\]
	Since $Ab=b\star h$ and $Aa=0$, it follows that
	\[
	\Lu_{[a,b]_\star}=\langle [a,b]_\star,h\rangle E.
	\]
	Using the relation $[\af,\df]=\df$, we deduce that for any $a\in\df$,
	\[
	\Lu_a=\langle a,h\rangle E.
	\]

	\medskip
	\noindent\textbf{Step 4: $A_{|\mathfrak{b}}=0$ and $\Lu_a=0$ for any $a\in \mathfrak{b}$.}\\
	
	Recall that $\langle\cdot,\cdot\rangle_{\so}$ denotes the canonical scalar product on $\so(\h)$ (see Subsection \ref{subsection21}). The map
	$
	\h_0\to\R,\; a\mapsto \langle \Lu_a,E\rangle_{\so}
	$
	is linear, hence there exists $k\in\h_0$ such that, for any $a\in\h_0$,
	\[
	\langle \Lu_a,E\rangle_{\so}=\langle k,a\rangle.
	\]
	Let $b=E(a)\in \mathfrak{b}$. By the third relation in \eqref{LeviE}, we have
	\begin{equation}\label{eqE_improved}
		\langle \Lu_b,E\rangle_{\so}
		=\langle [E,\Lu_a],E\rangle_{\so}=0.
	\end{equation}
	For any $x,y\in\mathfrak{b}$, using the first relation in \eqref{LeviE}, we obtain
	\[
	\Lu_{[x,y]_\star}-[\Lu_x,\Lu_y]
	=\langle(A-A^*)y,x\rangle\,E.
	\]
	Decompose $[x,y]_\star=[x,y]_1+[x,y]_2$ with $[x,y]_1\in\h_0$ and $[x,y]_2\in\mathfrak{b}$. Using \eqref{eqE_improved}, we compute
	\begin{align*}
		\langle \Lu_{[x,y]_\star},E\rangle_{\so}
		&=\langle \Lu_{[x,y]_1},E\rangle_{\so}
		=\langle [x,y]_1,k\rangle\\
		&=\langle [x,y]_\star,k\rangle
		=\langle x\star y-y\star x,k\rangle\\
		&=\langle (\Ru_k-(\Ru_k)^*)y,x\rangle.
	\end{align*}
	On the other hand,
	\[
	\langle [\Lu_x,\Lu_y],E\rangle_{\so}
	=\langle [E,\Lu_x],\Lu_y\rangle_{\so}
	=\langle \Lu_{E(x)},\Lu_y\rangle_{\so}.
	\]
	Therefore, for all $x,y\in\mathfrak{b}$,
	\begin{equation}\label{Important_improved}
		\langle \Lu_{E(x)},\Lu_y\rangle_{\so}
		=\langle (\Phi-\Phi^*)y,x\rangle\,|E|^2,
		\qquad
		\Phi=\frac{1}{|E|^2}\Ru_k-A.
	\end{equation}
	
	Since $k\in\h_0$, by virtue of the third relation in \eqref{LeviE}  $[E,\Ru_k]=0$, and we have also that $E$ commutes with $A$, hence also with $\Phi$ and $\Phi^*$. So $\Phi$ and $\Phi^*$ leaves invariant $\mathfrak{b}$.  In the sequel, we denote by $E$ and $\Phi$ their restrictions to $\mathfrak{b}$.
	
	Since $E$ is invertible on $\mathfrak{b}$ and preserves $\ker \Phi$, there exists an orthonormal basis $(a_i,b_i)_{i=1}^r$ of $\ker \Phi$ such that
	\[
	E(a_i)=\la_i b_i,\qquad E(b_i)=-\la_i a_i.
	\]
	Using \eqref{Important_improved} with $x,y\in\{a_i,b_i\}$, we deduce
	\[
	\|\Lu_{a_i}\|^2=\|\Lu_{b_i}\|^2=0,
	\]
	hence $\Lu_x=0$ for all $x\in\ker \Phi$. A similar argument shows that $\Lu_x=0$ for all $x\in \ker(\Phi-\Phi^*)$.
	
	Now let $x\in\ker \Phi$. For any $y\in\mathfrak{b}$, \eqref{Important_improved} gives
	\[
	0=\langle \Lu_{E(y)},\Lu_x\rangle_{\so}
	=\big(\langle \Phi x,y\rangle-\langle \Phi y,x\rangle\big)|E|^2,
	\]
	hence $x\in \ker(\Phi-\Phi^*)$. This shows that $\ker \Phi\subset \ker(\Phi-\Phi^*)$, and therefore $\ker \Phi\subset\ker \Phi^*=(\im \Phi)^\perp$. It follows that
	$
	\ker \Phi=(\mathrm{Im}\,\Phi)^\perp$ and hence
	$
	\mathfrak{b}=\ker \Phi\oplus \mathrm{Im}\,\Phi.
	$
	By Step 2, $\Phi$ is nilpotent, so there exists $p\geq 1$ such that $\Phi^p=0$ and $\Phi^{p-1}\neq 0$. If $p\geq 2$, then
	\[
	\mathrm{Im}\,\Phi^{p-1}\subset \ker \Phi\cap \mathrm{Im}\,\Phi=\{0\},
	\]
	a contradiction so $p=1$. Therefore $\Phi_{|\mathfrak{b}}=0$, and consequently $\Lu_x=0$ for all $x\in\mathfrak{b}$.
	
	Finally, since $\Phi=0$, we obtain
	\[
	Ax=\frac{1}{|E|^2}\,x\star k=0,\qquad \forall x\in\mathfrak{b},
	\]
	and thus $A_{|\mathfrak{b}}=0$. Finally, in the splitting $\h=\df\oplus\af\oplus\mathfrak{b}$, $A$ takes the form
	\[ A=\begin{pmatrix}
		0&\Ru_h&0\\
		0&0&0\\
		0&0&0
	\end{pmatrix}. \]

	\medskip
	\noindent\textbf{Step 5: Determination of ${F}$ and ${w}$.}\\

	Write	$w=w_{\df}+w_{\af},$
	$
	w_{\df}\in\df,
	w_{\af}\in\af.$
	We have $A^2=0$ and, from \eqref{LeviE},
	\[
	\begin{cases}
		[{F},\Lu_a]=\Lu_{{F}(a)+A(a)}+\langle a,{w}\rangle E,\\
		[{F},A]=\Ru_{{w}}.
	\end{cases}
	\]
	
	If $a\in\mathfrak{b}$, then ${F}(a)\in\mathfrak{b}$, $A(a)=0$, and $\langle a,{w}\rangle=0$. Hence the first relation is trivially satisfied since $\Lu_a=\Lu_{{F}(a)}=0$.
	
	Consider the restriction to $\h_0=\df\oplus\af$. In this decomposition,
	\[
	{F}_{|\h_0}=\begin{pmatrix}
		F_1 & -X^*\\
		X & F_2
	\end{pmatrix},
	\qquad
	A_{|\h_0}=\begin{pmatrix}
		0 & \Ru_h\\
		0 & 0
	\end{pmatrix}.
	\]
	A direct computation gives
	\[
	[{F},A]_{|\h_0}
	=\begin{pmatrix}
		-\Ru_h X & F_1\Ru_h^\circ-\Ru_h F_2\\
		0 & X\Ru_h
	\end{pmatrix}.
	\]
	Hence
	\[
	F_1\Ru_h-\Ru_h F_2=\Ru_{{w_\df}}.
	\]
	Let $a\in\af$. Then
	\[
	F_1(a\star h)-F_2(a)\star h=a\star {w_\df}.
	\]
	On the other hand, applying the first relation to $h$ yields
	\[
	F_1(a\star h)-a\star F_1(h)=F_2(a)\star h.
	\]
	Comparing the two identities, we obtain
	\[
	\rho_1(a)\big(F_1(h)-{w_\df}\big)=0.
	\]
	Since $\bigcap_{a\in\af}\ker\rho_1(a)=\{0\}$, it follows that
	\[
	{w_\df}=F_1(h).
	\]
	Now let $a\in\af$. Then
	\[
	[{F},\Lu_a]=\Lu_{{F}(a)}+\langle a,w_\af\rangle E.
	\]
	For $b\in\df$, $c\in\af$, and $x\in\mathfrak{b}$, we compute
	\[
	\begin{cases}
		{F}(a\star b)=F_1(a\star b)+X(a\star b)
		= a\star F_1(b)+F_2(a)\star b,\\[0.3em]
		{F}(a\star c)=0=-a\star X^*(c),\\[0.3em]
		F_3(\rho_2(a)x)-\rho_2(a)F_3(x)=\rho_2(F_2(a))+\langle a,w_\af\rangle E(x).
	\end{cases}
	\]
	The second identity implies $X=0$. Consequently,
	\[
	[F_1,\rho_1(a)]=\rho_1(F_2(a)),
	\qquad
	[F_3,\rho_2(a)]=\rho_2(F_2(a))+\langle a,w_\af\rangle E
	\]
	where $F_3={F}_{|\mathfrak{b}}$. The Lie subalgebras $\mathrm{Span}\{ F_1,\rho_1(a), a\in\af\}$ and $\mathrm{Span}\{ E_{|\mathfrak{b}}, F_3,\rho_2(a), a\in\af\}$ of $\so(\df)$ and $\so(\mathfrak{b})$, respectively, are solvable and, by virtue of Proposition \ref{solvable}, they are abelian. It follows that
	\[
	[F_1,\rho_1(a)]=\rho_1(F_2(a))=0,
	[F_3,\rho_2(a)]=0\esp \rho_2(F_2(a))=-\langle a,w_\af\rangle E_{|\mathfrak{b}}.
	\]
	In particular, if $a\in\ker F_2$,
	then
	$
	\langle a,w_{\af}\rangle=0.
	$
	Since $F_2$ is skew-symmetric,
	we obtain
	$
	w_{\af}\in(\ker F_2)^\perp
	=\operatorname{Im}F_2.
	$

	Finally, for $a\in\df$, since $\Lu_a=\langle a,h\rangle E$, the first relation gives
	\[
	\langle F_1(a),h\rangle E+\langle a,{w}\rangle E=0,
	\]
	which is true.

	The five steps combined give that in the splitting $\h=\df\oplus\af\oplus\mathfrak{b}$, for $a\in\af,b\in\df,x\in\mathfrak{b}$,
	\[
\begin{gathered}
A=\begin{pmatrix}0&\Ru_h&0\\0&0&0\\0&0&0\end{pmatrix},\qquad
F=\begin{pmatrix}F_1&0&0\\0&F_2&0\\0&0&F_3\end{pmatrix},\\
\Lu_a=\begin{pmatrix}
\rho_1(a)&0&0\\0&0&0\\0&0&\rho_2(a)
\end{pmatrix},\\
\Lu_b=\langle b,h\rangle E,\quad \Lu_x=0,\quad
w=F_1(h)+F_2(g).
\end{gathered}
\]
and
	\[
	[F_1,\rho_1(a)]=\rho_1(F_2(a))=0,
	[F_3,\rho_2(a)]=0\esp \rho_2(F_2(a))=\langle F_2(a),g\rangle E_{|\mathfrak{b}}.
	\]

	The verification of the Novikov property is given in Section \ref{section9}, Subsection \ref{subsection96}. Finally, combining \eqref{bracket} with the above results yields the structure of flat Lorentzian Lie algebras listed in Table \ref{4}.
\end{proof}


	Let us go further and simplify the structure of Lie Lorentzian algebra obtained in this proposition. In the model of Proposition~\ref{E}, let
	$\h=\df\oplus\af\oplus\mathfrak b$, $h\in\df$, $g\in\af$,
	$w=F_1h+F_2g$, and let $E_1=E|_{\mathfrak b}$.
	The nonzero Lie brackets in the original adapted null frame are
	\begin{equation}\label{eq:g4-old}
		\begin{cases}
			[e,x]=E_1x,\;
			[a,d]=\rho_1(a)d-\langle h,\rho_1(a)d\rangle e,\\
			[a,x]=\rho_2(a)x,\;
			[d,x]=\langle d,h\rangle E_1x,\\
			[f,a]=F_2a+\rho_1(a)h+\langle a,F_2g\rangle e,\\
			[f,d]=F_1d+\langle d,F_1h\rangle e,\;
			[f,x]=F_3x,
		\end{cases}
	\end{equation}
	where $a\in\af$, $d\in\df$, $x\in\mathfrak b$.
	
	Choose the adapted null frame corresponding to $x_0=-h$:
	\[
	e'=e,\qquad d'=d-\langle d,h\rangle e,\qquad
	a'=a,\qquad x'=x,\qquad
	f'=f+h-\tfrac12|h|^2e.
	\]
	Set $\df'=\{d':d\in\df\}$ and $\Phi(d)=d'$.
	The decomposition
	$
	\G_4=\mathbb Re'\oplus\df'\oplus\af\oplus
	\mathfrak b\oplus\mathbb Rf'
	$
	is adapted and orthogonal. Directly from \eqref{eq:g4-old}, we get
	\begin{equation}\label{eq:g4-oldbis}
		\begin{cases}
			[e,x]=E_1x,\;
			[a,d']=\rho_1'(a)d',\\
			[a,x]=\rho_2(a)x,\;
			[d',x]=0,\\
			[f',a]=F_2a+\langle a,F_2g\rangle e,\\
			[f',d']=F_1'd',\;
			[f,x]=F_3'x,
		\end{cases}
	\end{equation}where $\rho_1'(a)=\Phi\rho(a)\Phi^{-1}$ and $F_3'=F_3+\tfrac12|h|^2E_1$.
	
Thus the new data satisfy
	\[
	h'=0,\quad A'=0,\quad w'=F_2g,
	\quad F_3'=F_3+\tfrac12|h|^2E_1.
	\]
	The operators $\rho_1,F_1$ on $\df$ are transported by $\Phi$;
	$E_1,F_2,\rho_2$ are unchanged. Since
	$[F_3,E_1]=[E_1,\rho_2(a)]=0$, the replacement of $F_3$
	by $F_3'$ preserves every parameter condition in the proposition.
	The model therefore admits a presentation with $h=A=0$ and
	$w=F_2g$, without imposing any further normalization on $g$. We get the structure of Lorentzian Lie algebras listed in Table \ref{4}.

\subsection{The non-unimodular case: Solutions of \eqref{curvature} in the case $(\mathrm{NU1})$.} 

\begin{pr}\label{NU1} Let $(\h,\star,\prs)$ be a Euclidean algebra together with the data $(E,F,A,u,v,w,\al,\la)$ satisfying \eqref{curvature} such that the extension $\G$ is non-unimodular and $\tr(A)+\la\not=0$. Then:\begin{enumerate}
		\item[$(i)$] 
		$E=0$, $u=v=0$, $\al=0$, $\la\not=0$,\item[$(ii)$]   $\h$ splits orthogonally $\h=\df_1\oplus\df_2\oplus\af_1\oplus\af_2$,
		\item[$(iii)$] for $i=1,2$,  there exists $\rho_i:\af_1\too\so(\df_i)$ with $\bigcap_{a\in\af_1}\ker\rho_i(a)=\{0\}$, $h=h_1+h_2\in \df_1\oplus\df_2$,  such that:
		\begin{enumerate}
			
			\item The product $\star$ is given by
			\[ \Lu_{d_1}=\Lu_{d_2}=\Lu_{a_2}= (\Lu_{a_1})_{|\af_1\oplus\af_2}=0,\;
			(\Lu_{a_1})_{|\df_i}=\rho_i(a_1),\; d_i\in\df_i, a_i\in\af_i,\; i=1,2. \]
			\item  $A,F,w$ are given by, for $i=1,2$ and  $d_i\in\df_i$, $a_i\in\af_i$,
			\[ \begin{cases}
				A_{|\df_1}=0,\; A(d_2)=U(d_2)+\la d_2,\; A(a_1)=\rho_1(a_1)h_1+\rho_2(a_1)h_2,\; A(a_2)=V(a_2)+\la a_2,\\
				F_{|\df_i}=F_i,\; F_{|\af_i}=G_i,\; i=1,2,\\
				w=w_1+w_2,\; w_1=\lambda h_1-U(h_2)+F_1(h_1)+F_2(h_2),\; w_2\in\af_1\oplus\af_2,
			\end{cases} \]
			where 
			$F_i\in\so(\df_i),\;\;G_i\in\so(\af_i),\;i=1,2,\;
			U:\df_2\too\df_1,\;V:\af_2\too\af_1,$
			and
			\[ \begin{cases}\rho_1(G_1(a_1))=\rho_2(G_1(a_1))=\rho_1(V(a_2))=0,\; UF_2=F_1U,\; VG_2=G_1V,\\
				[F_1,\rho_1(a_1)]=[F_2,\rho_2(a_1)]=0,\; U\rho_2(a_1)=\rho_1(a_1)U,\; a_1\in\af_1,\; a_2\in\af_2.
			\end{cases} \]
		\end{enumerate}
	\end{enumerate}
	The structure of flat Lorentzian Lie algebra on $\G=\R e\oplus\h\oplus\R f$ after the simplification below is given in Table \ref{5}.
	
\end{pr}

\begin{proof}
	We have already shown that $\alpha=0$, $\tr(A)+\lambda\neq 0$, and
	\[
	\mathbf{h}=(\lambda+\tr(A))e.
	\]
	It follows, by virtue of Proposition \ref{pr1}, that $\Ru_{\mathbf h}$ is symmetric and that $H=\mathbb Re$ is a two-sided ideal of $H^\perp$. By \eqref{Levi}, we obtain
	\[
	E=0,\qquad v=u=0.
	\]
	Hence \eqref{curvature} reduces to
	\[
	\begin{cases}
		\ass_\star(a,b,c)-\ass_\star(b,a,c)=0,\\[1mm]
		A([a,b])+b\star A a-a\star A b=0,\\[1mm]
		[F,\Lu_a]=\Lu_{F(a)+A(a)},\\[1mm]
		A^2a-\lambda A a+a\star w=[F,A]a,\\[1mm]
		\tr(\Ru_a)=0,a,b,c\in\mathfrak h.
	\end{cases}
	\]
	We are therefore in the same situation as in Proposition~\ref{u0}, except that here $\tr(A)+\lambda\not=0.$
	Consequently, $\mathfrak h=\mathfrak d\oplus\mathfrak a$ is a flat Euclidean Lie algebra whose Lie bracket is determined by a representation
	\[
	\rho:\mathfrak a\to\so(\mathfrak d)
	\]
	satisfying $
	\bigcap_{a\in\mathfrak a}\ker \rho(a)=\{0\},$
	and such that
	\[
	A=
	\begin{pmatrix}
		A_1 & \Ru_h\\
		0 & A_2
	\end{pmatrix},
	\qquad
	F=
	\begin{pmatrix}
		H & 0\\
		0 & G
	\end{pmatrix},
	\]
	with
	\[
	\begin{cases}
		[H,\rho(a)]=[A_1,\rho(a)]=0,\\[1mm]
		[H,A_1]=A_1^2-\lambda A_1,\\[1mm]
		[G,A_2]=A_2^2-\lambda A_2,\\[1mm]
		H\Ru_h-\Ru_hG=A_1\Ru_h+\Ru_hA_2-\lambda \Ru_h+\Ru_{w_1},\\
		[F,\Lu_a]=\Lu_{F(a)+A(a)}=0.
	\end{cases}
	\]
	If $\lambda=0$, then both $A_1$ and $A_2$ are nilpotent, hence $\tr(A)=0$, contradicting $\tr(A)+\lambda\neq 0$. Therefore,
	$
	\lambda\neq 0.
	$
	
	By Lemma~\ref{LF}, there exist orthogonal decompositions
	\[
	\mathfrak d=\mathfrak d_1\oplus\mathfrak d_2,
	\qquad
	\mathfrak a=\mathfrak a_1\oplus\mathfrak a_2,
	\]
	such that
	\[
	A_1=
	\begin{pmatrix}
		0 & U\\
		0 & \lambda \mathrm{Id}_{\mathfrak d_2}
	\end{pmatrix},
	\qquad
	A_2=
	\begin{pmatrix}
		0 & V\\
		0 & \lambda \mathrm{Id}_{\mathfrak a_2}
	\end{pmatrix},
	\]
	and
	\[
	H=
	\begin{pmatrix}
		F_1 & 0\\
		0 & F_2
	\end{pmatrix},
	\qquad
	G=
	\begin{pmatrix}
		G_1 & 0\\
		0 & G_2
	\end{pmatrix},
	\]
	with
	\[
	UF_2=F_1U,
	\qquad
	VG_2=G_1V.
	\]
	The relation $[A_1,\rho(a)]=0$ implies that
	\[
	\rho(a)=
	\begin{pmatrix}
		\rho_1(a) & 0\\
		0 & \rho_2(a)
	\end{pmatrix},
	\qquad
	U\rho_2(a)=\rho_1(a)U.
	\]
	Moreover, the relation $[H,\rho(a)]=0$ is equivalent to
	\[
	[F_i,\rho_i(a)]=0,
	\qquad i=1,2.
	\]
	Since $[H,\rho(a)]=[A_1,\rho(a)]=0$, the identity
	\[
	H(a\star h)-G(a)\star h
	=
	A_1(a\star h)+A_2(a)\star h-\la a\star h+a\star w_1
	\]
	becomes
	\[
	\rho(a)\bigl(H(h)-A_1(h)+\la h-w_1\bigr)
	=
	\rho\bigl(G(a)+A_2(a)\bigr).
	\]
	On the other hand, we have already we have $\Lu_{F(a)+A(a)}=0,$
	which is equivalent to
	\[
	\rho\bigl(G(a)+A_2(a)\bigr)=0.
	\]
	Hence
	\[
	w_1=H(h)-A_1(h)+\la h
	=
	\lambda h_1-U(h_2)+F_1(h_1)+F_2(h_2).
	\]
	Finally, the relation
	\[
	\rho\bigl(G(a)+A_2(a)\bigr)=0
	\]
	is equivalent to
	\[
	\rho(G_1(a_1))=0,
	\qquad \rho(V(a_2))=0\esp 
	\rho\bigl(G_2(a_2)+\lambda a_2\bigr)=0,
	\]
	for all $a_1\in\mathfrak a_1$ and $a_2\in\mathfrak a_2$.
	
	Since $G_2$ is skew-symmetric, the endomorphism
	$
	G_2+\lambda \mathrm{Id}_{\mathfrak a_2}
	$
	is invertible, and therefore
	$
	\im(G_2+\lambda \mathrm{Id}_{\mathfrak a_2})=\mathfrak a_2.
	$
	It follows that $\rho(\mathfrak a_2)=0.$
\end{proof}

Let us simplify the structure of Lorentzian Lie algebra obtained in this proposition.
	In the notation of Proposition~\ref{NU1}, let
	$h=h_1+h_2\in\df_1\oplus\df_2$ and
	\[
	w_1=\lambda h_1+F_1h_1+F_2h_2-Uh_2,
	\qquad w_2\in\af_1\oplus\af_2.
	\]
	Write $D=\df_1\oplus\df_2$ and
	$T=\af_1\oplus\af_2$. In the original frame, the nonzero brackets
	are, for $a_i\in\af_i$ and $d_i\in\df_i$,
	\begin{equation}\label{eq:g5-old}
		\begin{cases}
			[f,e]=\lambda e,\\
			[a_1,d_i]=\rho_i(a_1)d_i-\langle h,\rho_i(a_1)d_i\rangle e,\quad i=1,2,\\
			[a_1,a_2]=-\langle a_1,Va_2\rangle e,\\
			[d_1,d_2]=-\langle d_1,Ud_2\rangle e,\\
			[f,d_1]=F_1d_1+\langle d_1,w_1\rangle e,\\
			[f,d_2]=(F_2+\lambda I)d_2+Ud_2+\langle d_2,w_1\rangle e,\\
			[f,a_1]=G_1a_1+\rho_1(a_1)h_1+\rho_2(a_1)h_2
			+\langle a_1,w_2\rangle e,\\
			[f,a_2]=(G_2+\lambda I)a_2+Va_2+\langle a_2,w_2\rangle e.
		\end{cases}
	\end{equation}
	Perform the adapted null-frame change with $x=-h$:
	\[
	e'=e,\quad d_i'=d_i-\langle d_i,h_i\rangle e,
	\quad a_i'=a_i,\quad
	f'=f+h-\tfrac12|h|^2e.
	\]
	Writing $\df_i'=\{d_i':d_i\in\df_i\}$ and
	$\Phi_i(d_i)=d_i'$, this gives the adapted orthogonal decomposition
	\[
	\G_5=\mathbb Re'\oplus\df_1'\oplus\df_2'
	\oplus\af_1\oplus\af_2\oplus\mathbb Rf'.
	\]
	Transport $F_i,\rho_i$ and $U:\df_2\to\df_1$ through the
	isometries $\Phi_i$; $G_i,V,\lambda,w_2$ are unchanged.
	For example, the terms involving $f'$ become
	\[
	\begin{aligned}
		[f',d_1']
		&=F_1d_1+\langle d_1,w_1+Uh_2-\lambda h_1\rangle e
		=\Phi_1(F_1d_1),\\
		[f',d_2']
		&=(F_2+\lambda I)d_2+Ud_2
		+\langle d_2,F_2h_2-\lambda h_2-U^*h_1\rangle e\\
		&=\Phi_2((F_2+\lambda I)d_2)+\Phi_1(Ud_2),\\
		[f',a_1]
		&=G_1a_1+\langle a_1,w_2\rangle e,\\
		[f',a_2]
		&=(G_2+\lambda I)a_2+Va_2+\langle a_2,w_2\rangle e.
	\end{aligned}
	\]
	Also $[a_1,d_i']=\Phi_i(\rho_i(a_1)d_i)$, whereas
	$[a_1,a_2]$ and $[d_1',d_2']$ retain their central expressions.
	Thus the new data satisfy
	\[
	h_1'=h_2'=0,\qquad w_1'=0,
	\]
	while the operators $F_i,G_i,U,V,\rho_i$ satisfy precisely the
	original conditions (after the indicated isometric identifications).
	No further change is made to $w_2$. Thus we get the model listed in Table 5.

\subsection{The non-unimodular case: Solutions of \eqref{curvature} in the case ($\mathrm{NU2}$).}

\begin{pr}\label{NU2} 
	
	Let $(\h,\star,\prs)$ be a Euclidean
	algebra and let
	$(E,F,A,u,v,w,\alpha,\lambda)$
	be data satisfying \eqref{curvature}.
	Suppose that the corresponding
	flat Lorentzian Lie algebra $\G$
	is non-unimodular,  $\alpha\neq0$, $v=\mathbf{h}_0$ and 
	$\tr(A)+\lambda=0$. Then: 
	\begin{enumerate}
		\item[$(i)$]
		$u=v=w=0$, $A=F=0$, and $\lambda=0$;
		
		\item[$(ii)$]
		The Euclidean Lie algebra $\h$ admits
		the orthogonal Milnor decomposition
		\[
		\h=\df\oplus\af,
		\]
		with Levi-Civita product given by
		\eqref{Levi-Milnor} and associated
		representation
		\[
		\rho:\af\longrightarrow\so(\df),
		\qquad
		\bigcap_{a\in\af}\ker\rho(a)=\{0\}.
		\]
		Moreover, relative to this
		decomposition,
		\[
		E=
		\begin{pmatrix}
			E_1&0\\
			0&E_2
		\end{pmatrix},
		\qquad
		E_1\in\so(\df),\quad
		E_2\in\so(\af),
		\]
		where, for every $a\in\af$,
		\[
		[E_1,\rho(a)]=0,
		\qquad
		\rho(E_2(a))=0.
		\]
	\end{enumerate}
	The structure of flat Lorentzian Lie algebra on $\G=\R e\oplus\h\oplus\R f$ is given in Table \ref{6}.
	
\end{pr}

\begin{proof}
	By Proposition~\ref{pr1},
	the operator $\mathrm{R}_f$
	is symmetric and $\mathbb Rf$
	is a two-sided ideal of $f^\perp$.
	
	Using the expression \eqref{Levi}
	of the Levi-Civita product,
	we deduce
	\[
	A=F=w=0,\qquad
	\lambda=0,\qquad u=v.
	\]
	
	The first equation of
	\eqref{curvature} now shows that
	$(\h,\star,\prs)$ is a flat
	Euclidean Lie algebra, and hence
	its modular vector vanishes:
	\[
	\mathbf{h}_0=0.
	\]
	Since the normalized frame
	satisfies $v=\mathbf{h}_0$,
	we obtain
	$
	u=v=0.
	$
	Consequently,
	\[
	u=v=w=0,\qquad
	A=F=0,\qquad\lambda=0.
	\]
	
	The remaining equations of
	\eqref{curvature} reduce to   	
	\[
	\begin{cases}
		\ass_\star(a,b,c)-\ass_\star(b,a,c)=0,\\[1mm]
		[E,\Lu_a]=\Lu_{E(a)},\\[1mm]
		\tr(\Ru_a)=0.
	\end{cases}
	\]
	Thus $\mathfrak h$ is a flat Euclidean Lie algebra, and Theorem~\ref{milnor} applies. Consequently,
	\[
	\mathfrak h=\mathfrak d\oplus\mathfrak a,
	\]
	where the product $\star$ is determined by a representation of an abelian Lie algebra
	$
	\rho:\mathfrak a\to\so(\mathfrak d)
	$
	satisfying
	$
	\bigcap_{a\in\mathfrak a}\ker\rho(a)=\{0\}.
	$
	By virtue of Proposition \ref{solvable},
	$
	[E,\Lu_a]=\Lu_{E(a)}=0.
	$
	In particular, $E$ preserves both $\mathfrak a$ and $\mathfrak d$. Therefore, with respect to the decomposition
	$
	\mathfrak h=\mathfrak d\oplus\mathfrak a,
	$
	the endomorphism $E$ takes the block diagonal form
	\[
	E=
	\begin{pmatrix}
		E_1 & 0\\
		0 & E_2
	\end{pmatrix},
	\]
	and satisfies
	$
	[E_1,\rho(a)]=0
	$,
	$\rho(E_2(a))=0.
	$
\end{proof}

\section{The Kundt Property of Flat Lorentzian Lie Groups}\label{sectionKundt}

\begin{pr}\label{allKundt}
	Let $(G,g)$ be a connected flat Lorentzian
	Lie group of dimension at least two.
	Then $(G,g)$ is of Kundt type.
	
	More precisely, there exists a
	nowhere-vanishing null vector field
	$V$ and a $1$-form $\eta$ such that, for any vector field $X\in\mathfrak X(G)$,
	\[
	\nabla_XV=\eta(X)V,
	\esp 
	\nabla_VV=0.
	\]
	
	Moreover, if $G$ admits a timelike
	parallel left-invariant vector field,
	then $V$ can be chosen parallel.
\end{pr}

\begin{proof} Denote by $\G=T_eG$ the Lie algebra of $G$. For any vector $u\in\G$ and $\la\in\G^*$, we denote by $u^\ell$ and $\la^\ell$ the corresponding left invariant vector field and left invariant 1-form, respectively, on $G$.
	By Proposition~\ref{ee}, we distinguish
	two cases.
	
	\smallskip
	\noindent
	\textit{Case 1: There exists a nonzero
		null vector $e\in\G$ such that
		$u\cdot e=\lambda(u)e$ for every
		$u\in\G$.}
	Then
	\[
	\nabla_Xe^\ell=\lambda^\ell(X)e^\ell,
	\qquad X\in\mathfrak X(G).
	\]
	Consider the positive modular
	character
	\[
	\Delta:G\longrightarrow\mathbb R_{>0},
	\qquad
	\Delta(g)=\det(\operatorname{Ad}_g).
	\]
	Its differential satisfies
	\[
	u^\ell(\Delta)
	=\tr(\operatorname{ad}_u)\Delta,
	\qquad u\in\G.
	\]
	Since $\Lu_e$ is skew-symmetric,
	$\tr \Lu_e=0$. Moreover,
	$
	\Ru_e(u)=u\cdot e=\lambda(u)e,
	$
	so that
	$
	\tr \Ru_e=\lambda(e).
	$
	Consequently,
	\[
	\tr(\operatorname{ad}_e)
	=\tr(\Lu_e-\Ru_e)
	=-\lambda(e).
	\]
	Define
	$
	V=\Delta e^\ell.
	$
	This is a globally defined,
	nowhere-vanishing null vector field.
	For any $X\in\mathfrak X(G)$,
	we have
	\[
	\begin{aligned}
		\nabla_XV
		&=X(\Delta)e^\ell
		+\Delta\nabla_Xe^\ell
		=\bigl(d\log\Delta(X)
		+\lambda^\ell(X)\bigr)V.
	\end{aligned}
	\]
	Thus $V$ is recurrent.
	
	Furthermore,
	\[
	\begin{aligned}
		\nabla_VV
		&=\Delta^2
		\bigl(\tr(\operatorname{ad}_e)
		+\lambda(e)\bigr)e^\ell
		=0.
	\end{aligned}
	\]
	Therefore $(G,g)$ is of Kundt type.

	\smallskip
	\noindent
	\textit{Case 2: There exists a timelike
		parallel left-invariant vector field.}
	
	By Proposition~\ref{G1}, the Lie
	algebra $\G$ admits an orthogonal
	decomposition
	$
	\G=\mathfrak t\oplus\mathfrak n,
	$
	where
	$
	\mathfrak t=\mathbb Re\oplus\af\oplus Z_1$ and
	$
	\mathfrak n=\df\oplus Z_0.
	$
	Here $\mathfrak t$ is abelian and
	Lorentzian, $\mathfrak n$ is an
	abelian Euclidean ideal, and the
	action of $\mathfrak t$ on
	$\mathfrak n$ is given by a
	representation
	$
	\rho:\mathfrak t\longrightarrow
	\so(\mathfrak n).
	$
	Moreover, the Levi-Civita product
	satisfies
	\[
	\Lu_t=\ad_t,\qquad \Lu_z=0,
	\qquad t\in\mathfrak t,\ z\in\mathfrak n.
	\]
	Thus every element of $\mathfrak t$
	defines a parallel left-invariant
	vector field.
	
	If $\dim\mathfrak t\geq2$, we can
	choose a nonzero null vector
	$V\in\mathfrak t$. Its left-invariant
	extension is then a parallel null
	vector field.

	Suppose now that $\dim\mathfrak t=1$.
	Then $\mathfrak t=\mathbb Re$ and
	$\mathfrak n=e^\perp$.
	
	Choose $z\in\mathfrak n$ such that
	\[
	\langle z,z\rangle=-\langle e,e\rangle,
	\]
	and denote by $z^r$ the corresponding
	right-invariant vector field.
	
	Since $\mathfrak n$ is an ideal,
	$z^r$ is everywhere orthogonal to
	$e^\ell$. Moreover, $\Lu_x=0$ for
	every $x\in\mathfrak n$, and hence
	\[
	\nabla_{z^r}u^\ell=0,
	\qquad u\in\G.
	\]
	As left-invariant and right-invariant
	vector fields commute, the
	torsion-free property gives
	\[
	\nabla_{u^\ell}z^r
	=[u^\ell,z^r]
	+\nabla_{z^r}u^\ell=0.
	\]
	Thus $z^r$ is parallel.
	
	Consequently,
	\[
	V=e^\ell+z^r
	\]
	is a globally defined parallel
	vector field. Since $e^\ell$ and
	$z^r$ are orthogonal and parallel,
	and $V$ is null at the identity,
	it is null everywhere.
	
	Therefore $(G,g)$ is of Kundt type.
		This completes the proof.
\end{proof}

\section{Conclusion}\label{section6}

Our results provide a complete structural
description of flat Lorentzian Lie algebras
in terms of six explicit families,
extending Milnor's classical description
of flat Riemannian Lie algebras.
They also show that every flat Lorentzian
Lie group is of Kundt type. A notable and practically useful feature of these six models is that all conditions imposed on their parameters — including commutativity relations between endomorphisms and compatibility conditions for linear maps — are explicit algebraic constraints that can be analyzed systematically. This makes it easy to construct explicit examples in arbitrary dimension and to systematically explore the geometry of flat Lorentzian Lie groups well beyond the low-dimensional setting.

It is worth emphasizing that the six
families are not necessarily disjoint,
and their parameters are not canonical
invariants. Indeed, the two alternatives
underlying our construction may coexist:
a flat Lorentzian Lie algebra admitting
a timelike parallel vector may also
admit a parallel null vector.
Moreover, different families, or
different choices of parameters within
the same family, may determine
isometrically isomorphic metric Lie
algebras.

Thus, our results should be understood
as an exhaustive structural description
rather than a classification up to
isometric Lie algebra isomorphism.
Determining the precise equivalence
relations between the parameters of
the six families, and hence obtaining
a classification by pairwise
non-isometric models, remains an
interesting problem for future research. 

We expect these results to serve as a foundation for further investigations, particularly the study of flat pseudo-Riemannian metrics of higher index.

\section{The tables}\label{sectionTables}

{\renewcommand*{\arraystretch}{2}
	\begin{center}
		\begin{tabular}{|l|c|}
			\hline
			Name&	$\G_1=\R e\oplus \af\oplus \df\oplus Z_0\oplus Z_1$, 	\\
			\hline
			N.V.B&		$[e,e_i]=\la_i f_i,\;[e,f_i]=-\la_i e_i,\;$, $[e,z_i]=\mu_i\bar{z_i},\;[e,\bar{z_i}]=-\mu_i z_i$, $[a,e_i]=\langle a,u_i\rangle f_i,\;$\\
			&$[a,f_i]=-\langle a,u_i\rangle e_i,\; a\in\af.$\\
			\hline 
			Parameters& $0\leq\la_1\leq\ldots\leq\la_r$,  $0\leq\mu_1\leq\ldots\leq\mu_s$, $u_1,\ldots, u_r\in\af\setminus\{0\}$.\\
			\hline
			Conditions&	$\af=\mathrm{span}\{u_1,\ldots,u_r\}$,\\
			\hline	
			Metric	&	$\langle e,e\rangle=-1$, $\af,\df, Z_0,Z_1$ are Euclidean vector spaces that are pairwise orthogonal,\\ &$\mathrm{span}\left\{(e_i,f_i)_{i=1}^r\right\}$ is an orthonormal basis of $\df$ and $\{(z_i,\bar{z_i})_{i=1}^s\}$ is an orthonormal basis of $Z_0$.\\
			\hline
			Remark&	$G_1$ endowed with the Levi-Civita product is a Novikov algebra.\\
			&$Z_1$ is in the center of $\G_1$.\\
			\hline
		\end{tabular}
	\end{center}
	\captionof{table}{ Lorentzian flat Lie algebras from Proposition \ref{G1} \label{1}}
}

\begin{center}
	\renewcommand{\arraystretch}{1.5}
	\begin{tabular}{|p{0.16\linewidth}|p{0.75\linewidth}|}
		\hline
		Name & $\G_2=\mathbb Re\oplus\df\oplus\A\oplus\mathbb Ru\oplus\mathbb Rf$,
		$\af=\A\oplus\mathbb Ru$ \\
		\hline
		N.V.B. &
		$\displaystyle
		\begin{aligned}
			[f,e]&=\lambda e, & [u,e]&=e,\\
			[a,d]&=\rho(a)d, & [u,d]&=\rho(u)d,\\
			[u,a]&=\langle n,a\rangle e, & [f,d]&=F_0d,\\
			[f,a]&=\lambda\langle a,n\rangle e, &
			[f,u]&=f+n-\lambda u-\lambda^2\|u\|^2e.
		\end{aligned}$
		\quad($a\in\A$, $d\in\df$) \\
		\hline
		Parameters & $\lambda\in\mathbb R$, $n\in\A$,
		$\rho:\af\longrightarrow\so(\df)$,
		$F_0=-\rho(n)+\lambda\rho(u)$. \\
		\hline
		Conditions & $[\rho(a),\rho(b)]=0$ for all $a,b\in\af$,
		$\displaystyle\bigcap_{a\in\af}\ker\rho(a)=\{0\}$;
		$\dim\df=2r$ (possibly $r=0$). \\
		\hline
		Metric & $\df\oplus\A\oplus\mathbb Ru$ is Euclidean and orthogonal to
		$\operatorname{span}\{e,f\}$;
		$\langle e,e\rangle=\langle f,f\rangle=0$,
		$\langle e,f\rangle=1$. \\
		\hline
		Remarks & The Levi-Civita product is Novikov.
		 \\
		\hline
	\end{tabular}
	\captionof{table}{Flat Lorentzian Lie algebras from Proposition~\ref{u} after simplification.\label{2}}
\end{center}

\begin{center}
	\begin{minipage}{\linewidth}
		\renewcommand{\arraystretch}{1.5}
		\begin{tabular}{|p{0.17\linewidth}|p{0.72\linewidth}|}
			\hline
			Name & $\G_3=\mathbb Re\oplus\df\oplus\af\oplus\mathbb Rf$.\\ \hline
			N.V.B. & $\displaystyle
			\begin{aligned}
				[d_1,d_2]&=\langle(A_1-A_1^*)d_1,d_2\rangle e,\\
				[a_1,a_2]&=\langle(A_2-A_2^*)a_1,a_2\rangle e,\\
				[a,d]&=\rho(a)d,\\
				[f,d]&=(F_1+A_1)d,\\
				[f,a]&=(F_2+A_2)a+\langle a,w_2\rangle e.
			\end{aligned}$
			\quad ($d,d_i\in\df$, $a,a_i\in\af$).\\ \hline
			Parameters & $A_1\in\operatorname{End}(\df)$,
			$A_2\in\operatorname{End}(\af)$, $F_1\in\so(\df)$,
			$F_2\in\so(\af)$, $w_2\in\af$, and
			$\rho:\af\to\so(\df)$.\\ \hline
			Conditions & For all $a,b\in\af$,
			$[\rho(a),\rho(b)]=0$,
			$[F_1,\rho(a)]=[A_1,\rho(a)]=0$,
			$\rho((F_2+A_2)a)=0$;
			$[F_i,A_i]=A_i^2$ ($i=1,2$), and
			$\bigcap_{a\in\af}\ker\rho(a)=\{0\}$.\\ \hline
			Metric & $\df\oplus\af$ is Euclidean, $\df\perp\af$,
			$e,f$ are null, $\langle e,f\rangle=1$, and
			$\operatorname{span}\{e,f\}\perp(\df\oplus\af)$.\\ \hline
			Remarks & The Levi-Civita product is Novikov if and only if
			$A_1=0$, $A_2^2=0$,
			$\operatorname{Im}F_2,\operatorname{Im}A_2\subseteq\ker\rho$,
			$A_2F_2=F_2A_2=0$, and $A_2w_2=0$.
			The family is a classical double extension.\\ \hline
		\end{tabular}
		\captionof{table}{Flat Lorentzian Lie algebras from Proposition~\ref{u0} after simplification.\label{3}}
	\end{minipage}
\end{center}

\begin{center}
	\begin{minipage}{\linewidth}
		\renewcommand{\arraystretch}{1.5}
		\begin{tabular}{|p{0.17\linewidth}|p{0.72\linewidth}|}
			\hline
			Name & $\G_4=\mathbb Re\oplus\df\oplus\af\oplus\mathfrak b\oplus\mathbb Rf$.\\ \hline
			N.V.B. & $\displaystyle
			\begin{aligned}
				[e,x]&=E_1x, & [a,d]&=\rho_1(a)d,\\
				[a,x]&=\rho_2(a)x, & [f,d]&=F_1d,\\
				[f,a]&=F_2a+\langle a,F_2g\rangle e,
				& [f,x]&=F_3x.
			\end{aligned}$
			\quad ($d\in\df$, $a\in\af$, $x\in\mathfrak b$).\\ \hline
			Parameters & $g\in\af$,
			$\rho_1:\af\to\so(\df)$,
			$\rho_2:\af\to\so(\mathfrak b)$,
			$F_1\in\so(\df)$, $F_2\in\so(\af)$,
			$F_3,E_1\in\so(\mathfrak b)$, $\det E_1\ne0$.\\ \hline
			Conditions & For every $a,b\in\af$, $[\rho_i(a),\rho_i(b)]=0$
			($i=1,2$), $[F_1,\rho_1(a)]=0$,
			$[F_3,\rho_2(a)]=[E_1,\rho_2(a)]=[F_3,E_1]=0$,
			$\rho_1(F_2a)=0$,
			$\rho_2(F_2a)=\langle F_2a,g\rangle E_1$,
			and $\bigcap_{a\in\af}\ker\rho_1(a)=\{0\}$.\\ \hline
			Metric & $\df\oplus\af\oplus\mathfrak b$ is Euclidean with
			pairwise orthogonal summands; $e,f$ are null,
			$\langle e,f\rangle=1$, and
			$\operatorname{span}\{e,f\}$ is orthogonal to the Euclidean summands.\\ \hline
			Remarks & $\dim\mathfrak b$ is positive and even.
			The Levi-Civita product is Novikov. \\ \hline
		\end{tabular}
		\captionof{table}{Flat Lorentzian Lie algebras from Proposition~\ref{E} after simplification.\label{4}}
	\end{minipage}
\end{center}

\begin{center}
	\begin{minipage}{\linewidth}
		\renewcommand{\arraystretch}{1.5}
		\begin{tabular}{|p{0.17\linewidth}|p{0.72\linewidth}|}
			\hline
			Name & $\G_5=\mathbb Re\oplus\df_1\oplus\df_2\oplus
			\af_1\oplus\af_2\oplus\mathbb Rf$.\\ \hline
			N.V.B. & $\displaystyle
			\begin{aligned}
				[f,e]&=\lambda e,\\
				[a_1,d_i]&=\rho_i(a_1)d_i\quad(i=1,2),\\
				[a_1,a_2]&=-\langle a_1,Va_2\rangle e,\\
				[d_1,d_2]&=-\langle d_1,Ud_2\rangle e,\\
				[f,d_1]&=F_1d_1,\\
				[f,d_2]&=(F_2+\lambda I)d_2+Ud_2,\\
				[f,a_1]&=G_1a_1+\langle a_1,w_2\rangle e,\\
				[f,a_2]&=(G_2+\lambda I)a_2+Va_2+\langle a_2,w_2\rangle e.
			\end{aligned}$
			\quad($a_i\in\af_i$, $d_i\in\df_i$).\\ \hline
			Parameters & $\lambda\ne0$, $w_2\in\af_1\oplus\af_2$,
			$F_i\in\so(\df_i)$, $G_i\in\so(\af_i)$,
			$\rho_i:\af_1\to\so(\df_i)$ ($i=1,2$),
			$U:\df_2\to\df_1$, $V:\af_2\to\af_1$.\\ \hline
			Conditions & For $a,b\in\af_1$, $c\in\af_2$,
			$[\rho_i(a),\rho_i(b)]=[F_i,\rho_i(a)]=0$,
			$\rho_i(G_1a)=\rho_i(Vc)=0$ ($i=1,2$),
			$UF_2=F_1U$, $VG_2=G_1V$,
			$U\rho_2(a)=\rho_1(a)U$,
			and $\bigcap_{a\in\af_1}\ker\rho_i(a)=\{0\}$ ($i=1,2$).\\ \hline
			Metric & $\df_1,\df_2,\af_1,\af_2$ are mutually orthogonal
			Euclidean subspaces, orthogonal to the null plane
			$\operatorname{span}\{e,f\}$; $\langle e,f\rangle=1$ and
			$\langle e,e\rangle=\langle f,f\rangle=0$.\\ \hline
			Remarks & $\G_5$ is non-unimodular.
			\\ \hline
		\end{tabular}
		\captionof{table}{Flat Lorentzian Lie algebras from Proposition~\ref{NU1} after simplification. \label{5}}
	\end{minipage}
\end{center}

{\renewcommand*{\arraystretch}{2}
	\begin{center}
		\begin{tabular}{|l|c|}
			\hline
			Name& $\G_6=\R e\oplus\df\oplus\af\oplus\R f$	 	\\
			\hline
			N.V.B& $[f,e]=\al f,\; [e,b]=E_1(b),$	 $[e,a]=E_2(a),\; [a,b]=\rho(a)(b)$, $a\in\af,b\in\df$.	\\
			\hline 
			Parameters&$\al\in\R$, $\rho:\af\too\so(\df)$, $E_1\in\so(\df),\; E_2\in\so(\af).$   \\
			\hline
			Conditions&$\al\not=0$, $\forall a,b\in\af$, $[\rho(a),\rho(b)]=[E_1,\rho(a)]=\rho(E_2(a))=0$, $\bigcap_{a\in\af}\ker\rho(a)=\{0\}$, $\dim\df=2r$.\\
			\hline
			Metric& $(\df\oplus\af,\prs)$ is an Euclidean vector space, $\af,\df$ are pairwise orthogonal,\\& $\h^\perp=\operatorname{span}\{e,f\}$,\;$\langle e,e\rangle=\langle f,f\rangle=0$ and $\langle e,f\rangle=1$.	 			\\
			\hline
			Remarks& $\G_6$ is a non-unimodular Lie algebra.\\	
			\hline
			
		\end{tabular}
	\end{center}
	\captionof{table}{ Lorentzian flat Lie algebras from Proposition \ref{NU2}\label{6}}
}

\section{Flat Lorentzian Lie algebras in dimensions at most four}
\label{Section7}

The classification of flat Lorentzian Lie algebras has previously been addressed in several special cases. Guediri~\cite{Guediri2003} studied the two-step nilpotent case, while Aubert and Medina~\cite{Aub-Med} showed that nilpotent flat Lorentzian Lie algebras arise from a double extension process. Ait Ben Haddou, Boucetta and Lebzioui~\cite{ABL} extended this approach to Lie algebras with degenerate center, determining all such algebras up to dimension six. Boucetta and Lebzioui~\cite{BH} classified the non-unimodular case up to dimension four. More recently, Bajo, Benayadi and Lebzioui~\cite{bajo} obtained a complete classification of nilpotent flat Lorentzian Lie algebras, including the equivalence classes of their metrics.

Here we adopt a unified approach to obtain all non-abelian Lie algebras of dimensions three and four admitting a flat Lorentzian metric directly from our six general families. We first specialize Tables~\ref{1}--\ref{6} to these dimensions, retaining the parameters and metrics. We then give explicit changes of basis that identify the resulting brackets with the standard low-dimensional models and express the metrics in the corresponding dual bases. The elementary verifications are left to the reader, and the results are summarized in tables. We do not attempt a classification of the metrics up to isometric isomorphism.

In dimension two, besides the abelian Lie algebra $2A_1$, the only possible non-abelian type is $A_2$. In a basis $(f_1,f_2)$, it is given by $[f_2,f_1]=f_1$ and admits the flat Lorentzian metric $g=f^1\odot f^2$. This model is obtained from either $\mathfrak G_5$ or $\mathfrak G_6$; the two descriptions are related by interchanging the null directions. We therefore concentrate below on dimensions three and four.

\subsection{Three-dimensional models extracted from Tables~\ref{1}--\ref{6}}

We first list the three-dimensional metric Lie algebras obtained directly
from our six families, without reducing the brackets to the usual
low-dimensional normal forms. For each model, we choose a basis
$(e_1,e_2,e_3)$; all brackets not displayed are zero. We use the notation
\[
\xi\odot\zeta=\xi\otimes\zeta+\zeta\otimes\xi,
\qquad (\xi)^2=\xi\otimes\xi.
\]We give only the non-vanishing Lie brackets and $g$ denote the metric.

\medskip
\noindent\textbf{Family $\mathfrak G_1$.}
We may take $e_1$ timelike and $(e_2,e_3)$ spacelike and orthonormal.
The only possibly nonzero brackets are
\[
[e_1,e_2]=\mu e_3,\qquad [e_1,e_3]=-\mu e_2,
\qquad \mu\geq0\esp g=-(e^1)^2+(e^2)^2+(e^3)^2.
\]

\medskip
\noindent\textbf{Family $\mathfrak G_2$.}
Here $\mathfrak d=\mathcal A=0$, so that $\mathfrak h=\mathbb Ru$.
Taking $(e_1,e_2,e_3)=(e,u,f)$ and
$r=\langle u,u\rangle>0$, we obtain
\[
	[e_2,e_1]=e_1,\;
	[e_3,e_1]=\lambda e_1,\;
	[e_3,e_2]=e_3-\lambda e_2-\lambda^2r\,e_1,
\qquad \lambda\in\mathbb R,
\]
with $g=e^1\odot e^3+r(e^2)^2$.

\medskip
\noindent\textbf{Family $\mathfrak G_3$.}
In this dimension $\mathfrak d=0$, $\mathfrak a=\mathbb Ra$, and
$F_2=A_2=0$. Taking $(e_1,e_2,e_3)=(e,a,f)$,
where $a$ is a unit vector, we find
\[
[e_3,e_2]=\gamma e_1,\qquad \gamma\in\mathbb R\esp g=e^1\odot e^3+(e^2)^2.
\]

\medskip
\noindent\textbf{Family $\mathfrak G_4$.}
There is no three-dimensional model in this family: the nonzero
skew-symmetric endomorphism $E$ requires
$\dim\operatorname{Im}E\geq2$, whereas $\dim\mathfrak h=1$.

\medskip
\noindent\textbf{Family $\mathfrak G_5$.}
Here $\mathfrak d_1=\mathfrak d_2=0$ and
$\dim\mathfrak a_1+\dim\mathfrak a_2=1$.
Writing $(e_1,e_2,e_3)=(e,a,f)$, where $a$ is a unit vector,
we obtain two possibilities:
\[
\begin{array}{ll}
	\mathfrak G_{5,3}^{(1)}:\quad \mathfrak a_1=\mathbb Ra,\ \mathfrak a_2=0,
	&\begin{cases}
		[e_3,e_1]=\lambda e_1,\\
		[e_3,e_2]=\gamma e_1;
	\end{cases}\\[3mm]
	\mathfrak G_{5,3}^{(2)}:\quad \mathfrak a_1=0,\ \mathfrak a_2=\mathbb Ra,
	&\begin{cases}
		[e_3,e_1]=\lambda e_1,\\
		[e_3,e_2]=\lambda e_2+\gamma e_1.
	\end{cases}
\end{array}
\]
In both cases $\lambda\neq0$, $\gamma\in\mathbb R$, and $g=e^1\odot e^3+(e^2)^2$.

\medskip
\noindent\textbf{Family $\mathfrak G_6$.}
Here $\mathfrak d=0$ and $\mathfrak a=\mathbb Ra$. Taking
$(e_1,e_2,e_3)=(e,a,f)$, where $a$ is a unit vector, we obtain
\[
[e_3,e_1]=\alpha e_3,\qquad \alpha\neq0\esp g=e^1\odot e^3+(e^2)^2.
\]

These are the models to be considered in the next step, where
we identify their Lie brackets by changes of basis and transform
the corresponding metric matrices.

\subsection{Identification of the non-abelian three-dimensional models}

We now identify the non-abelian three-dimensional Lie algebras
obtained above with the standard models of~\cite{Zi}. For each
$\mathfrak G_{i,3}$, we exhibit a basis $(f_1,f_2,f_3)$,
write the corresponding nonzero brackets, and express the metric
in the dual basis $(f^1,f^2,f^3)$. 
The changes of basis and the formulas for the metrics can be
checked directly; we omit these elementary verifications.

\medskip
\noindent\textbf{Model $\mathfrak G_{1,3}$ ($\mu\neq0$).}
Set
\[
f_1=\frac{1}{\mu}e_1,\qquad f_2=e_2,\qquad f_3=e_3.
\]
Then
\[
[f_1,f_2]=f_3,\qquad [f_1,f_3]=-f_2\esp g
=-\frac1{\mu^2}(f^1)^2+(f^2)^2+(f^3)^2.
\]
Thus $\mathfrak G_{1,3}\simeq A_{3,6}$.

\medskip
\noindent\textbf{Model $\mathfrak G_{2,3}$.}
Recall that $r=\langle e_2,e_2\rangle>0$. Put
\[
	f_1=e_1,\;
	f_2=e_3-\lambda e_2-\frac{\lambda^2r}{2}e_1,\;
	f_3=e_2+\lambda r e_1.
\]
We obtain
\[
[f_3,f_1]=f_1,\qquad [f_3,f_2]=-f_2\esp g
=f^1\odot f^2+r(f^3)^2.
\]
Hence $\mathfrak G_{2,3}\simeq A_{3,4}$.

\medskip
\noindent\textbf{Model $\mathfrak G_{3,3}$ ($\gamma\neq0$).}
Choose
\[
f_1=\gamma e_1,\qquad f_2=e_2,\qquad f_3=e_3.
\]
We obtain
\[
[f_3,f_2]=f_1\esp g
=\gamma f^1\odot f^3+(f^2)^2.
\]
Thus $\mathfrak G_{3,3}\simeq A_{3,1}$.

\medskip
\noindent\textbf{Model $\mathfrak G_{4,3}$.}
There is no three-dimensional Lie algebra in this family.

\medskip
\noindent\textbf{Model $\mathfrak G_{5,3}^{(1)}$.}
For $\lambda\neq0$, set
\[
	f_1=e_1,\;
	f_2=e_2-\frac{\gamma}{\lambda}e_1,\;
	f_3=\frac1\lambda e_3
	+\frac{\gamma}{\lambda^2}e_2
	-\frac{\gamma^2}{2\lambda^3}e_1.
\]
Then
\[
[f_3,f_1]=f_1,\qquad [f_3,f_2]=0\esp g
=\frac1\lambda f^1\odot f^3+(f^2)^2.
\]
Consequently,
$\mathfrak G_{5,3}^{(1)}\simeq A_2\oplus A_1$.

\medskip
\noindent\textbf{Model $\mathfrak G_{5,3}^{(2)}$.}
Here $\lambda\neq0$. There are two possibilities.

If $\gamma\neq0$, take
\[
f_1=\frac\gamma\lambda e_1,\qquad
f_2=e_2,\qquad f_3=\frac1\lambda e_3.
\]
Then
\[
[f_3,f_1]=f_1,\qquad [f_3,f_2]=f_2+f_1\esp g
=\frac\gamma{\lambda^2}f^1\odot f^3+(f^2)^2.
\]
Thus $\mathfrak G_{5,3}^{(2)}\simeq A_{3,2}$.

If $\gamma=0$, take instead
\[
f_1=e_1,\qquad f_2=e_2,\qquad
f_3=\frac1\lambda e_3.
\]
We obtain
\[
[f_3,f_1]=f_1,\qquad [f_3,f_2]=f_2\esp g
=\frac1\lambda f^1\odot f^3+(f^2)^2.
\]
Hence in this case
$\mathfrak G_{5,3}^{(2)}\simeq A_{3,3}$.

\medskip
\noindent\textbf{Model $\mathfrak G_{6,3}$.}
For $\alpha\neq0$, put
\[
f_1=e_3,\qquad f_2=e_2,\qquad
f_3=-\frac1\alpha e_1.
\]
Then
\[
[f_3,f_1]=f_1\esp g
=-\frac1\alpha f^1\odot f^3+(f^2)^2.
\]
Therefore $\mathfrak G_{6,3}\simeq A_2\oplus A_1$.

\medskip
\noindent
In summary, the non-abelian three-dimensional types arising from
our six families are precisely
\[
A_2\oplus A_1,\quad A_{3,1},\quad A_{3,2},\quad
A_{3,3},\quad A_{3,4},\quad A_{3,6}.
\]
The overlap between $\mathfrak G_{5,3}^{(1)}$ and
$\mathfrak G_{6,3}$ illustrates again that our general families
are not required to be pairwise disjoint.

\medskip
\noindent
The preceding computations yield the
following table, which summarizes all
non-abelian three-dimensional Lie
algebras admitting a flat left-invariant
Lorentzian metric, together with the
corresponding metrics obtained from
our six families. When a metric depends on a single parameter, we denote it by $\rho$
if it is a diagonal coefficient $\rho=\langle f_i,f_i\rangle$,
and by $\gamma$ if it is an off-diagonal coefficient
$\gamma=\langle f_i,f_j\rangle$ with $i\ne j$.

\begin{table}[htbp]
	\centering
	\begingroup
	\footnotesize
	\setlength{\tabcolsep}{4pt}
	\renewcommand{\arraystretch}{1.4}
	\begin{tabular}{|c|c|l|l|}
		\hline
		Model & Lie algebra & Nonzero brackets & Lorentzian metric $g$\\
		\hline
		$\mathfrak G_{1,3}$
		& $A_{3,6}$
		& \shortstack[l]{
			$[f_1,f_2]=f_3$\\
			$[f_1,f_3]=-f_2$}
		& $\rho(f^1)^2+(f^2)^2+(f^3)^2,\; \rho<0$
		\\
		\hline
		$\mathfrak G_{2,3}$
		& $A_{3,4}$
		& \shortstack[l]{
			$[f_3,f_1]=f_1$\\
			$[f_3,f_2]=-f_2$}
		& $f^1\odot f^2+\rho(f^3)^2,\; \rho>0$
		\\
		\hline
		$\mathfrak G_{3,3}$
		& $A_{3,1}$
		& $[f_3,f_2]=f_1$
		& $\gamma f^1\odot f^3+(f^2)^2,\;\gamma\not=0$
		\\
		\hline
		$\mathfrak G_{5,3}^{(1)}$, $\mathfrak G_{6,3}$
		& $A_2\oplus A_1$
		& $[f_3,f_1]=f_1$
		& $\gamma f^1\odot f^3+(f^2)^2,\;\gamma\not=0$
		\\
		\hline
		$\mathfrak G_{5,3}^{(2)}$
		& $A_{3,2}$
		& \shortstack[l]{
			$[f_3,f_1]=f_1$\\
			$[f_3,f_2]=f_2+f_1$}
		& $\gamma f^1\odot f^3+(f^2)^2,\gamma\not=0$
		\\
		\hline
		$\mathfrak G_{5,3}^{(2)}$
		& $A_{3,3}$
		& \shortstack[l]{
			$[f_3,f_1]=f_1$\\
			$[f_3,f_2]=f_2$}
		& $\gamma f^1\odot f^3+(f^2)^2,\;\gamma\not=0$
		\\
		\hline
	\end{tabular}
	\endgroup
	\caption{Non-abelian three-dimensional Lie
		algebras admitting flat left-invariant
		Lorentzian metrics.}
	\label{tab:flat-lorentzian-3d}
\end{table}

Notice that $A_2\oplus A_1$ occurs
in two different families, illustrating
the overlaps between our general
models. The table is not intended
to give a classification of the
corresponding metrics up to isometry.


\subsection{Four-dimensional models extracted from Tables~\ref{1}--\ref{6}}

We now consider the non-abelian four-dimensional Lie algebras arising
from our six families. We first retain their original parameters, without
reducing the brackets to standard normal forms. In each case
$\G_{i,4}=\operatorname{span}\{e_1,e_2,e_3,e_4\}$; only the nonzero
brackets are displayed. We write
$\xi\odot\zeta=\xi\otimes\zeta+\zeta\otimes\xi$ and
$(\xi)^2=\xi\otimes\xi$.

\medskip
\noindent\textbf{Family $\mathfrak G_1$.}
With $e_1$ timelike and $e_2,e_3,e_4$ spacelike orthonormal, we obtain
\[
[e_1,e_2]=\mu e_3,\quad [e_1,e_3]=-\mu e_2,\quad
[e_4,e_2]=\nu e_3,\quad [e_4,e_3]=-\nu e_2,
\]
where $\mu,\nu\in\mathbb R$ are not both zero, and
$g=-(e^1)^2+(e^2)^2+(e^3)^2+(e^4)^2$.

\medskip
\noindent\textbf{Family $\mathfrak G_2$.}
Here $\df=0$, $\mathcal A=\mathbb Ra$, and
$\h=\mathbb Ru\oplus\mathbb Ra$.
Set $(e_1,e_2,e_3,e_4)=(e,u,a,f)$, where
$\langle u,u\rangle=r>0$, $\langle a,a\rangle=1$ and $n=\nu a$.
Then
\[
[e_2,e_1]=e_1,\quad [e_4,e_1]=\lambda e_1,\quad
[e_2,e_3]=\nu e_1,\quad [e_4,e_3]=\lambda\nu e_1,
\]
\[
[e_4,e_2]=e_4+\nu e_3-\lambda e_2-\lambda^2r e_1,
\qquad \lambda,\nu\in\mathbb R,
\]
and $g=e^1\odot e^4+r(e^2)^2+(e^3)^2$.

\medskip
\noindent\textbf{Family $\mathfrak G_3$.}
In dimension four the condition on $\rho$ forces $\df=0$;
thus $\af=\operatorname{span}\{X,Y\}$ is Euclidean.
Put $(e_1,e_2,e_3,e_4)=(e,X,Y,f)$, with $(X,Y)$ orthonormal.
The condition $[F_2,A_2]=A_2^2$ leads to the following cases:
\[
\begin{aligned}
\textup{(a)}\quad &A_2=0,\quad F_2=\omega J,\quad w_2=pX+qY:\\
&[e_4,e_2]=\omega e_3+p e_1,\quad
[e_4,e_3]=-\omega e_2+q e_1;\\[1mm]
\textup{(b)}\quad &A_2Y=\kappa X,\quad A_2X=0,\quad F_2=0,\quad w_2=pX+qY:\\
&[e_4,e_2]=p e_1,\quad [e_4,e_3]=\kappa e_2+q e_1,\\
&[e_3,e_2]=\kappa e_1.
\end{aligned}
\]
Here $JX=Y$, $JY=-X$, $\omega,p,q\in\mathbb R$ and
$\kappa>0$ in case~\textup{(b)}. In both cases
$g=e^1\odot e^4+(e^2)^2+(e^3)^2$.
The subcase $\omega=p=q=0$ in~\textup{(a)} is abelian and is excluded.

\medskip
\noindent\textbf{Family $\mathfrak G_4$.}
Necessarily $\df=\af=0$ and $\mathfrak b=\operatorname{span}\{X,Y\}$.
With $(e_1,e_2,e_3,e_4)=(e,X,Y,f)$, $E=\varepsilon J$ and
$F_3=\varphi J$, we obtain
\[
[e_1,e_2]=\varepsilon e_3,\quad [e_1,e_3]=-\varepsilon e_2,
\quad [e_4,e_2]=\varphi e_3,\quad [e_4,e_3]=-\varphi e_2,
\]
where $\varepsilon\neq0$, $\varphi\in\mathbb R$, and
$g=e^1\odot e^4+(e^2)^2+(e^3)^2$.

\medskip
\noindent\textbf{Family $\mathfrak G_5$.}
The common-kernel conditions imply $\df_1=\df_2=0$;
thus $\dim\af_1+\dim\af_2=2$. We write
$w_2=pX+qY$, take $(X,Y)$ orthonormal and set
$(e_1,e_2,e_3,e_4)=(e,X,Y,f)$. There are three cases.

\noindent\textup{(a)} $\af_1=\operatorname{span}\{X,Y\}$, $\af_2=0$,
$G_1=\omega J$:
\[
[e_4,e_1]=\lambda e_1,\quad [e_4,e_2]=\omega e_3+p e_1,
\quad [e_4,e_3]=-\omega e_2+q e_1.
\]
\noindent\textup{(b)} $\af_1=0$,
$\af_2=\operatorname{span}\{X,Y\}$, $G_2=\omega J$:
\[
[e_4,e_1]=\lambda e_1,\quad [e_4,e_2]=\lambda e_2+\omega e_3+p e_1,
\quad [e_4,e_3]=-\omega e_2+\lambda e_3+q e_1.
\]
\noindent\textup{(c)} $\af_1=\mathbb RX$, $\af_2=\mathbb RY$,
$V(Y)=vX$:
\[
[e_4,e_1]=\lambda e_1,\quad [e_2,e_3]=-v e_1,\quad
[e_4,e_2]=p e_1,\quad [e_4,e_3]=\lambda e_3+v e_2+q e_1.
\]
Here $\lambda\neq0$ and $\omega,p,q,v\in\mathbb R$.
The metric in every case is
$g=e^1\odot e^4+(e^2)^2+(e^3)^2$.

\medskip
\noindent\textbf{Family $\mathfrak G_6$.}
Here $\df=0$, $\af=\operatorname{span}\{X,Y\}$ and
$E_2=\omega J$. With $(e_1,e_2,e_3,e_4)=(e,X,Y,f)$, we obtain
\[
[e_4,e_1]=\alpha e_4,\qquad
[e_1,e_2]=\omega e_3,\qquad [e_1,e_3]=-\omega e_2,
\quad \alpha\neq0,\ \omega\in\mathbb R,
\]
and $g=e^1\odot e^4+(e^2)^2+(e^3)^2$.

\subsection{Identification of the non-abelian four-dimensional models}

We identify these brackets with the standard four-dimensional Lie
algebras of~\cite{Zi}. In each case, $(f_1,f_2,f_3,f_4)$ is the new
basis and $(f^1,f^2,f^3,f^4)$ its dual. Brackets and metrics refer
to the same basis. All direct verifications are omitted.
For compactness, when an adapted null-frame change is needed, we use
\[
e'=e,\qquad z'=z+\langle z,x\rangle e\ (z\in\h),\qquad
f'=f-x-\tfrac12\langle x,x\rangle e,
\tag{$*$}\label{eq:4d-null-change}
\]
where $x\in\h$ and $\h'$ is the image of $\h$ under $z\mapsto z'$.

\medskip
\noindent\textbf{Model $\mathfrak G_{1,4}$.}
If $\mu\neq0$ and $\nu=0$, take
\[
f_1=\mu^{-1}e_1,\quad f_2=e_4,\quad f_3=e_2,\quad f_4=e_3.
\]
Then $[f_1,f_3]=f_4$, $[f_1,f_4]=-f_3$, and
\[
g=-\mu^{-2}(f^1)^2+(f^2)^2+(f^3)^2+(f^4)^2.
\]
If $\nu\neq0$, take instead
\[
f_1=\nu^{-1}e_4,\quad f_2=e_1-(\mu/\nu)e_4,\quad
f_3=e_2,\quad f_4=e_3.
\]
The same two brackets hold, and
\[
g=\nu^{-2}(f^1)^2-\frac{\mu}{\nu^2}f^1\odot f^2
+\left(\frac{\mu^2}{\nu^2}-1\right)(f^2)^2
+(f^3)^2+(f^4)^2.
\]
Thus $\mathfrak G_{1,4}\simeq A_{3,6}\oplus A_1$.

\medskip
\noindent\textbf{Model $\mathfrak G_{2,4}$.}
Apply~\eqref{eq:4d-null-change} with $x=-\nu e_3$ and write
$e'=e_1$, $a'=e_3-\nu e_1$,
$f'=e_4+\nu e_3-\frac12\nu^2e_1$. Put
\[
f_1=e_1,\quad f_2=f'-\lambda e_2-\tfrac12\lambda^2r e_1,
\quad f_3=e_2+\lambda r e_1,\quad f_4=a'.
\]
Then
\[
[f_3,f_1]=f_1,\qquad [f_3,f_2]=-f_2,\qquad
g=f^1\odot f^2+r(f^3)^2+(f^4)^2.
\]
Hence $\mathfrak G_{2,4}\simeq A_{3,4}\oplus A_1$.

\medskip
\noindent\textbf{Model $\mathfrak G_{3,4}$, case~\textup{(a)}.}
If $\omega\neq0$, use~\eqref{eq:4d-null-change} with
$x=-F_2^{-1}w_2$, so that
$[f',X']=\omega Y'$ and $[f',Y']=-\omega X'$.
Taking
\[
f_1=\omega^{-1}f',\quad f_2=e_1,\quad f_3=X',\quad f_4=Y',
\]
we find
\[
[f_1,f_3]=f_4,\quad [f_1,f_4]=-f_3,\qquad
g=\omega^{-1}f^1\odot f^2+(f^3)^2+(f^4)^2.
\]
Thus $\mathfrak G_{3,4}\simeq A_{3,6}\oplus A_1$.
If $\omega=0$ and $w_2\neq0$, put $\gamma=\|w_2\|$,
$X=w_2/\gamma$ and choose a unit $Y\perp X$. Then
\[
f_1=\gamma e_1,\quad f_2=X,\quad f_3=e_4,\quad f_4=Y.
\]
We obtain
\[
[f_3,f_2]=f_1,\qquad
g=\gamma f^1\odot f^3+(f^2)^2+(f^4)^2,
\]
which is $A_{3,1}\oplus A_1$.

\medskip
\noindent\textbf{Model $\mathfrak G_{3,4}$, case~\textup{(b)}.}
With $A_2Y=\kappa X$, $w_2=pX+qY$, take
\[
f_1=\kappa^2e_1,\quad f_2=\kappa e_2+qe_1,\quad
f_3=e_3,\quad f_4=e_4-(p/\kappa)e_3.
\]
Then
\[
[f_4,f_3]=f_2,\quad [f_3,f_2]=f_1,
\]
\[
g=\kappa^2f^1\odot f^4+\kappa^2(f^2)^2+(f^3)^2
+q f^2\odot f^4-\frac p\kappa f^3\odot f^4
+\frac{p^2}{\kappa^2}(f^4)^2.
\]
Thus $\mathfrak G_{3,4}\simeq A_{4,1}$.

\medskip
\noindent\textbf{Model $\mathfrak G_{4,4}$.}
Set
\[
f_1=\varepsilon^{-1}e_1,\quad
f_2=e_4-(\varphi/\varepsilon)e_1,\quad
f_3=e_2,\quad f_4=e_3.
\]
Then
\[
[f_1,f_3]=f_4,\quad [f_1,f_4]=-f_3,\qquad
g=\varepsilon^{-1}f^1\odot f^2-\frac{2\varphi}{\varepsilon}(f^2)^2
+(f^3)^2+(f^4)^2.
\]
Hence $\mathfrak G_{4,4}\simeq A_{3,6}\oplus A_1$.

\medskip
\noindent\textbf{Model $\mathfrak G_{5,4}$, case~\textup{(a)}.}
If $\omega\neq0$, apply~\eqref{eq:4d-null-change} with
$x=-(\lambda I+G_1)^{-1}w_2$ and take
\[
f_1=\omega^{-1}f',\quad f_2=e_1,\quad f_3=X',\quad f_4=Y'.
\]
Writing $s=\lambda/\omega\neq0$, we obtain
\[
[f_1,f_2]=s f_2,\quad [f_1,f_3]=f_4,\quad
[f_1,f_4]=-f_3,\qquad
g=\omega^{-1}f^1\odot f^2+(f^3)^2+(f^4)^2.
\]
Thus $\mathfrak G_{5,4}\simeq A_{4,6}^{s,0}$.
If $\omega=0$, use~\eqref{eq:4d-null-change} with
$x=-\lambda^{-1}w_2$ and take
\[
f_1=\lambda^{-1}f',\quad f_2=e_1,\quad f_3=X',\quad f_4=Y'.
\]
Then
\[
[f_1,f_2]=f_2,\qquad
g=\lambda^{-1}f^1\odot f^2+(f^3)^2+(f^4)^2,
\]
so that $\mathfrak G_{5,4}\simeq A_2\oplus2A_1$.

\medskip
\noindent\textbf{Model $\mathfrak G_{5,4}$, case~\textup{(b)}.}
If $\omega\neq0$, apply~\eqref{eq:4d-null-change} with
$x=-G_2^{-1}w_2$ and take
\[
f_1=\omega^{-1}f',\quad f_2=e_1,\quad f_3=X',\quad f_4=Y'.
\]
With $s=\lambda/\omega\neq0$, this gives
\[
[f_1,f_2]=s f_2,\quad [f_1,f_3]=s f_3+f_4,\quad
[f_1,f_4]=-f_3+s f_4,\qquad
g=\omega^{-1}f^1\odot f^2+(f^3)^2+(f^4)^2.
\]
Thus $\mathfrak G_{5,4}\simeq A_{4,6}^{s,s}$.
If $\omega=0$ and $w_2=0$, take
\[
f_1=\lambda^{-1}e_4,\quad f_2=e_1,\quad f_3=e_2,\quad f_4=e_3.
\]
We obtain
\[
[f_1,f_2]=f_2,\quad [f_1,f_3]=f_3,\quad [f_1,f_4]=f_4,
\qquad g=\lambda^{-1}f^1\odot f^2+(f^3)^2+(f^4)^2,
\]
which is $A_{4,5}^{1,1}$.
If $\omega=0$ and $w_2\neq0$, choose
$X=w_2/\gamma$, $Y\perp X$ unit, with $\gamma=\|w_2\|>0$,
and take
\[
f_1=\lambda^{-1}e_4,\quad f_2=(\gamma/\lambda)e_1,
\quad f_3=X,\quad f_4=Y.
\]
Then
\[
[f_1,f_2]=f_2,\quad [f_1,f_3]=f_3+f_2,\quad
[f_1,f_4]=f_4,\qquad
g=\frac{\gamma}{\lambda^2}f^1\odot f^2+(f^3)^2+(f^4)^2,
\]
which is $A_{4,2}^{1}$.

\medskip
\noindent\textbf{Model $\mathfrak G_{5,4}$, case~\textup{(c)}.}
Suppose first that $v=0$. Use~\eqref{eq:4d-null-change} with
$x=-(p/\lambda)X$ and denote the new vectors by $X',Y',f'$.
If $q=0$, take
\[
f_1=\lambda^{-1}f',\quad f_2=e_1,\quad f_3=Y',\quad f_4=X'.
\]
Then
\[
[f_1,f_2]=f_2,\quad [f_1,f_3]=f_3,\qquad
g=\lambda^{-1}f^1\odot f^2+(f^3)^2+(f^4)^2.
\]
Thus $\mathfrak G_{5,4}\simeq A_{3,3}\oplus A_1$.
If $q\neq0$, take instead
\[
f_1=\lambda^{-1}f',\quad f_2=(q/\lambda)e_1,
\quad f_3=Y',\quad f_4=X'.
\]
We obtain
\[
[f_1,f_2]=f_2,\quad [f_1,f_3]=f_3+f_2,\qquad
g=\frac{q}{\lambda^2}f^1\odot f^2+(f^3)^2+(f^4)^2,
\]
and hence $\mathfrak G_{5,4}\simeq A_{3,2}\oplus A_1$.

If $v\neq0$, set $c=q/(\lambda v)+p/\lambda^2$ and choose
\[
f_1=\lambda^{-1}e_4+c e_2-\frac{\lambda c^2}{2}e_1,\quad
f_2=v e_1,\quad f_3=e_2-\frac p\lambda e_1,\quad
f_4=e_3+\frac v\lambda e_2-cv e_1.
\]
We obtain
\[
[f_1,f_2]=f_2,\quad [f_1,f_4]=f_4,\quad [f_4,f_3]=f_2,
\]
\[
g=\frac v\lambda f^1\odot f^2+\frac q{\lambda v}f^1\odot f^3
+(f^3)^2+\frac v\lambda f^3\odot f^4
+\left(1+\frac{v^2}{\lambda^2}\right)(f^4)^2.
\]
Hence $\mathfrak G_{5,4}\simeq A_{4,9}^{0}$.

\medskip
\noindent\textbf{Model $\mathfrak G_{6,4}$.}
If $\omega=0$, choose
\[
f_1=-\alpha^{-1}e_1,\quad f_2=e_4,\quad f_3=e_2,\quad f_4=e_3.
\]
Then
\[
[f_1,f_2]=f_2,\qquad
g=-\alpha^{-1}f^1\odot f^2+(f^3)^2+(f^4)^2,
\]
so that $\mathfrak G_{6,4}\simeq A_2\oplus2A_1$.
If $\omega\neq0$, take
\[
f_1=\omega^{-1}e_1,\quad f_2=e_4,\quad f_3=e_2,\quad f_4=e_3.
\]
Writing $s=-\alpha/\omega\neq0$, we find
\[
[f_1,f_2]=s f_2,\quad [f_1,f_3]=f_4,\quad
[f_1,f_4]=-f_3,\qquad
g=\omega^{-1}f^1\odot f^2+(f^3)^2+(f^4)^2.
\]
Thus $\mathfrak G_{6,4}\simeq A_{4,6}^{s,0}$.

The metrics obtained above can be further normalized by automorphisms
of the corresponding Lie algebras. We retain the normalized brackets
and display only the resulting changes of basis.

\medskip
\noindent\textbf{Further simplification for $A_{3,6}\oplus A_1$.}
The Lie algebra $A_{3,6}\oplus A_1$ occurs in the families
$\mathfrak G_1$, $\mathfrak G_3$ and $\mathfrak G_4$.
In all four cases displayed above, its nonzero brackets are
\[
[f_1,f_3]=f_4,\qquad [f_1,f_4]=-f_3,
\]
and the metric can be written as
\begin{equation}\label{eq:A36-general-metric}
	g=A(f^1)^2+Bf^1\odot f^2+C(f^2)^2+(f^3)^2+(f^4)^2,
	\qquad AC-B^2<0.
\end{equation}
The coefficients arising from our families are
\[
\begin{array}{c|ccc}
	\text{Source} & A & B & C\\ \hline
	\mathfrak G_1\ (\nu=0) & -\mu^{-2} &0&1\\
	\mathfrak G_1\ (\nu\ne0)&\nu^{-2}&-\mu/\nu^2&\mu^2/\nu^2-1\\
	\mathfrak G_3\ (\omega\ne0)&0&\omega^{-1}&0\\
	\mathfrak G_4\ (\varepsilon\ne0)&0&\varepsilon^{-1}&-2\varphi/\varepsilon
\end{array}
\]
(the assumptions on $\mu,\nu,\omega,\varepsilon$ are those made above).

\smallskip
\noindent\textit{Case $C\ne0$.}
Set
\[
F_1=f_1-\frac BCf_2,\qquad
F_2=\frac1{\sqrt{|C|}}f_2,\qquad F_3=f_3,\quad F_4=f_4.
\]
Then
\[
[F_1,F_3]=F_4,\qquad [F_1,F_4]=-F_3,
\]
while
\[
g=\left(A-\frac{B^2}{C}\right)(F^1)^2
+\operatorname{sgn}(C)(F^2)^2+(F^3)^2+(F^4)^2.
\]
Put $a=\sqrt{(B^2-AC)/|C|}>0$. Depending on the sign of $C$,
we obtain either
\begin{equation}\label{eq:A36-space-centre}
	g_{+,a}=-a^2(F^1)^2+(F^2)^2+(F^3)^2+(F^4)^2,
	\qquad a>0, \quad (C>0),
\end{equation}
or
\begin{equation}\label{eq:A36-time-centre}
	g_{-,a}=a^2(F^1)^2-(F^2)^2+(F^3)^2+(F^4)^2,
	\qquad a>0, \quad (C<0).
\end{equation}
The center is respectively spacelike or timelike.

\smallskip
\noindent\textit{Case $C=0$.}
Nondegeneracy implies $B\ne0$. Set
\[
F_1=f_1-\frac{A}{2B}f_2,\qquad F_2=B^{-1}f_2,\qquad
F_3=f_3,\quad F_4=f_4.
\]
The brackets are unchanged, and
\begin{equation}\label{eq:A36-null-centre}
	g_0=F^1\odot F^2+(F^3)^2+(F^4)^2.
\end{equation}
Thus all the metrics in our list for which the center is null
are equivalent by automorphisms.

\smallskip
\noindent
More explicitly, the occurrences in the six tables reduce as follows:
\[
\begin{array}{c|c|c}
	\text{Source and parameter range}&\text{Canonical metric}&\text{Remaining parameter}\\ \hline
	\mathfrak G_1,\ \nu=0 &g_{+,a}&a=|\mu|^{-1}\\
	\mathfrak G_1,\ \nu\ne0,\ \mu^2>\nu^2&g_{+,a}&a=(\mu^2-\nu^2)^{-1/2}\\
	\mathfrak G_1,\ \nu\ne0,\ \mu^2<\nu^2&g_{-,a}&a=(\nu^2-\mu^2)^{-1/2}\\
	\mathfrak G_1,\ \nu\ne0,\ \mu^2=\nu^2&g_0&\text{none}\\
	\mathfrak G_3,\ \omega\ne0&g_0&\text{none}\\
	\mathfrak G_4,\ \varepsilon\varphi<0&g_{+,a}&a=(2|\varepsilon\varphi|)^{-1/2}\\
	\mathfrak G_4,\ \varepsilon\varphi>0&g_{-,a}&a=(2|\varepsilon\varphi|)^{-1/2}\\
	\mathfrak G_4,\ \varphi=0&g_0&\text{none}
\end{array}
\]

\medskip
\noindent\textbf{Further simplification of the metrics on $A_{4,1}$ and $A_{4,9}^{0}$.}

\smallskip
\noindent\textit{The algebra $A_{4,1}$.}
Recall that
\[
[f_4,f_3]=f_2,\qquad [f_3,f_2]=f_1,
\]
and that the metric is
\[
g=\kappa^2 f^1\odot f^4+\kappa^2(f^2)^2+(f^3)^2
+q f^2\odot f^4-\frac p\kappa f^3\odot f^4
+\frac{p^2}{\kappa^2}(f^4)^2,\qquad \kappa>0.
\]
Take the following new basis:
\[
F_1=\frac1\kappa f_1,\quad F_2=\frac1\kappa f_2,\quad
F_3=f_3+\frac{p}{\kappa^3}f_1,\quad
F_4=\frac1\kappa f_4-\frac{p^2}{2\kappa^5}f_1.
\]
The brackets are unchanged in form:
\[
[F_4,F_3]=F_2,\qquad [F_3,F_2]=F_1,
\]
whereas the metric becomes
\[
	g=F^1\odot F^4+(F^2)^2+(F^3)^2
	+\eta F^2\odot F^4,\qquad
	\eta=\frac q{\kappa^2}.
\]
Replacing $(F_2,F_3)$ by $(-F_2,-F_3)$ preserves the brackets
and changes $\eta$ into $-\eta$. Thus one may assume $\eta\geq0$.

\smallskip
\noindent\textit{The algebra $A_{4,9}^{0}$.}
Here
\[
[f_1,f_2]=f_2,\quad [f_1,f_4]=f_4,\quad [f_4,f_3]=f_2,
\]
and the metric obtained above is
\[
g=t f^1\odot f^2+u f^1\odot f^3+(f^3)^2
+t f^3\odot f^4+(1+t^2)(f^4)^2,
\qquad t=\frac v\lambda\ne0,\quad u=\frac q{\lambda v}.
\]
Put $d=\sqrt{1+t^2}$, let $\delta=\operatorname{sgn}(t)$, and set
$\varepsilon=\operatorname{sgn}(u)$ if $u\ne0$, $\varepsilon=1$ otherwise.
The change of basis
\[
F_1=f_1,\quad F_2=\frac{\delta}{d}f_2,\quad
F_3=\varepsilon f_3,\quad
F_4=\frac{\varepsilon\delta}{d}f_4
\]
preserves the brackets. In the dual basis the metric takes the form
\[
\displaystyle
	g=s\bigl(F^1\odot F^2+F^3\odot F^4\bigr)
	+\beta F^1\odot F^3+(F^3)^2+(F^4)^2,\quad
	s=\frac{|t|}{\sqrt{1+t^2}}\in(0,1),\quad \beta=|u|\geq0.
\]

\smallskip
\noindent
These remaining parameters cannot be eliminated by automorphisms.
Indeed, every automorphism of $A_{4,1}$ is described by a basis
\[
F_3=a f_3+b f_4+c f_2+d f_1,\quad
F_4=h f_4+j f_2+k f_1,\quad
F_2=ah f_2-aj f_1,\quad F_1=a^2h f_1,\qquad ah\ne0.
\]
If such an automorphism carries the displayed canonical metric
with parameter $\eta$ to one with parameter $\eta'$, the equalities
$g(F_1,F_3)=g(F_2,F_3)=0$, $g(F_3,F_3)=g(F_1,F_4)=1$
force $b=c=0$, $a^2=h^2=1$, and hence $\eta'=a\eta$.
Likewise, every automorphism of $A_{4,9}^{0}$ is given by
\[
F_1=f_1+a f_2+c f_4,\quad F_2=mr f_2,\quad
F_3=m(f_3-c f_2),\quad F_4=r f_4+z f_2,\qquad mr\ne0.
\]
For the canonical metric, the identities $g(F_3,F_3)=g(F_4,F_4)=1$
force $m^2=r^2=1$, while
$g(F_1,F_2)=mr s$ and $g(F_1,F_3)=m\beta$.
Since $s>0$, the canonical metrics are equivalent under automorphisms
only when the respective pairs $(s,\beta)$ agree.

\subsection{Summary of the four-dimensional classification}

We collect in Table~\ref{tab:flat-lorentzian-dim4} the non-abelian
four-dimensional Lie algebras obtained from the six families, together
with the corresponding normal forms of their flat Lorentzian metrics.
In every row, $(f_1,f_2,f_3,f_4)$ is a basis of the Lie algebra and
$(f^1,f^2,f^3,f^4)$ is its dual basis. Only nonzero brackets are listed.
We use the abbreviations
\[
\xi\odot\zeta=\xi\otimes\zeta+\zeta\otimes\xi,
\qquad Q=(f^3)^2+(f^4)^2.
\]
When a metric depends on a single parameter, we denote it by $\rho$
if it is a diagonal coefficient $\rho=\langle f_i,f_i\rangle$,
and by $\gamma$ if it is an off-diagonal coefficient
$\gamma=\langle f_i,f_j\rangle$ with $i\ne j$.

\begingroup
\footnotesize
\setlength{\tabcolsep}{3pt}
\renewcommand{\arraystretch}{1.45}
\begin{longtable}{@{}>{\raggedright\arraybackslash}p{.185\textwidth}
		>{\raggedright\arraybackslash}p{.350\textwidth}
		>{\raggedright\arraybackslash}p{.438\textwidth}@{}}
	\caption{Non-abelian four-dimensional flat Lorentzian Lie algebras and normalized metrics.}
	\label{tab:flat-lorentzian-dim4}\\
	\hline
	\textbf{Lie algebra} & \textbf{Nonzero Lie brackets} & \textbf{Metric and parameter range}\\
	\hline
	\endfirsthead
	\multicolumn{3}{c}{\emph{Table~\thetable\ (continued)}}\\
	\hline
	\textbf{Lie algebra} & \textbf{Nonzero Lie brackets} & \textbf{Metric and parameter range}\\
	\hline
	\endhead
	\hline
	\endfoot
	\hline
	\endlastfoot
	
	$A_2\oplus 2A_1$
	& $[f_1,f_2]=f_2$
	& $g=f^1\odot f^2+Q$ \\
	
	\hline
	$A_{3,1}\oplus A_1$
	& $[f_3,f_2]=f_1$
	& $\displaystyle g=f^1\odot f^3+(f^2)^2+(f^4)^2$ \\
	
	\hline
	$A_{3,2}\oplus A_1$
	& $\displaystyle\begin{aligned}[t]
		[f_1,f_2]&=f_2,\\[-1pt]
		[f_1,f_3]&=f_3+f_2
	\end{aligned}$
	& $\displaystyle\begin{aligned}[t]
		g&=\gamma f^1\odot f^2+Q,\\[-1pt]
		&\gamma>0
	\end{aligned}$ \\
	
	\hline
	$A_{3,3}\oplus A_1$
	& $\displaystyle\begin{aligned}[t]
		[f_1,f_2]&=f_2,\\[-1pt]
		[f_1,f_3]&=f_3
	\end{aligned}$
	& $g=f^1\odot f^2+Q$ \\
	
	\hline
	$A_{3,4}\oplus A_1$
	& $\displaystyle\begin{aligned}[t]
		[f_3,f_1]&=f_1,\\[-1pt]
		[f_3,f_2]&=-f_2
	\end{aligned}$
	& $\displaystyle\begin{aligned}[t]
		g&=f^1\odot f^2+\rho(f^3)^2+(f^4)^2,\\[-1pt]
		&\rho>0
	\end{aligned}$ \\
	
	\hline
	$A_{3,6}\oplus A_1$
	& $\displaystyle\begin{aligned}[t]
		[f_1,f_3]&=f_4,\\[-1pt]
		[f_1,f_4]&=-f_3
	\end{aligned}$
	& $\displaystyle\begin{aligned}[t]
		g_{\rho}={}&\rho(f^1)^2
		-\operatorname{sgn}(\rho)(f^2)^2+Q,\\[-1pt]
		&\rho\ne0;\\[2pt]
		g_{\mathrm{null}}={}&f^1\odot f^2+Q
	\end{aligned}$ \\
	
	\hline
	$A_{4,1}$
	& $\displaystyle\begin{aligned}[t]
		[f_4,f_3]&=f_2,\\[-1pt]
		[f_3,f_2]&=f_1
	\end{aligned}$
	& $\displaystyle\begin{aligned}[t]
		g={}&f^1\odot f^4+(f^2)^2+(f^3)^2\\[-1pt]
		&+\gamma f^2\odot f^4,\qquad\gamma\ge0
	\end{aligned}$ \\
	
	\hline
	$A_{4,2}^{1}$
	& $\displaystyle\begin{aligned}[t]
		[f_1,f_2]&=f_2,\\[-1pt]
		[f_1,f_3]&=f_3+f_2,\\[-1pt]
		[f_1,f_4]&=f_4
	\end{aligned}$
	& $\displaystyle\begin{aligned}[t]
		g&=\gamma f^1\odot f^2+Q,\\[-1pt]
		&\gamma>0
	\end{aligned}$ \\
	
	\hline
	$A_{4,5}^{1,1}$
	& $\displaystyle\begin{aligned}[t]
		[f_1,f_2]&=f_2,\\[-1pt]
		[f_1,f_3]&=f_3,\\[-1pt]
		[f_1,f_4]&=f_4
	\end{aligned}$
	& $g=f^1\odot f^2+Q$ \\
	
	\hline
	$A_{4,6}^{s,0}$
	& $\displaystyle\begin{aligned}[t]
		[f_1,f_2]&=s f_2,\\[-1pt]
		[f_1,f_3]&=f_4,\\[-1pt]
		[f_1,f_4]&=-f_3
	\end{aligned}$
	& $\displaystyle\begin{aligned}[t]
		g&=f^1\odot f^2+Q,\\[-1pt]
		&s\ne0
	\end{aligned}$ \\
	
	\hline
	$A_{4,6}^{s,s}$
	& $\displaystyle\begin{aligned}[t]
		[f_1,f_2]&=s f_2,\\[-1pt]
		[f_1,f_3]&=s f_3+f_4,\\[-1pt]
		[f_1,f_4]&=-f_3+s f_4
	\end{aligned}$
	& $\displaystyle\begin{aligned}[t]
		g&=f^1\odot f^2+Q,\\[-1pt]
		&s\ne0
	\end{aligned}$ \\
	
	\hline
	$A_{4,9}^{0}$
	& $\displaystyle\begin{aligned}[t]
		[f_1,f_2]&=f_2,\\[-1pt]
		[f_1,f_4]&=f_4,\\[-1pt]
		[f_4,f_3]&=f_2
	\end{aligned}$
	& $\displaystyle\begin{aligned}[t]
		g={}&\ga\bigl(f^1\odot f^2+f^3\odot f^4\bigr)\\[-1pt]
		&+\beta f^1\odot f^3+Q,\\[-1pt]
		&0<\ga<1,\quad\beta\ge0
	\end{aligned}$ \\
	
\end{longtable}
\endgroup

\section{Appendix} \label{section9}

This appendix collects the detailed computations that are needed to complete the proofs of the main results but would interrupt the flow of the argument if included in the body of the paper. Subsection \ref{subsection91}  derives the system \eqref{curvature} by expanding the left-symmetry condition on the Levi-Civita product \eqref{Levi}. Subsection \ref{Subsection92} derives the Novikov conditions \eqref{NE=0} by applying Proposition \ref{Novikovpr}. Subsections \ref{subsection93}–\ref{subsection96} verify the Novikov property in each of the unimodular cases $(U1)$, $(U2)$, $(U3)$ respectively.  All computations are self-contained and can be read independently of one another.

\subsection{Computation for Proposition \ref{double}} \label{subsection91}

Let us determine the conditions under which the Lie algebra
$
\G=\mathbb Re\oplus\mathfrak h\oplus\mathbb Rf,
$
endowed with the Levi-Civita product $\cdot$ given by \eqref{Levi}, is left-symmetric; that is, such that for all $u,v,w\in\G$,
\[
Q(u,v,w):=\ass(u,v,w)-\ass(v,u,w)=0.
\]
Observe that
\[
\langle Q(u,v,w),x\rangle=-\langle Q(u,v,x),w\rangle,
\]
for all $u,v,w,x\in\G$. Consequently, $(\G,\cdot,\langle\cdot,\cdot\rangle)$ is a left-symmetric algebra if and only if, for all $a,b,c\in\h$,
\begin{equation}\label{Q}
	Q(a,b,c)=Q(a,b,e)=Q(a,e,b)=Q(a,e,e)=Q(a,f,b)=Q(a,f,e)=Q(e,f,a)=Q(e,f,e)=0.
\end{equation}
Indeed, under these assumptions, for all $u,v\in\G$ and $a\in\h$, one has
\[
\langle Q(u,v,f),e\rangle=0,
\qquad
\langle Q(u,v,f),a\rangle=0.
\]
Moreover, the quantity $\langle Q(u,v,f),f\rangle$ vanishes identically. Hence
\[
Q(u,v,f)=0,
\]
which proves the claim. Let us expand the equations \eqref{Q}.
\begin{enumerate}
	\item For $a,b,c\in\h$.
	\begin{align*}
		\ass(a,b,c)&=(a\star b+\langle Aa,b\rangle_\h e).c-a.(b\star c+\langle Ab,c\rangle_\h e)\\
		&=(a\star b)\star c+\langle A(a\star b),c\rangle_\h e+
		\langle Aa,b\rangle_\h(Ec+\langle c,v\rangle_\h e)\\
		&-a\star (b\star c)-\langle Aa,b\star c\rangle_\h e-\langle Ab,c\rangle_\h\langle a,u\rangle_\h e\\
		&=\ass_{\star}(a,b,c)+\langle Aa,b\rangle_\h Ec+\langle  A(a\star b)+b\star Aa+\langle Aa,b\rangle v-\langle a,u\rangle Ab,c\rangle_\h e.
	\end{align*}
	Then $Q(a,b,c)=0$ if and only if
	\[ \begin{cases}
		\ass_{\star}(a,b,c)-\ass_{\star}(b,a,c)=\left(\langle Ab,a\rangle_\h-\langle Aa,b\rangle_\h \right)Ec,\\
		A([a, b])+b\star Aa-a\star Ab+\left(\langle Aa,b\rangle_\h-\langle Ab,a\rangle_\h\right) v
		+\langle b,u\rangle_\h Aa-\langle a,u\rangle_\h Ab=0.
	\end{cases} \]
	We obtain the first and the second equation in \eqref{curvature}.
	\item $a,b\in\h$.
	\begin{align*}
		\ass(a,b,e)&=(a\star b+\langle Aa,b\rangle_\h e).e-\langle b,u\rangle_\h \langle a,u\rangle_\h e
		=\langle a\star b,u\rangle_\h e+\al \langle Aa,b\rangle_\h e
		-\langle b,u\rangle_\h \langle a,u\rangle_\h e\\
		\ass(b,a,e)&=\langle b\star a,u\rangle_\h e+\al \langle Ab,a\rangle_\h e
		-\langle b,u\rangle_\h \langle a,u\rangle_\h e.
	\end{align*}
	Then $Q(a,b,e)=0$ if and only if
	$$\langle [a,b]_\star,u\rangle=\al\left(\langle Ab,a\rangle-\langle Aa,b\rangle\right).$$
	We obtain the second relation in the last equation in \eqref{curvature}
	
	\item $a,b\in\h$.
	\begin{align*}
		\ass(a,e,b)&=\langle a,u\rangle_\h e.b-a.(Eb+\langle b,v\rangle_\h e)\\
		&=\langle a,u\rangle_\h(Eb+\langle b,v\rangle_\h e)
		-a\star Eb-\langle Aa,Eb\rangle_\h e
		-\langle b,v\rangle_\h\langle a,u\rangle_\h e\\
		&=\langle a,u\rangle_\h Eb-a\star Eb
		+\langle EAa,b\rangle_\h e,\\
		\ass(e,a,b)&=(Ea+\langle a,v\rangle_\h e).b-e.(a\star b+\langle Aa,b\rangle_\h e)\\
		&=Ea\star b+\langle AEa,b\rangle_\h e
		+\langle a,v\rangle_\h(Eb+\langle b,v\rangle_\h e)
		-E(a\star b)\\
		&\quad-\langle a\star b,v\rangle_\h e-\al \langle Aa,b\rangle_\h e\\
		&=Ea\star b+\langle a,v\rangle_\h Eb-E(a\star b)
		+\bigl(\langle AEa,b\rangle_\h
		+\langle a,v\rangle_\h\langle b,v\rangle_\h\\
		&\qquad-\langle a\star b,v\rangle_\h
		-\al\langle Aa,b\rangle_\h\bigr)e.
	\end{align*}
	Hence $Q(a,e,b)=0$ if and only if
	\[ \begin{cases}
		E(a\star b)	-a\star Eb-Ea\star b=
		(\langle a,v\rangle_\h-\langle a,u\rangle_\h) Eb
		,\\
		[E,A]a-\langle a,v\rangle_\h v-a\star v+\al Aa=0.
	\end{cases} \]We obtain the third and fourth equation in \eqref{curvature}.
	
	\item $a\in\h$.
	\begin{align*}
		\ass(a,e,e)&=\langle a,u\rangle_\h \al e-\al\langle a,u\rangle_\h e=0,\\
		\ass(e,a,e)&=(Ea+\langle a,v\rangle_\h e).e-e.(\langle a,u\rangle_\h e)\\
		&=\langle Ea,u\rangle_\h e+\langle a,v\rangle_\h \al e-\langle a,u\rangle_\h\al e.
	\end{align*}
	So $Q(a,e,e)=0$ if and only if $E(u)=\al(v-u)$. We obtain the last relation in \eqref{curvature}.

	\item $a,b\in\h$.
	\begin{align*}
		\ass(a,f,b)&=(-\langle a,u\rangle_\h f-Aa).b-a.(Fb+\langle b,w\rangle_\h e) ,\\
		&=-\langle a,u\rangle_\h(Fb+\langle b,w\rangle_\h e)-Aa\star b-\langle A^2a,b\rangle_\h e
		-a\star Fb-\langle Aa,Fb\rangle_\h e-\langle b,w\rangle_\h\langle a,u\rangle_\h e\\
		&=-\langle a,u\rangle_\h Fb-Aa\star b-a\star Fb\\
		&\quad-\bigl(2\langle a,u\rangle_\h\langle b,w\rangle_\h
		+\langle A^2a,b\rangle_\h+\langle Aa,Fb\rangle_\h\bigr)e,\\
		\ass(f,a,b)&=(Fa+\langle a,w\rangle_\h e).b-f.(a\star b+\langle Aa,b\rangle_\h e)\\
		&=Fa\star b+\langle AFa,b\rangle_\h e
		+\langle a,w\rangle_\h(Eb+\langle b,v\rangle_\h e)\\
		&\quad-F(a\star b)-\langle a\star b,w\rangle_\h e
		-\langle Aa,b\rangle_\h\lambda e\\
		&=Fa\star b+\langle a,w\rangle_\h Eb-F(a\star b)
		+\bigl(\langle AFa,b\rangle_\h
		+\langle a,w\rangle_\h\langle b,v\rangle_\h\\
		&\qquad-\langle a\star b,w\rangle_\h
		-\lambda\langle Aa,b\rangle_\h\bigr)e.
	\end{align*}
	So $Q(a,f,b)=0$ if and only if
	\[ \begin{cases}
		F(a\star b)-a\star Fb-Fa\star b=\langle a,w\rangle_\h Eb+\langle a,u\rangle_\h Fb+Aa\star b,\\
		[F,A]a=A^2a-\la Aa+2\langle a,u\rangle_\h w+\langle a,w\rangle_\h v+a\star w.
	\end{cases} \]
	We get the fifth and sixth in \eqref{curvature}.
	\item $a\in\h$.
	\begin{align*}
		\ass(a,f,e)&=(-\langle a,u\rangle_\h f-Aa).e-\la a.e=
		-\la\langle a,u\rangle_\h e-\langle Aa,u\rangle e-\la \langle a,u\rangle_\h e\\
		\ass(f,a,e)&=(Fa+\langle a,w\rangle_\h e).e-\la\langle a,u\rangle_\h e\\
		&=\langle Fa,u\rangle_\h e-\la\langle a,u\rangle_\h e+\al \langle a,w\rangle_\h e.
	\end{align*}So $Q(a,f,e)=0$ if and only if

	$$A^*u+\la u-Fu+\al w=0.$$ We obtain the ninth equation in \eqref{curvature}.

	\item $a\in\h$.
	\begin{align*}
		\ass(e,f,a)&=(-v-\al f).a-e.(Fa+\langle a,w\rangle_\h e)\\
		&=-v\star a-\langle Av,a\rangle_\h e-\al(Fa+\langle a,w\rangle_\h e)
		-EFa-\langle Fa,v\rangle_\h e-\langle a,w\rangle_\h\al e\\
		&=-v\star a-\al Fa-EFa-\left(\langle Av,a\rangle_\h+2\al \langle a,w\rangle_\h
		+\langle Fa,v\rangle_\h\right)e ,\\
		\ass(f,e,a)&=\la (Ea+\langle a,v\rangle_\h e)-f.(Ea+\langle a,v\rangle_\h e)
		\\
		&=\la (Ea+\langle a,v\rangle_\h e)-FEa-\langle Ea,w\rangle e-\la \langle a,v\rangle_\h e\\
		&=\la Ea-FEa- \langle Ea,w\rangle_\h
		e.
	\end{align*}
	Then $Q(e,f,a)=0$ if and only if
	\[ \begin{cases}[F,E]=\la E+\al F+\Lu_v,\\
		Av-Fv+Ew+2\al w=0.\end{cases} \]We obtain the seventh and eighth equations in \eqref{curvature}.
	
	\item Finally,
	\begin{align*}
		\ass(e,f,e)&=(-v-\al f).e=-\langle u,v\rangle e-\al \la e\\
		\ass(f,e,e)&=\la \al e.
	\end{align*}
	Then $Q(e,f,e)=0$ if and only if $\langle u,v\rangle=-2\al\la.$ We obtain the first relation in the last equation in \eqref{curvature}.
	
\end{enumerate}

\subsection{Computation for Proposition  \ref{NENO}}\label{Subsection92}

According to Proposition \ref{Novikovpr},  $\G$ endowed with the product given by \eqref{Levi} is a Novikov algebra if and only if, for any $u,v\in\G$,
\[ \Lu_{u.v}=[\Lu_u,\Lu_v]=0. \]

Let's start by writing the conditions for which $[\Lu_u,\Lu_v]=0$.
\begin{enumerate}
	\item $u=e$ and $v=e$.
	\begin{align*}
		e.(f.e)-f.(e.e)&=0,\\
		e.(f.f)-f.(e.f)&=e.(-\la f-w)+f.v+\al f.f\\&=\la v+\al\la f-Ew-\langle w,v\rangle_\h e+Fv+\langle v,w\rangle_\h e-\al \la f-\al w\\
		&=\la v-Ew+Fv-\al w,\\
		e.(f.a)-f.(e.a)&=e.Fa+\al\langle a,w\rangle_\h e-f.Ea-\langle a,v\rangle_\h\la e\\
		&=EFa+\langle Fa,v\rangle e+\al\langle a,w\rangle e-FEa-\langle Ea,w\rangle e -\langle a,v\rangle\la e\\
		&=[E,F]a+\langle a, -Fv+Ew-\la v+\al w\rangle_\h e.
	\end{align*}So
	\[ [E,F]=0\esp \la v-Ew+Fv-\al w=0. \]

	\item $u=e$ and $v=a\in\h$.
	\begin{align*}
		e.(a.e)-a.(e.e)&=0,\\
		e.(a.f)-a.(e.f)&=-e.(\langle a,u\rangle_\h f+Aa)+a.v+\al a.f\\
		&=\langle a,u\rangle_\h v+\langle a,u\rangle_\h\al f-EAa-\langle Aa,v\rangle_\h e+a\star v+\langle Aa,v\rangle_\h e-\al \langle a,u\rangle_\h f-\al Aa \\
		&=a\star v-EAa-\al Aa+\langle a,u\rangle_\h v.\\
		e.(a.b)-a.(e.b)&=E(a\star b)+\langle a\star b,v\rangle_\h e+\al \langle Aa,b\rangle_\h e
		-a.Eb-\langle b,v\rangle_\h a.e\\
		&=E(a\star b)+\langle a\star b,v\rangle_\h e
		+\al\langle Aa,b\rangle_\h e-a\star Eb\\
		&\quad-\langle Aa,Eb\rangle_\h e
		-\langle b,v\rangle_\h\langle a,u\rangle_\h e\\
		&=E(a\star b)-a\star Eb+\langle -a\star v+\al Aa+EAa-\langle a,u\rangle_\h v,b\rangle_\h e.
	\end{align*}So
	\[ a\star v-EAa-\al Aa+\langle a,u\rangle_\h v=0\esp [E,\Lu_a]=0. \]
	
	\item $u=f$ and $v=a\in\h$.
	\begin{align*}
		f.(a.e)-a.(f.e)&=\la \langle a,u\rangle_\h e-\la \langle a,u\rangle_\h e=0,\\
		f.(a.f)-a.(f.f)&=-f.(\langle a,u\rangle_\h f+Aa)+\la a.f+a.w\\
		&=\langle a,u\rangle_\h(\la f+w)-FAa-\langle Aa,w\rangle_\h e
		-\la (\langle a,u\rangle f+Aa)+a\star w+\langle Aa,w\rangle e\\
		&=\langle a,u\rangle_\h w-FAa-\la Aa+a\star w,\\
		f.(a.b)-a.(f.b)&=f.(a\star b)+\la\langle Aa,b\rangle_\h e-a.(Fb+\langle b,w\rangle_\h e)\\
		&=F(a\star b)+\langle a\star b,w\rangle_\h e+\la\langle Aa,b\rangle_\h e
		-a\star Fb-\langle Aa,Fb\rangle_\h e-\langle b,w\rangle_\h\langle a,u\rangle_\h e\\
		&=F(a\star b)-a\star Fb
		+\langle \la Aa+FAa-\langle u,a\rangle w-a\star w,b\rangle_\h e.
	\end{align*}
	So
	\[ [F,\Lu_a]=0\esp \la Aa+FAa-\langle u,a\rangle_\h w-a\star w=0. \]
	\item $u=a\in\h$ and $v=b\in\h$.
	\begin{align*}
		a.(b.e)-b.(a.e)&=0,\\
		a.(b.f)-b.(a.f)&=-a.(\langle b,u\rangle_\h f+Ab)+b.(\langle a,u\rangle_\h f+Aa)\\
		&=\langle b,u\rangle_\h(\langle a,u\rangle_\h f+Aa)
		-a\star Ab-\langle Aa,Ab\rangle_\h e\\
		&\quad-\langle a,u\rangle_\h(\langle b,u\rangle_\h f+Ab)
		+b\star Aa+\langle Aa,Ab\rangle_\h e\\
		&=\langle b,u\rangle_\h Aa-\langle a,u\rangle_\h Ab+b\star Aa-a\star Ab,\\
		a.(b.c)-b.(a.c)&=a.(b\star c+\langle Ab,c\rangle_\h e)	-b.(a\star c+\langle Aa,c\rangle_\h e)\\
		&=a\star (b\star c)-b\star (a\star c)
		+\langle Aa,b\star c\rangle_\h e\\
		&\quad-\langle Ab,a\star c\rangle_\h e
		+\langle Ab,c\rangle_\h\langle a,u\rangle_\h e\\
		&\quad-\langle Aa,c\rangle_\h\langle b,u\rangle_\h e.
	\end{align*}
	
\end{enumerate}

Finally, we have $[\Lu_u,\Lu_v]=0$ if and only if the following conditions hold:

\[ \begin{cases} a\star (b\star c)-b\star (a\star c)=0,\\
	b\star Aa-a\star Ab=\langle a,u\rangle Ab-\langle b,u\rangle Aa,\\
	[E,F]=[E,\Lu_a]=[F,\Lu_a]=0,\\
	EA=\Ru_v+\langle\bullet, u\rangle v-\al A,\\
	FA=\Ru_w-\la A+\langle\bullet, u\rangle w,\;\\
	\;\la v-Ew+Fv-\al w=0.
\end{cases} \]

Let us now determine the conditions under which $\Lu_{u\star v}=0$. 
\begin{enumerate}
	\item $u=v=e$
	\begin{align*}
		(e.e).e&=\al^2 e,\\
		(e.e).f&=\al(-v-\al f),\\
		(e.e).a&=\al(Ea+\langle a,v\rangle_\h e).
	\end{align*}So $\al=0$.

	\item $u=e$, $v=f$
	\begin{align*}
		(e.f).e&=-v.e=-\langle u,v\rangle_\h e,\\
		(e.f).f&=-v.f=\langle u,v\rangle_\h f+Av,\\
		(e.f).a&=-v.a=-v\star a-\langle Av,a\rangle_\h e.
	\end{align*}So $\langle u,v\rangle_\h=0$,  $Av=0$ and $\Lu_v=0$.

	\item $u=f$, $v=e$
	\begin{align*}
		(f.e).e&=0,\\
		(f.e).f&=\la (-v),\\
		(f.e).a&=\la(Ea+\langle a,v\rangle_\h e).\\
	\end{align*}So $\la v=0$ and $\la E=0$.

	\item $u=e$, $v=a\in\h$
	\begin{align*}
		(e.a).e&=(Ea+\langle a,v\rangle_\h e).e=\langle Ea,u\rangle_\h e,\\
		(e.a).f&=(Ea+\langle a,v\rangle_\h e).f=-\langle Ea,u\rangle_\h f-AEa-\langle a,v\rangle_\h v,\\
		(e.a).b&=(Ea+\langle a,v\rangle_\h e).b=Ea\star b+\langle AEa,b\rangle_\h e
		+\langle a,v\rangle_\h(Eb+\langle v,b\rangle_\h e).
	\end{align*}So $Eu=0$, $AE=-\langle\bullet,v\rangle_\h v$ and $\Lu_{Ea}=-\langle a,v\rangle_\h E$.

	\item $u=a$, $v=e$
	\begin{align*}
		(a.e).e&=0\\
		(a.e).f&=-\langle u,a\rangle_\h v\\
		(a.e).b&=\langle u,a\rangle_\h (Eb+\langle b,v\rangle_\h e).
	\end{align*}So $|u| v=0$ and $|u| E=0$.

	\item $u=f$, $v=f$
	\begin{align*}
		(f.f).e&=(-\la f-w).e=-\la^2 e-\langle w,u\rangle_\h e\\
		(f.f).f&=(-\la f-w).f=-\la(-\la f-w)+\langle w,u\rangle_\h f+Aw\\
		(f.f).a&=(-\la f-w).a=-\la(Fa+\langle a,w\rangle_\h e)-w\star a-\langle Aw,a\rangle_\h e. 
	\end{align*}So $\langle u,w\rangle_\h=-\la^2$, $Aw=-\la w$, $\Lu_w=-\la F$.

	\item $u=f$, $v=a\in\h$
	\begin{align*}
		(f.a).e&=(Fa+\langle a,w\rangle_\h e).e=\langle Fa,u\rangle_\h e,\\
		(f.a).f&=(Fa+\langle a,w\rangle_\h e).f=-\langle Fa,u\rangle_\h f-AFa-\langle a,w\rangle_\h v\\
		(f.a).b&=(Fa+\langle a,w\rangle_\h e).b=Fa\star b+\langle AFa,b\rangle_\h e
		+\langle a,w\rangle_\h(Ea+\langle v,b\rangle_\h e).
	\end{align*}So $Fu=0$, $AF=-\langle \bullet,w\rangle_\h v$ and $\Lu_{Fa}=
	-\langle a,w\rangle_\h E.$

	\item $u=a\in\h$, $v=f$
	\begin{align*}
		(a.f).e&=-(\langle a,u\rangle_\h f+Aa).e=-\la \langle a,u\rangle_\h e-\langle Aa,u\rangle e,\\
		(a.f).f&=-(\langle a,u\rangle_\h f+Aa).f=\langle a,u\rangle_\h(\la f+w)
		+\langle Aa,u\rangle_\h f+A^2a\\
		(a.f).b&=-(\langle a,u\rangle_\h f+Aa).b=-\langle a,u\rangle_\h(Fb+\langle b,w\rangle_\h e)-Aa\star b-\langle A^2a,b\rangle_\h e.
	\end{align*}So $A^*u=-\la u$, $A^2a=-\langle a,u\rangle_\h w$ and $\Lu_{Aa}=-\langle a,u\rangle F$.

	\item $u=a\in\h$, $v=b\in\h$
	\begin{align*}
		(a.b).e&=(a\star b+\langle Aa,b\rangle_\h e).e=\langle a\star b,u\rangle_\h e,\\
		(a.b).f&=(a\star b+\langle Aa,b\rangle_\h e).f=-\langle a\star b,u\rangle_\h f-A(a\star b)-\langle Aa,b\rangle_\h v\\
		(a.b).c&=(a\star b+\langle Aa,b\rangle_\h e).c=(a\star b)\star c+\langle A(a\star b),c\rangle_\h e+\langle Aa,b\rangle_\h(Ec+\langle c,v\rangle_\h e).
	\end{align*}So $\Ru_u=0$, $A(a\star b)=-\langle Aa,b\rangle_\h v$
	and $\Lu_{a\star b}=-\langle Aa,b\rangle_\h E.$

\end{enumerate}
Combining all the conditions derived above, we obtain \eqref{NE=0}.

\subsection{Verification of the Novikov property in Proposition \ref{u}}\label{subsection93}

We have seen in the proof of Proposition \ref{u} that $\al=0$, $E=0$, $v=0$,  $\h=\df\oplus\A\oplus\R u$ and \[ \begin{cases}
	A_{|\df}=0,\; Aa=\Ru_h(a)=\rho(a)h,\; a\in\A,\; Au=m+n-\la u,\\
	F_{|\df}=-\frac1{|u|^2}(\rho(n)-\la\rho(u)),\; F_{|\A\oplus\R u}=0,\\
	m=\rho(u)h+|u|^2h,\; w=w_1+w_2+w_3,\\
	w_1=\la h-\frac1{|u|^2}\rho(n)(h)+\frac{\la}{|u|^2}\rho(u)h,\; w_2=\frac{\la n}{|u|^2},\; w_3=-\frac{\la^2}{|u|^2}u,
\end{cases} \]where $\rho:\af\too\so(\df)$ is a representation of an abelian Lie algebra satisfying $\bigcap_{a\in\af}\ker\rho(a)=\{0\}$ and $\star$ is given by \eqref{Levi-Milnor}. Let us show that in this case the extension $\G=\R e\oplus\h\oplus\R f$ endowed with the Levi-Civita product given by \eqref{Levi} is a Novikov algebra, i.e, the data satisfy \eqref{NE=0}.

Note first that \eqref{NE=0} reduces to
\begin{equation}\label{bis} \begin{cases} a\star (b\star c)-b\star (a\star c)=0,\; \Lu_{a\star b}=0,\\
		b\star Aa-a\star Ab=\langle a,u\rangle Ab-\langle b,u\rangle_\h Aa,\\
		A(a\star b)=0,\;A^2a=-\langle a,u\rangle_\h w,\\
		[F,\Lu_a]=0,\;
		FA=\Ru_w-\la A+\langle\bullet, u\rangle_\h w,\;
		AF=0,\\
		\Lu_{Fa}=0,\; \Lu_{Aa}=-\langle a,u\rangle_\h F,\\
		Fu=0,\\
		\; A^*u=-\la u,\;\Ru_u=0,\\
		Aw=-\la w,\langle u,w\rangle=-\la^2,\;\Lu_w=-\la F.
\end{cases} \end{equation}

\medskip
\noindent\textbf{Verification of the identities.}

\emph{\underline{$(1)$ $a\star (b\star c)-b\star (a\star c)=0,\; \Lu_{a\star b}=0.$} }\\

Since $\h$ is a flat Euclidean Lie algebra, by virtue of Theorem \ref{milnor}, it satisfies this two identities. \\

\emph{\underline{ $(2)$ $A(a\star b)=0$.}}\\

We have $\df=\h\star\h$ and then $A(a\star b)=0$.\\

\emph{\underline{$(3)$ $b\star Aa-a\star Ab=\langle a,u\rangle Ab-\langle b,u\rangle_\h Aa,$}}\\

This relation holds obviously if $a\in\df$ or $b\in\df$. 

If $a,b\in\A$ then $\langle u,a\rangle_\h=\langle u,b\rangle_\h=0$ and
\[ b\star Aa-a\star Ab=b\star(a\star h)-a\star(b\star h)=0. \]
So the relation holds in this case.

Let $a\in\A$ and $b=u$. So
\[ \begin{cases}
	u\star Aa-a\star Au=u\star (a\star h)-a\star m=\rho(u)\rho(a)(h)-\rho(a)\rho(u)h-|u|^2\rho(a)h=-|u|^2\rho(a)h\\
	\langle a,u\rangle Au-\langle u,u\rangle_\h Aa=-|u|^2\rho(a)h
\end{cases} \]and the relation holds.\\

\emph{\underline{$(4)$ $A^2a=-\langle a,u\rangle_\h w.$}}\\

For $a\in\df$ or $a\in\A$ the relation holds obviously. For $a=u$,
\begin{align*} A^2u&=A(m+n-\la u)=\rho(n)h-\la(m+n-\la u)\\&=\rho(n)h-\la n-\la(\rho(u)h+|u|^2h)+\la^2u,\\
	-\langle u,u\rangle w
	&=-|u|^2(\la h-\frac1{|u|^2}\rho(n)(h)+\frac{\la}{|u|^2}\rho(u)h+\frac{\la n}{|u|^2}-\frac{\la^2}{|u|^2}u)\\
	&=-\la|u|^2 h+\rho(n)(h)-\la\rho(u)h-{\la n}+{\la^2}u
\end{align*}
and the relation holds.\\

\emph{\underline{$(5)$ $[F,\Lu_a]=0$.} }\\

For $a\in\df$, $\Lu_a=0$ and for $a\in\A\oplus\R u$, $(\Lu_{a})_{|\df}=\rho(a)$ and $(\Lu_{a})_{|\A\oplus\R u}=0$  so $[F,\Lu_a]=0$ for any $a\in\h$.\\

\emph{\underline{$(6)$ $FA=\Ru_w-\la A+\langle\bullet, u\rangle_\h w$.} }\\

For $a\in\df$ 
\begin{align*}
	FAa=0
	=a\star w-\la Aa+\langle a, u\rangle_\h w.
\end{align*} 
For $a\in\A$, we have
\begin{align*}
	FAa&=-\frac1{|u|^2}(\rho(n)\rho(a)h-\la\rho(u)\rho(a)h),\\
	\rho(a)w-\la Aa&=\rho(a)w_1-\la\rho(a)h\\
	&=\rho(a)\left(\la h-\frac1{|u|^2}\rho(n)(h)+\frac{\la}{|u|^2}\rho(u)h\right)
	-\la\rho(a)h\\
	&=-\frac1{|u|^2}(\rho(a)\rho(n)h-\la\rho(u)\rho(a)h)
\end{align*}and the relation holds
since $\rho$ is a representation of abelian Lie algebra.

For $u$, we have
\begin{align*}FAu&=F(m+n-\la u)\\&= -\frac1{|u|^2}(\rho(n)-\la\rho(u))(\rho(u)h+|u|^2h),\\
	&=-\frac1{|u|^2}\rho(n)\rho(u)h-\rho(n)h+\frac{\la}{|u|^2}\rho(u)^2h+\la\rho(u)h,\\
\end{align*}
\begin{align*}
	u\star w-\la (m+n-\la u)+|u|^2w&=
	\rho(u)(\la h-\frac1{|u|^2}\rho(n)(h)+\frac{\la}{|u|^2}\rho(u)h)
	-\la(\rho(u)h+|u|^2h)-\la n+\la^2u\\&+|u|^2\left(\la h-\frac1{|u|^2}\rho(n)(h)+\frac{\la}{|u|^2}\rho(u)h+\frac{\la n}{|u|^2} -\frac{\la^2}{|u|^2}u\right)\\
	&=\la\rho(u)h-\frac1{|u|^2}\rho(u)\rho(n)(h)+\frac{\la}{|u|^2}\rho(u)^2h
	-\la \rho(u)h-\la|u|^2h-\la n+\la^2u\\
	&+\la|u|^2h-\rho(n)(h)+\la\rho(u)h+\la n-\la^2u\\
	&=-\frac1{|u|^2}\rho(n)\rho(u)h-\rho(n)h+\frac{\la}{|u|^2}\rho(u)^2h+\la\rho(u)h
\end{align*}	and the relation holds.

\emph{\underline{$(7)$ $AF=0,\; \Lu_{Fa}=0,\; \Lu_{Aa}=-\langle a,u\rangle F $.}}\\

Since $\im F\subset \df$, we have $AF=0$ and $\Lu_{Fa}=0$. Let us check the relation $\Lu_{Aa}=-\langle a,u\rangle F$. The relation hold obviously if $a\in\df\oplus\A$. For $a=u$, we have
\begin{align*}
	(\Lu_{Au})_{|\A\oplus\R u}&=0=-|u|^2F_{|\A\oplus\R u},\\
	(\Lu_{Au})_{|\df}&=\rho(n)-\la\rho(u),\\
	-|u|^2F_{|\df}&=\rho(n)-\la\rho(u)
\end{align*}	and the relation holds.

\emph{\underline{$(8)$ $Fu=0,\;\langle w,u\rangle_\h=-\la^2,\; \Ru_u=0,\; A^*u=-\la u,\;\Lu_w=-\la F,\;Aw=-\la w$. } }\\

We have obviously $Fu=0$, $\langle w,u\rangle_\h=-\la^2$,  $\Ru_u=0$ and $A^*u=-\la u$. Moreover,
\begin{align*}
	\Lu_{w}&=\frac{1}{|u|^2}\left(\la\Lu_n-\la^2\Lu_u\right)=-\la F,\\
	Aw&=\frac{\la }{|u|^2}An-\frac{\la^2}{|u|^2}Au\\
	&=\frac{\la }{|u|^2}\rho(n)h-\frac{\la^2}{|u|^2}(\rho(u)h+|u|^2h)-\frac{\la^2}{|u|^2}n
	+\frac{\la^3}{|u|^2}u\\
	&=-\la w.
\end{align*}
This completes the verification.

\subsection{Verification of the Novikov property in Proposition \ref{u0}}\label{subsection94}

Let $(\h,\star,\prs)$ be a Euclidean algebra together with the data $(E,F,A,u,v,w,\al,\la)$ as in Proposition \ref{u0}.
Let us check when the extension $\G$ is a Novikov algebra. We have $\al=\la=0$, $u=v=0$, $E=0$. Moreover, $(\h,\star,\prs)$ and $(F,A,w)$ are given in Proposition \ref{u0}.
In this case, \eqref{NE=0} reduces to
\begin{equation*} \begin{cases} a\star (b\star c)-b\star (a\star c)=0,\; \Lu_{a\star b}=0,\\
		b\star Aa-a\star Ab=0,\;
		A(a\star b)=0,\;A^2=0,\\
		[F,\Lu_a]=0,\;
		FA=\Ru_w,\;
		AF=0,\\
		\Lu_{Fa}=0,\; \Lu_{Aa}=0,\\
		Aw=0,\Lu_w=0,\\
\end{cases} \end{equation*}
Since $\h$ is a flat Euclidean Lie algebra, by virtue of Theorem \ref{milnor}, the first relations hold.

Since $\df=\h\star\h$ the relation $A(a\star b)=0$ is equivalent to $A_1=0$ and $A^2=0$ is equivalent to $A_2^2=0$. 

The relations $b\star Aa-a\star Ab=0$ and $[F,\Lu_a]=0$ hold obviously.
The relation $\Lu_{Fa}=\Lu_{Aa}=0$ is equivalent to $\im F_2\subset\ker\rho$ and 
$\im A_2\subset\ker\rho$. We have $\Lu_{w}=\Lu_{w_2}=0$ if and only if $\rho(w_2)=0$. Moreover, $A(w)=\rho(w_2)h+A_2w_2=0$ if and only if $\rho(w_2)h=0$ and $A(w)=0$. We have
\[ FA=\begin{pmatrix}
	0&F_1\circ\Ru_h\\
	0&F_2A_2
\end{pmatrix},\;AF=\begin{pmatrix}
	0&\Ru_h\circ F_2\\
	0&A_2F_2
\end{pmatrix}\esp \Ru_{w}=\begin{pmatrix}
	0&\Ru_{F_1(h)}\\
	0&0
\end{pmatrix}. \]So $AF=0$ and $FA=\Ru_w$ are equivalent to
\[ A_2F_2=F_2A_2=0,\; \rho(F_2(a))h=0\esp F_1(a\star h)=a\star F_1(h),\; a\in\af. \]
Finally, we get that $\G$ is a Novikov algebra if and only if
\[
\begin{gathered}
[F_1,\rho(a)]=0,\quad A_1=0,\quad A_2^2=0,\\
\operatorname{Im}F_2\subseteq\ker\rho,\quad
\operatorname{Im}A_2\subseteq\ker\rho,\\
A_2F_2=F_2A_2=0,\quad \rho(w_2)h=0,\quad A_2w_2=0.
\end{gathered}
\]

\subsection{Verification of the Novikov property in Proposition \ref{E}}\label{subsection96}

Let $(\h,\star,\prs)$ be a Euclidean algebra together with the data $(E,F,A,u,v,w,\al,\la)$ as in Proposition \ref{E}.
Let us check that the extension $\G$ is a Novikov algebra. We have $\al=\la=0$ and $u=v=0$. Moreover, $(\h,\star,\prs)$ and $(E,F,A,w)$ are given in Proposition \ref{E}.
In this case, \eqref{NE=0} reduces to
\begin{equation}\label{bisE} \begin{cases} a\star (b\star c)-b\star (a\star c)=0,\; \Lu_{a\star b}=-\langle Aa,b\rangle_\h E,\\
		b\star Aa-a\star Ab=0,\;
		A(a\star b)=0,\;A^2=0,\\
		[E,F]=[E,\Lu_a]=[F,\Lu_a]=0,\\
		EA=0,\;
		FA=\Ru_w,\;\\
		AE=0,\; AF=0,\\
		\Lu_{Ea}=0,\Lu_{Fa}=-\langle a,w\rangle_\h E,\; \Lu_{Aa}=0,\\
		Ew=0,\;
		Aw=0,\;\Lu_w=0.
\end{cases} \end{equation}

Recall from Proposition~\ref{E} that
\[
\h=\df\oplus\af\oplus\mathfrak b,
\qquad h\in\df,\quad g\in\af,
\]
with mutually orthogonal summands, and that, for
$a\in\af$, $d\in\df$ and $b\in\mathfrak b$,
\[
\begin{aligned}
\Lu_a&=\rho_1(a)\oplus 0\oplus\rho_2(a),
&\Lu_d&=\langle d,h\rangle_\h E,
&\Lu_b&=0,\\
A(a)&=\rho_1(a)h,
&A(d)&=A(b)=0,
&w&=F_1(h)+F_2(g).
\end{aligned}
\]
Here $E=0\oplus0\oplus E_1$, $F=F_1\oplus F_2\oplus F_3$,
$E_1\in\so(\mathfrak b)$ is invertible, and the representation
$\rho=\rho_1\oplus\rho_2$ satisfies
\[
\begin{gathered}
[F_1,\rho_1(a)]=\rho_1(F_2a)=0,\qquad
[F_3,\rho_2(a)]=[F_3,E_1]=[E_1,\rho_2(a)]=0,\\
\rho_2(F_2a)=\langle F_2a,g\rangle_\h E_1,
\qquad \bigcap_{a\in\af}\ker\rho_1(a)=\{0\}.
\end{gathered}
\]	

Put $\mathfrak m=\df\oplus\mathfrak b$ and
$
\rho=\rho_1\oplus\rho_2:
\af\longrightarrow\so(\mathfrak m).
$
Thus $\h=\mathfrak m\oplus\af$
orthogonally. We extend $\rho(a)$
by zero on $\af$, so that it acts
on $\h$.

For $a,b\in\af$ and $x,y\in\mathfrak m$,
the product $\star$ is given by
\begin{equation}\label{eq:Novikov-U3-product}
	a\star b=0,\;
	a\star x=\rho(a)x,\; x\star a=0,\esp
	x\star y=\langle x,h\rangle_\h E(y).
\end{equation}
Moreover,
\begin{equation}\label{eq:Novikov-U3-data}
	\begin{cases}
		A(a)=\rho(a)h,\qquad A(x)=0,\\
		\Lu_a=\rho(a),\qquad
		\Lu_x=\langle x,h\rangle_\h E,\\
		w=F(h)+F_2(g).
	\end{cases}
\end{equation}
Recall also that
\[
Eh=0,\qquad
[E,F]=0,\qquad
[E,\rho(a)]=[F,\rho(a)]=0,
\]
and
\begin{equation}\label{eq:Novikov-U3-rho}
	\rho(F_2a)
	=\langle F_2(a),g\rangle_\h E,
	\qquad a\in\af.
\end{equation}
Therefore,
$
\rho(F_2g)
=-\langle g,F_2g\rangle_\h E=0,
$
because $F_2$ is skew-symmetric.
In particular,
\begin{equation}\label{eq:Novikov-U3-kernel}
	\rho(w_{\af})=0,\qquad
	\rho(F_2a)h=0.
\end{equation}

We now verify the identities
in \eqref{bisE}.

\medskip
\noindent
\textbf{(1) $[\Lu_p,\Lu_q]=0$.}

For $a,b\in\af$, the relation follows
from the fact that $\rho$ is a
representation of an abelian Lie algebra:
\[
[\Lu_a,\Lu_b]
=[\rho(a),\rho(b)]=0.
\]
For $a\in\af$ and $x,y\in\mathfrak m$,
we have
\[
[\Lu_a,\Lu_x]
=\langle x,h\rangle_\h[\rho(a),E]=0
\esp
[\Lu_x,\Lu_y]=0.
\]
Thus all left multiplication
operators commute.

\medskip
\noindent
\textbf{(2)
	$\Lu_{p\star q}
	=-\langle Ap,q\rangle_\h E$.}

For $a,b\in\af$, both sides vanish.
For $a\in\af$ and $x\in\mathfrak m$,
we obtain
\[
\begin{aligned}
	\Lu_{a\star x}
	=\Lu_{\rho(a)x}
	=\langle\rho(a)x,h\rangle_\h E
	=-\langle\rho(a)h,x\rangle_\h E
	=-\langle Aa,x\rangle_\h E.
\end{aligned}
\]
On the other hand,
$
\Lu_{x\star a}=0$ and
$
Ax=0.
$
Finally, for $x,y\in\mathfrak m$,
\[
\begin{aligned}
	\Lu_{x\star y}
	&=\langle x,h\rangle_\h\Lu_{Ey}
	=\langle x,h\rangle_\h
	\langle Ey,h\rangle_\h E
	=0,
\end{aligned}
\]
since $Eh=0$. This agrees with
$Ax=0$.

\medskip
\noindent
\textbf{(3)
	$q\star Ap-p\star Aq=0$.}

For $a,b\in\af$, we have
\[
\begin{aligned}
	b\star Aa-a\star Ab
	&=\rho(b)\rho(a)h
	-\rho(a)\rho(b)h
	=0.
\end{aligned}
\]
For $x\in\mathfrak m$ and $a\in\af$,
\[
x\star Aa
=\langle x,h\rangle_\h
E(\rho(a)h)=0,
\]
whereas $Ax=0$.
The remaining cases are immediate.

\medskip
\noindent
\textbf{(4) $A(p\star q)=0$
	and $A^2=0$.}

By \eqref{eq:Novikov-U3-product},
every product $p\star q$ belongs
to $\mathfrak m$.
Since $A$ vanishes on $\mathfrak m$,
we obtain
\[
A(p\star q)=0.
\]

Moreover,
\[
\operatorname{Im}A\subseteq\df
\subseteq\mathfrak m
\subseteq\ker A.
\]
Therefore $A^2=0$.

\medskip
\noindent
\textbf{(5)
	$[E,F]=[E,\Lu_p]=[F,\Lu_p]=0$.}

The equality $[E,F]=0$ is part
of the conditions of
Proposition~\ref{E}.

For $a\in\af$, we have
\[
[E,\Lu_a]=[E,\rho(a)]=0,
\qquad
[F,\Lu_a]=[F,\rho(a)]=0.
\]
For $x\in\mathfrak m$, since
$\Lu_x=\langle x,h\rangle_\h E$,
we obtain
\[
[E,\Lu_x]=[F,\Lu_x]=0.
\]

\medskip
\noindent
\textbf{(6) $EA=0$ and
	$FA=\Ru_w$.}

Since
\[
\operatorname{Im}A\subseteq\df
\subseteq\ker E,
\]
we have $EA=0$.

Let $a\in\af$. By
\eqref{eq:Novikov-U3-data},
\[
FA(a)=F\rho(a)h
=\rho(a)F(h).
\]
On the other hand,
\[
\begin{aligned}
	\Ru_w(a)
	&=a\star w
	=a\star(F(h)+w_{\af})
	=\rho(a)F(h).
\end{aligned}
\]
Thus $FA(a)=\Ru_w(a)$.

Now let $x\in\mathfrak m$.
Since $Ax=0$, we have $FA(x)=0$.
Moreover,
\[
\begin{aligned}
	\Ru_w(x)
	&=x\star(F(h)+w_{\af})
	=\langle x,h\rangle_\h E(F(h))
	=0,
\end{aligned}
\]
because $F(h)\in\df\subseteq\ker E$.
Hence $FA=\Ru_w$.

\medskip
\noindent
\textbf{(7) $AE=0$ and $AF=0$.}

Since $\operatorname{Im}E
\subseteq\mathfrak b
\subseteq\ker A$, we have
$
AE=0.
$
For $a\in\af$, using
\eqref{eq:Novikov-U3-kernel},
\[
AF(a)=A(F_2a)
=\rho(F_2a)h=0.
\]
For $x\in\mathfrak m$, we have
$F(x)\in\mathfrak m$, and hence
$
AF(x)=0.
$
Therefore $AF=0$.

\medskip
\noindent
\textbf{(8)
	$\Lu_{Ep}=0$,
	$\Lu_{Fp}=-\langle p,w\rangle_\h E$
	and $\Lu_{Ap}=0$.}

Since $Ep\in\mathfrak b$ and
$\Lu_x=0$ for every
$x\in\mathfrak b$, we obtain
$
\Lu_{Ep}=0.
$
Let $a\in\af$. Then
\[
\begin{aligned}
	\Lu_{Fa}
	&=\Lu_{F_2a}
	=\rho(F_2a)
	=-\langle a,w_{\af}\rangle_\h E
	=-\langle a,w\rangle_\h E.
\end{aligned}
\]
For $x\in\mathfrak m$, we have
\[
\begin{aligned}
	\Lu_{Fx}
	&=\langle Fx,h\rangle_\h E
	=-\langle x,F(h)\rangle_\h E
	=-\langle x,w\rangle_\h E.
\end{aligned}
\]
Thus
$
\Lu_{Fp}=-\langle p,w\rangle_\h E
$
for every $p\in\h$.

Finally, since $A$ vanishes on
$\mathfrak m$, it is enough to
consider $a\in\af$. We have
\[
\begin{aligned}
	\Lu_{Aa}
	&=\Lu_{\rho(a)h}
	=\langle\rho(a)h,h\rangle_\h E
	=0,
\end{aligned}
\]
because $\rho(a)$ is skew-symmetric.
Therefore $\Lu_{Ap}=0$.

\medskip
\noindent
\textbf{(9)
	$Ew=0$, $Aw=0$, and $\Lu_w=0$.}

Since
$
w=F_1(h)+F_2(g)
\in\df\oplus\af=\ker E,
$
we have $Ew=0$.

Moreover, using
\eqref{eq:Novikov-U3-kernel},
\[
\begin{aligned}
	Aw
	=A(F_1(h))+A(F_2(g))
	=\rho(F_2(g))h=0.
\end{aligned}
\]

Finally,
\[
\begin{aligned}
	\Lu_w
	&=\Lu_{F_1(h)}+\Lu_{F_2(g)}
	=\langle F_1(h),h\rangle_\h E
	+\rho(F_2(g))
	=0,
\end{aligned}
\]
since $F$ is skew-symmetric
and $\rho(F_2(g))=0$.

All the identities in
\eqref{bisE} are therefore satisfied.
Consequently, the Levi-Civita
product of the flat Lorentzian
Lie algebra described in
Proposition~\ref{E} is a
Novikov product.

{\bf Acknowledgments.} The author would like to express sincere gratitude to the anonymous referee for their careful reading of the manuscript and their insightful comments and valuable suggestions, which have played a crucial role in substantially improving both the results and their presentation. In particular, some of the referee's suggestions led to remarkable simplifications and significantly enhanced the clarity and completeness of the paper.

\end{document}